\documentclass[final]{siamltex}

\usepackage[top=2.5cm,bottom=2.5cm,right=2.5cm,left=2.5cm]{geometry}

\def\g{\gamma}
\def\t{\theta}
\def\o{\omega}

\def\clS{\mathcal{S}}
\def\clA{\mathcal{A}}

\newcommand{\tsum}{\textstyle{\sum}}

\def\la{\langle}
\def\ra{\rangle}

\def\eqnok#1{(\ref{#1})}
\usepackage{epsfig}
\usepackage{calc}
\usepackage{amstext}
\usepackage{amsmath}
\usepackage{multicol}
\usepackage{amsfonts}
\usepackage{amssymb}
\usepackage{mathrsfs}
\usepackage{bm}
\usepackage{bbm}
\usepackage{xcolor}
\usepackage{comment}
\usepackage{footnote}

\newcommand{\AR}[2]{\left[\begin{array}{#1}#2\end{array}\right]}

\newcommand{\T}{\mathrm{T}}

\newcommand{\res}{\mathrm{res}}

\newcommand{\bbr}{\Bbb{R}}
\newcommand{\beq}{\begin{equation}}
\newcommand{\eeq}{\end{equation}}
\newcommand{\beqa}{\begin{eqnarray}}
\newcommand{\eeqa}{\end{eqnarray}}
\newcommand{\beqas}{\begin{eqnarray*}}
\newcommand{\eeqas}{\end{eqnarray*}}

\usepackage{algorithm,algorithmic}
\def\argmin{{\rm argmin}}
\def\argmax{{\rm argmax}}

\def\bbe{{\mathbb{E}}}
\def\vgap{\vspace*{.1in}}
\newcommand{\nn}{\nonumber}

\title{Simple and optimal methods for stochastic variational inequalities, II: Markovian noise and policy evaluation in reinforcement learning
	\thanks{
		This research was partially supported by the ARO grant W911NF-18-1-0223 and ONR grant N00014-20-1-2089.
		Coauthors of this paper are listed according to the alphabetic order.
}}
\author{
	Georgios Kotsalis
	\thanks{H. Milton Stewart School of Industrial and Systems Engineering, Georgia Institute of Technology, Atlanta, GA, 30332 .
		(email: {\tt gkotsalis3@gatech.edu}).}
	\and
	Guanghui Lan
	\thanks{H. Milton Stewart School of Industrial and Systems Engineering, Georgia Institute of Technology, Atlanta, GA, 30332 .
		(email: {\tt george.lan@isye.gatech.edu}).}
	\and
	Tianjiao Li
	\thanks{H. Milton Stewart School of Industrial and Systems Engineering, Georgia Institute of Technology, Atlanta, GA, 30332 .
		(email: {\tt tli432@gatech.edu}).}
}

\date{\today}

\begin{document} 

\maketitle

\begin{abstract}
	The focus of this paper is on stochastic variational inequalities (VI) under Markovian noise. 
	A prominent application of our algorithmic developments is the stochastic policy evaluation problem in reinforcement learning. Prior investigations in the literature focused on temporal difference (TD) learning by employing  nonsmooth finite time analysis motivated by stochastic subgradient descent leading to certain limitations.  These  limitations encompass the requirement of analyzing a modified TD algorithm that involves projection to an a-priori defined Euclidean ball, achieving a non-optimal convergence rate and no clear way of deriving the beneficial effects of parallel implementation. Our approach remedies these shortcomings 
	in the broader context of stochastic VIs and in particular when it comes to stochastic policy evaluation. We developed a variety of simple TD learning  type algorithms motivated by its original version that maintain its simplicity, while offering distinct advantages from a non-asymptotic analysis point of view. We first provide an improved 
	analysis of the standard TD algorithm that can benefit from parallel implementation. 
	Then we present versions of a conditional TD algorithm (CTD), that involves periodic updates of the stochastic iterates, which reduce the bias and therefore exhibit improved iteration complexity. This brings us to the fast TD (FTD) algorithm which combines elements of CTD and the stochastic operator extrapolation method	of the companion paper.  
	For a novel index resetting stepsize policy FTD exhibits the best known convergence rate. We also devised a robust version of the algorithm  
	that is particularly suitable for discounting factors close to 1.

\end{abstract}

\section{Introduction}   
In this paper we continue our algorithmic investigations in  variational inequality (VI) problems that originated in the companion paper \cite{GGT_20a}. We consider stochastic VI problems with inexact information, where biased estimators of the underlying operator are obtained via a stochastic oracle. For the purposes of completeness we include the formulation of the generalized monotone variational inequality (GMVI) problem:
\begin{equation}
\label{VIP}
\text{Find}~ x^* \in X: ~~~~ \langle F(x^*) , x - x^* \rangle \geq 0, ~~~ \forall x \in X,
\end{equation}
where $X \subseteq \mathbb{R}^n$ is a nonempty closed convex set. The operator  $F : X \rightarrow \mathbb{R}^n$ is a $L$-Lipschitz continuous map, i.e., for some $L > 0 $,
\begin{equation}
\label{Lipschitz}
\| F(x_1) - F(x_2) \|_*   \leq L \| x_1 - x_2 \|,~~~~ \forall x_1, x_2 \in X,
\end{equation}
and satisfies a generalized monotonicity condition 
\begin{equation} \label{G_monotone0}
\langle F(x), x - x^* \rangle \ge \mu \|x - x^*\|^2, \ \ \forall x \in X\textcolor{black}{,}
\end{equation}
for some $\mu \ge 0$. Throughout the paper we assume the existence of a solution $x^*$ to the problem \eqnok{VIP}-\eqnok{G_monotone0}, while $X^*$  denotes the set of all solutions.  A special case of interest is the generalized strongly monotone variational inequality (GSMVI) problem, consisting of \eqnok{VIP}-\eqnok{G_monotone0}, where the latter is satisfied with $\mu > 0 $. 

The stochastic oracle generates at the point $x \in X$ the  operator value
 $\tilde{F}(x, \xi)$, where  $\xi $ is a random vector, whose probability distribution is supported on a set $\Xi \subset \mathbb{R}^d$. In \cite{GGT_20a} we analyzed the setting, where one can obtain at each time instant an unbiased estimator of the operator $F$ by 
drawing independent and identically distributed (i.i.d.) samples $ \xi_1, \xi_2, \hdots, $ of the random vector $\xi$ according to a distribution  $\Pi $, 
\begin{equation}
\label{unbiased_estimate}
F(x) = \mathbb{E}[\tilde{F}(x, \xi)] = \int_{\xi \in \Xi} \tilde{F}(x, \xi) d \Pi(\xi), ~~~ \forall x \in X.
\end{equation}
In this paper we relax the i.i.d. assumption and consider $ \{\xi_{t} : t \in \mathbb{Z}_+\}$ to be a Markov process, on some underlying probablity space $(\Omega, \mathcal{F}, \mathbb{P})$ whose state space is $\Xi$. We further assume that the probability distribution $\Pi$ appearing in \eqref{unbiased_estimate} corresponds to the unique invariant distribution of $ \{\xi_{t} : t \in \mathbb{Z}_+\}$, i.e.,
$$
\Pi(A) = \int_{\xi \in \Xi} P(\xi, A)  d \Pi(\xi), ~~~~ \forall A \in \mathcal{B}(\Xi),
$$
where the transition kernel of $ \{\xi_{t} : t \in \mathbb{Z}_+\}$ is denoted by $P$.
This modification results in biased estimators of the operator $F$ posing additional challenges to the algorithmic analysis.
The scope of our investigations is naturally aligned with a wealth of stochastic approximation problems in machine learning, statistics, optimization and control, where stochastic data exhibit serial correlation,  see e.g.\textcolor{black},  \cite{benv90}, \cite{KusYin03}, \cite{Spall03}, \cite{meyn_tweedie_glynn_2009} and the references therin. In terms of mirror descent, the  case of Markovian noise has been explored in \cite{duchi2012ergodic}. 

A  problem of interest  that falls within our framework is the analysis of stochastic strongly monotone VIs for policy evaluation.
For a given linear operator $T: \bbr^n \to \bbr^n$,
the basic policy evaluation problem can be formulated as a fixed-point equation:
\beq
\label{policy_evaluation_fixed_point}
\text{Find}~ x^* \in X: ~~~~ x^* = T(x^*),
\eeq
which is a special case of \eqref{VIP} with $X = \bbr^n$ and $F(x) = x - T(x)$.  
The point $x^* \in \mathbb{R}^n$ encodes the value function corresponding to a specific policy in a Markov decision process (MDP)~\cite{puterman2014markov,bertsekas1996stochastic}. The connection between VI problems and approximate dynamic programming has been brought to the foreground in 
\cite{bertsekas_2009}. 

{ The computation of the fixed point $x^*$
	according to \eqref{policy_evaluation_fixed_point} requires knowledge of the transition kernel of the Markov process corresponding to the policy under consideration and the related reward structure. In model-free settings  where this information is not available or hard to construct, one resorts to 
	reinforcement learning (RL) algorithms \cite{bertsekas2018dynamic,Sutt18} in order to estimate this value function by relying on a stream of state-reward pairs generated by the underlying Markov process. In the context of RL, policy evaluation is  an important step of policy iteration algorithms that alternate between computing the value function for a specific policy and performing a policy improvement step until a (near) optimal one is determined \cite{konda2000actor,lagoudakis2003least}.  Hence iteration complexity results for obtaining accurate estimates of the value function  in policy evaluation is a problem of interest both from theoretical as well as practical point of view. 
}


One of the most popular stochastic iterative algorithm 
used to estimate this value function is temporal difference (TD) learning introduced in \cite{sutton1988learning}. 
{We focus on the application of our algorithms in regards to the policy evaluation problem in the context of finite state MDPs with linear function approximation. Asymptotic convergence of TD with linear function approximation
	has been established in \cite{tsitsiklis_vanroy_97}. While practical situations involve the observation of a single Markovian data stream, the first attempts towards finite time convergence results focussed to what is termed the i.i.d. observation model \cite{sutton2009fast,lakshminarayanan2018linear}. In this model one can receive unbiased estimators of the underlying  operator and the analysis mirrors features of stochastic gradient descent. The challenge of Markovian noise stems from the presence of dependent data that lead to biased samples. 
	Finite time analysis of TD learning under Markovian noise has been the recent subject of \cite{russo_18}, where the authors employ nonsmooth analysis to a  variant of the traditional TD algorithm that requires a projection step to an a-priori specified Euclidean ball. Another consequence of the nonsmooth approach in \cite{russo_18} is that there is no obvious way of benefiting from the variance reduction effect of distributed/parallel computing. Our smooth analysis offers remedies to these points, while achieving the best known so far convergence rate. Additionally we provide a robust analysis for the important case when the discount factor $\beta $ is close to 1, which to our knowledge has not been addressed in the prior literature. 
}

We now discuss our contributions in the broader context of stochastic VI problems under Markovian noise before focusing to the stochastic policy evaluation problem. 
Our investigations encompass a variety of simple TD learning  type algorithms motivated by its original version that maintain its simplicity, while offering distinct advantages from a non-asymptotic analysis point of view. These algorithms are applicable in the more general setting of stochastic VIs. 

Our main contributions are summarized as follows. We start 
by analyzing the TD algorithm for stochastic VIs in a proximal setting, which only requires the update of one sequence $ \{ x_t\}$, while receiving samples $ \xi_1, \xi_2, \hdots,$ from the underlying  Markov process.
 Each iteration involves the evaluation of a potentially biased estimator 
 $ \tilde{F}(x_t,\xi_t) $ of the operator value $F(x_t)$ and updating from $x_t$ to $ x_{t+1}$  through only one projection subproblem:
$$
	x_{t+1} = \argmin_{x\in X} ~   \gamma_t \big\langle \tilde{F}(x_t,\xi_t)  , x \big\rangle + V(x_t,x),
$$
where $V : X \times X \rightarrow \mathbb{R}_+$ denotes the Bregman distance,  see also \eqref{Bregman_definition}.
In the single oracle setting we prove the 
$
\mathcal{O}( \tfrac{\tau (\bar L^2+\varsigma^2)}{\mu^2\sqrt{\epsilon}}+ 
\sqrt{\tfrac{\tau(\bar L^2+\varsigma^2)}{\mu^{3} (1-\rho)\epsilon}}
 + \tfrac{\tau (\sigma^2+\|F(x^*)\|_*^2)}{\mu^2 \epsilon}   )
$
and 
$
\mathcal{O}( \tfrac{\tau (\bar L^2+\varsigma^2)}{\mu^2}
\log{\tfrac{1}{\epsilon(1-\rho)}} + \tfrac{\tau (\sigma^2+\|F(x^*)\|_*^2)}{\mu^2 \epsilon} \log{\tfrac{1}{\epsilon}} )
$
sampling complexity bounds, by using a diminishing stepsize and a new constant stepsize policy respectively. 
The parameter $\rho$ corresponds to the covergence rate of the underlying Markov process, while
$ \tau = \mathcal{O}(
\tfrac{\log(1/\mu)}{\log(1/\rho)})$.
 The constants $\bar L$ and $\sigma$,  $\varsigma$ are  described in the problem setup (see Section~\ref{sec_problem_statement}).\footnote{Note that the  smaller Lipschitz constant $L$ instead of $\bar L$ is used in the next algorithms and related iteration complexity expressions. The parameter $\sigma^2$ is the constant variance term,  while $\varsigma^2$ is the coefficient of the state-dependent noise.} To our knowledge these complexity bounds are new for solving stochastic GSMVI problems with Markovian noise and can benefit by the use of mini-batches for variance reduction purposes when $F(x^*)=0$.
The parameter $\tau$, closely connected to the mixing time of the Markov process, affects
the stepsize selection. Our analysis of the standard TD  algorithm requires the estimation of $\tau$. This circumstance motivated us to devise a new algorithmic scheme, called conditional temporal difference (CTD) method, applying updates to the iterates in periodic intervals of length 
$ \tau$, while achieving several appealing properties. Specifically, at time instant $t$ we collect $\tau$ samples 
$\{\xi_t^1, \xi_t^2, \dots, \xi_t^\tau\}$, and use the last one to update the iterate according to
$$ 
x_{t+1} = \argmin_{x\in X} ~   \gamma_t \big\langle \tilde{F}(x_t,\xi_t^\tau)  , x \big\rangle + V(x_t, x).
$$
In the single oracle setting the CTD algorithm achieves the 
$
\mathcal{O} (  \tfrac{\tau (L^2+\varsigma^2)}{\mu^2\sqrt{\epsilon}}
+ \tfrac{\tau \sigma^2}{\mu^2 \epsilon}  )
$
and 
$
\mathcal{O} ( \tfrac{\tau (L^2+\varsigma^2)}{\mu^2}
\log{\tfrac{1}{\epsilon}} + \tfrac{\tau \sigma^2}{\mu^2 \epsilon} \log{\tfrac{1}{\epsilon}} ) 
$
sampling complexity bounds, by using a diminishing stepsize and a new constant stepsize policy respectively.  
CTD improves TD in two different aspects: a) its complexity bounds have weak dependence
on $\rho$ in particular when $ \rho $ is close to 1, and b) it uses relaxed assumptions on the
smoothness of the stochastic operators.  The real advantage of the 
CTD over the standard TD algorithm is exemplified by a new index-resetting stepsize policy 
that achieves iteration complexity of ${\cal O}( \tfrac{\tau (L^2+\varsigma^2) }{\mu^2} \log\tfrac{1}{\epsilon}+ \tfrac{\tau \sigma^2 }{\mu^2 \epsilon} ) $, which to our knowledge is the best possible, without the use of operator extrapolation. 
This brings us to our next algorithm that combines the elements of CTD and the stochastic operator extrapolation (SOE) introduced in the companion paper \cite{GGT_20a}, which 
is referred to as fast temporal difference (FTD) method: 
$$
x_{t+1} = \argmin_{x\in X_{t+1}} ~   \gamma_t \big\langle \tilde{F}(x_t,\xi_t^\tau) + \lambda_t[\tilde{F}(x_t,\xi_t^\tau) -\tilde{F}(x_{t-1},\xi_{t-1}^\tau) ] , x \big\rangle + V(x_t, x).
$$
In this algorithm the projecting set may be time varying depending whether the feasible set is bounded or not.
We show that FTD when employed with different stepsize policies can achieve either nearly optimal or optimal complexity for stochastic GSMVI under Markovian noise in both settings.
When the feasible region is bounded by using an index-resetting stepsize policy 
FTD achieves the complexity bound
${\cal O}( \tfrac{\tau L }{\mu} \log\tfrac{V(x_1,x^*)}{\epsilon}+ \tfrac{\tau(\sigma^2+\varsigma^2 D_X^2) }{\mu^2 \epsilon} )$
which improves the one obtained for CTD in terms of the dependence on
the condition number $L/\mu$ in the first term to the optimal one.    This rate can be matched in the case of an unbounded feasible region by employing projection, while utilizing an a-priori upper bound on the size of the optimal solution. 
The claims of optimality or near optimality of specific algorithms in this paper are grounded on the following references. In the deterministic setting a lower bound on the linear rate of convergence was established in the context of solving linear systems of equations,  a special case of VIs,  in 
\cite{nemirovski1992information}.   As for the stochastic case in the  specific context of  convex optimization,  the ${\cal O}(\tfrac{1}{\epsilon})$ lower bound on the sample complexity can be found in \cite{ghad_lan_2012} and \cite{jud_nest_2014}.

To the best of our knowledge all these complexity bounds are new and can benefit from the use of mini-batches for reducing the variance and hence facilitate distributed stochastic optimization in terms of multiprocessor and multiagent parallelization.
In addition, we establish convergence of FTD for stochastic GMVI ($\mu = 0$) under Markovian noise in terms of the expected residual,
which, to our best knowledge, can not be achieved by either TD or CTD. 

 In terms of policy evaluation for MDPs, our smooth analysis of the standard TD algorithm does not require a projection step to an a-priori specified ball, while improving on the achieved convergence rate.  Furthermore the aforementioned benefits of CTD and FTD naturally carry over to this reinforcement learning problem as well. Those advantages in terms of improved iteration complexity and amenability to parallel implementation are amplified by our robust analysis of FTD. The latter analysis in terms of the expected residual, which corresponds to the expected Bellman error, is suitable for cases where the discount factor $ \beta \approx 1$, i.e., $ \mu \approx 0 $.  

Finally we conduct numerical experiments on the proposed algorithms in regards to policy evaluation. We use the standard 2D Grid-World example as a testbed to demonstrate the performance of our algorithms while demonstrating the advantages over the TD method requiring the projection over a bounded set.

This paper is organized as follows. The underlying assumptions to our algorithmic developments are stated in 
Section~\ref{sec_problem_statement} and elaborated by means of a nonlinear estimation problem involving autoregressive dynamics.
We discuss our novel analysis of standard TD as well as the newly developed CTD and FTD algorithms in 
Section~\ref{sec-algorithms}. In Section~\ref{MDP_RL} we demonstrate the applicability of our methods to reinforcement learning for the policy evaluation problem and compare to prior work in the literature. We report out numerical results in Section~\ref{sec-num} and complete the paper with some brief concluding remarks.

\textbf{Notation and terminology} 
For a given strongly convex function $\o$ with modulus $1$, we define the prox-function (or Bregman's distance) associated with $\o$ as
\begin{equation}
\label{Bregman_definition}
V(x,y) \equiv V_{\o}(x,y): = \o(x)-\o(y)-\la \o^{\prime}(y),x-y\ra, \quad \forall x,y \in X,
\end{equation}
where $\o^{\prime}(y)\in \partial \o(y)$  is an arbitrary subgradient of $\o$ at $y$.
Note that by the strong convexity of $\o$, we have
\beq \label{strong_convex_V}
V(x,y) \ge \tfrac{1}{2} \|x - y\|^2.
\eeq
With the definition of the Bregman's distance, we can replace \eqnok{G_monotone0} by
\begin{equation} \label{G_monotone1}
\langle F(x), x - x^* \rangle \ge 2 \mu V(x, x^*), \ \ \forall x \in X.
\end{equation}
For a given Markov process $\{ \xi_{t} : t \in \mathbb{Z}_{+}\}$ over some probability space $ (\Omega, \mathcal{F}, \mathbb{P})$  with state space $\Xi$, we denote the $\sigma$-field generated by the first  $ \xi_1, \hdots, \xi_{s}$ samples by  $ \mathcal{F}_s = \sigma( \xi_1, \hdots, \xi_{s})$. 
The distribution of $\xi_t$ conditioned on $ \mathcal{F}_s$ is denoted by $P^t_{[s]}$, i.e., for a measurable event $A \in \mathcal{B}(\Xi)$, $ P^t_{[s]}(A) = \Pr(\xi_t \in A | \mathcal{F}_s)$. 

\section{Problem statement and assumptions}   \label{sec_problem_statement}

The problem at hand is the computation of a point $x^* \in X$ such that \eqref{VIP} holds, while the operator $F$
satisfies Lipshitz-continuity \eqref{Lipschitz} and the generalized monotonicity condition \eqref{G_monotone0} for some $\mu \geq 0$. At our disposal are samples received from a Markov process $ \{\xi_{t} : t \in \mathbb{Z}_+\}$
defined on some probability space $ (\Omega, \mathcal{F},\mathbb{P})$
with state space $\Xi$. The underlying Markov process admits a unique invariant distribution $\Pi$. Our main stipulation is that when $\xi \in \Xi$ is a random vector  with distribution $\Pi$, the stochastic oracle delivers 
an unbiased estimate $\tilde{F}(x, \xi)$ of the operator $ F(x)$ in the sense of \eqref{unbiased_estimate}.
Our algorithmic developments rely on some additional assumptions.

\textit{Assumption A.} 
The maps $ \tilde{F}( \cdot, \xi)$ are Lipschitz-continuous uniformly in $\xi$, i.e., there exists 
$\tilde{L} >  0$, such that  $\forall \xi \in \Xi$, 
\begin{equation}
\label{Lipschitz_pointwise}
\| \tilde{F}(x_1, \xi) - \tilde{F}(x_2, \xi) \|_*   \leq \tilde L \| x_1 - x_2 \|,~~~~ \forall x_1, x_2 \in X.
\end{equation}
We set $\bar{L} := \max \{L,\tilde{L}\}$.  This assumption is needed only in the analysis of the TD algorithm. For the CTD and FTD algorithms the standard Lipschitz-continuity condition as in \eqref{Lipschitz}
suffices.

\textit{Assumption B.} At the point $x^* \in X$ for every $t> s \in \mathbb{Z}_+$, with probability 1,
\begin{equation} \label{unbiased_x*}
\mathbb{E}[ \tilde{F}(x^*, \xi_{t})|\mathcal{F}_{s}] = F(x^*).
\end{equation}

\textit{Assumption C.}  
There exists constants $ \sigma, \varsigma \in \mathbb{R}_+$ such that  for every $t, \tau \in \mathbb{Z}_+$ and
 iterates $\{x_s\}, $
\begin{equation} \label{variance_bound}
\mathbb{E}[ \|\tilde{F}(x_t, \xi_{t+\tau}) -  
\mathbb{E}[ \tilde{F}(x_t, \xi_{t+\tau}) | \mathcal{F}_{t-1}] \|_*^2 |\mathcal{F}_{t-1} ] \leq \tfrac{\sigma^2}{2} + \tfrac{\varsigma^2}{2}\|x_t-x^*\|^2.
\end{equation} 
Implementation of our algorithms requires estimation of $\varsigma$, but does not necessarily require 
estimation of  $\sigma$.

\textit{Assumption D.}   There exists  constants $C > 0$ and $ \rho \in (0,1)$, such that for 
 every $t, \tau \in \mathbb{Z}_+$ and \textcolor{black}{$x \in X$}, with probability 1,
\begin{equation} \label{bound_delta_inner_lhs}
\| F(x) - \mathbb{E}[\tilde{F}(x, \xi_{t+\tau})|\mathcal{F}_{t-1}] \|_*  \leq C\rho^\tau \|x-x^*\|.
\end{equation}
As a consequence of this assumption it follows that with probability 1, 
\begin{equation} \label{bound_delta_inner1}
 \langle F(x) - \mathbb{E}[\tilde{F}(x, \xi_{t+\tau})|\mathcal{F}_{t-1}], x -x^*\rangle  \leq C\rho^\tau \|x-x^*\|^2.
\end{equation}
Our algorithms exhibit only logarithmic dependence on $\rho$ and $C$ (see \eqref{def_tau}).
As such only some rough estimation of these parameters is sufficient. The parameter $\rho$ relates to the convergence rate of the underlying Markov chain,  i.e. how fast the chain approaches its stationary distribution. Estimation of the 	
convergence rate, as well as the related notion of mixing time, i.e. the number of steps required for the Markov chain to be within a fixed threshold of its stationary distribution has been the topic of active research. Nontrivial confidence intervals for the reversible case can be found in \cite{everybody2019}. The more challenging and prevalent case when the underlying Markov chain is non-reversible
is addressed in 
\cite{wolfer2019estimating}. 

Moreover,  whenever there is a-priori knowledge in terms of a positive lower bound to the multi-step transition probabilities to a specific state, one can build estimates for $\rho$ and $C$ via the Doeblin minorization condition, see \cite{rosenthal95} and the references therein.

Below are two useful results for the analysis of our algorithms.

\begin{lemma} \label{bound_delta_norm}
Let \textit{Assumptions C} and \textit{D} hold. For every $t,\tau \in \mathbb{Z}_+$ and  sequence  of iterates  $\{x_t\} $, 
\beq \label{bound_delta_norm1}
\mathbb{E}[\| F(x_t) - \tilde{F}(x_t, \xi_{t+\tau}) \|_*^2 ] \le \sigma^2 + (2C^2\rho^{2\tau} + \varsigma^2 )\mathbb{E}[\|x_t-x^*\|^2] .
\eeq 
\end{lemma}

\begin{proof} 
	By the triangle inequality, we have 
	\begin{align*}
	\| F(x_t) - \tilde{F}(x_t, \xi_{t+\tau}) \|_* \le \| F(x_t) - \bbe[\tilde{F}(x_t, \xi_{t+\tau})|\mathcal{F}_{t-1}] \|_* + \| \bbe[\tilde{F}(x_t, \xi_{t+\tau})|\mathcal{F}_{t-1}] -  \tilde{F}(x_t, \xi_{t+\tau})\|_*.
	\end{align*}
	Taking squares on both sides  and applying Young's inequality, we obtain
	\begin{align}\label{splitting_terms}
	\| F(x_t) - \tilde{F}(x_t, \xi_{t+\tau}) \|_*^2  
	 \le 2\| F(x_t) - \bbe[\tilde{F}(x_t, \xi_{t+\tau})|\mathcal{F}_{t-1}] \|_*^2 + 2\| \bbe[\tilde{F}(x_t, \xi_{t+\tau})|\mathcal{F}_{t-1}] -  \tilde{F}(x_t, \xi_{t+\tau})\|_*^2.
	\end{align}
By taking \eqnok{bound_delta_inner_lhs} into account it follows that, 
\begin{align*}
\| F(x_t) - \tilde{F}(x_t, \xi_{t+\tau}) \|_*^2  
\le  2C^2\rho^{2\tau}\|x_t-x^*\|^2+ 2\| \bbe[\tilde{F}(x_t, \xi_{t+\tau})|\mathcal{F}_{t-1}] -  \tilde{F}(x_t, \xi_{t+\tau})\|_*^2.
\end{align*}
Relation \eqref{bound_delta_norm1} is obtained by applying \eqnok{variance_bound} and subsequently taking expectations.
\end{proof}

\begin{lemma} \label{bound_delta_norm_0}
	Let \textit{Assumptions A, B} and \textit{C} hold. For every $t \in \mathbb{Z}_+$ and   sequence of iterates  $\{x_s\}, $ 
	\beq \label{bound_delta_norm0}
	\mathbb{E}[\| F(x_t) - \tilde{F}(x_t, \xi_{t}) \|_*^2 ] \le \sigma^2 + ( \varsigma^2+8\bar{L}^2)\   \mathbb{E}[ \|x_t-x^*\|^2] .
	\eeq
\end{lemma}
\begin{proof} 
	By letting $\tau=0$ in \eqnok{splitting_terms},  we obtain
	\begin{align}\label{splitting_terms2}
	\| F(x_t) - \tilde{F}(x_t, \xi_{t}) \|_*^2  
	 \le 2\| F(x_t) - \bbe[\tilde{F}(x_t, \xi_{t})|\mathcal{F}_{t-1}] \|_*^2 + 2\| \bbe[\tilde{F}(x_t, \xi_{t})|\mathcal{F}_{t-1}] -  \tilde{F}(x_t, \xi_{t})\|_*^2.
	\end{align}
The term  $\| F(x_t) - \bbe[\tilde{F}(x_t, \xi_{t})|\mathcal{F}_{t-1}] \|_*$ is upper bounded as 
	\begin{align*}
	\| F(x_t) - \bbe[\tilde{F}(x_t, \xi_{t})|\mathcal{F}_{t-1}] \|_* 
	 & =\| F(x_t) - F(x^*) + \bbe[\tilde{F}(x^*, \xi_{t})|\mathcal{F}_{t-1}] - \bbe[\tilde{F}(x_t, \xi_{t})|\mathcal{F}_{t-1}] \|_*\\
	& = 	\| \bbe_{\Pi}[\tilde F(x_t,\xi)-\tilde F(x^*,\xi)]- \bbe[\tilde{F}(x_t, \xi_{t})-\tilde{F}(x^*, \xi_{t})|\mathcal{F}_{t-1}] \|_*\\
 & \le \int_{\xi \in \Xi} \|\tilde{F}(x_t, \xi)-\tilde{F}(x^*, \xi)\|_* |d \Pi(\xi)-d P^t_{[t-1]}(\xi)| \le 2\bar{L}\|x_t-x^*\|. 
	\end{align*}
	The above derivation utilizes \eqnok{Lipschitz_pointwise}, \eqnok{unbiased_x*}, and the boundedness of the total variation distance. By taking  \eqnok{splitting_terms2} into account we obtain the desired result. 
\end{proof}

Next we present an example where we discuss the aforementioned assumptions. 

\subsection*{Signal estimation and generalized linear models} \label{arkadi examples}

We revisit the nonlinear signal estimation problem involving generalized linear models (GLMs) that we
considered in \cite{GGT_20a} 
(see also Chapter 5.2 of  \cite{ArkadiBook2020}).  The setup is the same with one notable difference.
In the spirit of the current setting we relax the i.i.d. assumption on the regressors $ \{ \eta_t : t\in \mathbb{Z}_+\}$, who  now form a Markov process. Specifically we assume that the regressor sequence is generated by 
an autoregressive process of the form
$$
\eta_{t+1} = B \eta_t + \epsilon_t, ~~ t= 1, 2, \hdots,~ \eta_1 = 0,~ \epsilon_t \in \mathbb{R}^n, $$
where $ \epsilon_t$ is an i.i.d. sequence of  zero-mean random vectors with finite support,
finite moments of all orders, and
covariance matrix  $ \mathbb{E}[\epsilon_t \epsilon_t^{\T} ] = Q $.  Furthermore the matrix $B$ is stable,  i.e.
$\sigma(B) < 1 $,  where $ \sigma(B) = \max \{ |\lambda| ,   \lambda ~ \text{eigenvalue of}~ B \}$ denotes the spectral radius of $B$.
 A sequence of regressor-label observations 
$
\xi^K = \{  \xi_k = (\eta_k, y_k), 1 \leq k \leq K  \}
$
is generated according to a distribution 
$P_{x^*}$, where $x^* \in \mathcal{X} \subset 
\mathbb{R}^n$ is an unknown signal lying in a convex compact set. 
At each instant $t$, 
$
\mathbb{E}[y_t | \eta_t] = f(\eta_t^{\T}  x^*),
$
where  $f: \mathbb{R} \rightarrow \mathbb{R}$ is a Lipschitz continuous link function.
The inference problem for $x^*$ admits a VI formulation, the relevant operators are  
$$
F(x) = \mathbb{E}[\eta f(\eta^{\T}  x) ]-
\mathbb{E}[\eta f(\eta^{\T}  x^*)], ~~
\tilde{F}(x, \xi_t) = \eta_t f(\eta_t^{\T}  x) - \eta_t y_t.
$$
In the above expression for $F(x)$, 
the r.v. $\eta$ is distributed according to the stationary distribution of the autoregressive process.
As a consequence of the fact that $B$ is a stable matrix, it follows that the trajectory of  $ \{ \eta_t : t\in \mathbb{Z}_+\}$ remains bounded, while its covariance matrix $ X_t = \mathbb{E}[\eta_t \eta_t^{\T}]$ reaches a steady state value $\lim_{t \rightarrow \infty} X_t = X$. The covariance dynamics are given by 
$$
X_{t+1} = B X_t B^{\T}+  Q, ~ t= 1, 2, \hdots,\text{where} ~ 
 X =  B X B^{\T}+  Q.
$$
With this in mind we note that Assumption A  follows from the boundedness of the regressor  trajectory and the Lipschitz continuity of the link function. Assumption B is a consequence of 
$
\mathbb{E}[y_t | \eta_t] = f(\eta_t^{\T}  x^*),
$
note also that $ F(x^*) = 0 $.
Assumption C follows again from the fact that $ \{ \eta_t : t\in \mathbb{Z}_+\}$ remains bounded. 
As for Assumption D, let $\pi$ and $p^{t+\tau}_{[t-1]}$ denote the densities of $\Pi$ and  $P^{t+\tau}_{[t-1]}$, respectively. Note that
\begin{align*}
\| F(x) - \bbe[\tilde{F}(x, \xi_{t+\tau}) | \mathcal{F}_{t-1}]\| =& \| F(x) - \bbe[\tilde{F}(x, \xi_{t+\tau}) | \mathcal{F}_{t-1}] - \big(F(x^*) - \bbe[\tilde{F}(x^*, \xi_{t+\tau}) | \mathcal{F}_{t-1}]\big)\|\\
=&\bigg\| \int_{\xi \in \Xi}\big(\tilde{F}(x, \xi)-\tilde{F}(x^*, \xi)\big)\big( \pi(\xi)- p^{t+\tau}_{[t-1]}(\xi) \big)d\mu(\xi)\bigg\|\\
  \le& \int_{\xi \in \Xi} \|\tilde{F}(x, \xi)-\tilde{F}(x^*, \xi)\|| \pi(\xi)- p^{t+\tau}_{[t-1]}(\xi) |d\mu(\xi)
\le \tilde{L}\|x-x^*\| ~ d_{TV}(\Pi, P^{t+\tau}_{[t-1]}).
\end{align*}
In view of this,
starting from a given $\eta_t$ at time $t$ the mean of the regressor converges to zero with a rate $\sigma(B)$, while 
the covariance matrix converges with $\sigma(B)^2$.  Given the prior calculations condition  \eqref{bound_delta_inner_lhs}
follows from proposition 2.1 in \cite{DMR_20_total_variation_gaussian}, that shows the convergence of the total variation distance $  d_{TV}(\Pi, P^{t+\tau}_{[t-1]})$ at a geometric rate.

\section{Algorithms for stochastic variational inequalities with Markovian noise}  \label{sec-algorithms}

The algorithmic schemes we consider are simple in the sense that they involve a single sequence of iterates $\{x_t\}$ along with at most two sequences of nonnegative parameters $ \{ \gamma_t\} $,  $\{ \lambda_t\}$
and a prox-function $V : X \times X \rightarrow \mathbb{R}$. The parameters  $\{ \lambda_t\}$, when employed, define the extrapolation step, while the parameters $\{ \gamma_t \} $  can be viewed as step-sizes. The sequence $ \{\xi_t \}$ represents a single draw from the underlying  Markov process.

\subsection{Temporal difference algorithm}   

We start by analyzing the TD algorithm in the proximal setting, which is essentially the stochastic projected gradient/operator method for VI
under the Markovian \textcolor{black}{noise} setting. Notice that we do not require the feasible set $X$ to be bounded.

\begin{algorithm}[H]  \caption{Temporal Difference Algorithm}  
	\label{alg:TD}
	\begin{algorithmic} 
		\STATE{Let $ x_1 \in X$, and the nonnegative parameters $\{ \gamma_t\}$ be given. }
		\FOR{$ t = 1, \ldots, k$}
		\STATE{
		\beq \label{TD_step}
		x_{t+1} = \argmin_{x\in X} ~   \gamma_t \big\langle \tilde{F}(x_t,\xi_t)  , x \big\rangle + V(x_t,x).
		\eeq}
		\ENDFOR
	\end{algorithmic}
\end{algorithm} 
Before we analyze the convergence behavior of the TD method,
we discuss a few different termination criteria for the VI problem in \eqnok{VIP}.
If $F$ satisfies the generalized strong monotonicity condition in \eqnok{G_monotone1} for some $\mu> 0$, then
the distance to the optimal solution $V(x_k, x^*)$
will be a natural termination criterion. Otherwise, when $ \mu = 0$, the termination 
criterion will be based on the residual. 
To this end, see also Section 3.8.2 of~\cite{LanBook2020},
let us denote the normal cone of $X$ at $\bar x$ by
\beq \label{def_N_X}
N_X(\bar x) := \{y \in \bbr^n | \langle y, x - \bar x\rangle \le 0, \forall x \in X \}.
\eeq
Noting that $\bar x \in X$ is an optimal solution for problem~\eqnok{VIP} if and only if $F(\bar x) \in -N_X(\bar x)$,
we define the residual of $\bar x$ as
\beq \label{def_res}
\res(\bar x) := \min_{y \in - N_X(\bar x)} \|y - F(\bar x)\|_*.
\eeq
In particular, if $X = \bbr^n$, then $N_X(\bar x) = \{0\}$ and $\res(\bar x) = \|F(\bar x)\|_*$, which is exactly
the residual of solving the nonlinear equation $F(\bar x) = 0$.

The following lemma, often referred to as the ``three-point lemma", characterizes the optimality condition of problem~\eqnok{TD_step}.
\begin{lemma} \label{lemma_wo_projection}
	Let $x_{t+1}$ be defined in \eqnok{TD_step}. Then,
	\beq \label{opt_deter_step1}
	\gamma_t \langle  \tilde{F}(x_t,\xi_t) , x_{t+1} - x \rangle + V(x_t,x_{t+1}) \le V(x_t,x) - V(x_{t+1},x), ~~\forall x \in X.
	\eeq
\end{lemma}
Henceforth we will also make use of the contraction property of the iterates  $ \{x_t\}$ as stated in the following lemma, a proof of which can be found in Lemma 6.5 in  \cite{LanBook2020}. 
\begin{lemma} \label{contraction_TD}
	Let $x_{t+1}$ be defined in \eqnok{TD_step}. Then,
	\beq \label{contraction_TD_step}
	\|x_{t+1}-x_t\|\le \gamma_t\|\tilde F(x_t,\xi_t)\|_*.
	\eeq
\end{lemma}

\begin{proposition}\label{lemma_dist1}
	{\color{black}Let \textit{Assumption A} hold.} Let $\{x_t\}$ be generated according to Algorithm \ref{alg:TD},
	then for any $x \in X$, $\kappa \in [t-1]$,
	\begin{align*}
	&(1-2\gamma_t^2\bar{L}^2)V(x_{t+1},x)  +  \gamma_t \langle F(x_{t+1}),x_{t+1}-x \rangle + \gamma_t\langle\tilde{F}(x_{t-\kappa},\xi_t)-F(x_{t-\kappa}) , x_{t-\kappa}-x\rangle \\
	& \leq  V(x_t,x) +\gamma_t^2\|\tilde{F}(x_t,\xi_t)-F(x_{t}) \|_*^2+ \tsum_{t' = t-\kappa}^{t-1}\gamma_t\gamma_{t'} \big(\|\tilde{F}(x_t,\xi_t)-F(x_{t}) \|_*^2 \\& \quad +  \tfrac{3\bar{L}^2}{2}\|x_{t'}-x\|^2 +\tfrac{3}{2}\|\tilde{F}(x,\xi_{t'})\|_*^2+2\bar{L}^2\|x_{t-\kappa}-x\|^2 \big).
	\end{align*}
\end{proposition}
\begin{proof} 
	The inner product term $\langle \tilde{F}(x_t,\xi_t) , x_{t+1} - x \rangle  $ in \eqnok{opt_deter_step1}
	is written equivalently as
	\begin{align}\label{bound_1}
	\langle \tilde{F}(x_t,\xi_t) , x_{t+1} - x \rangle  
	& = \langle F(x_{t+1}),x_{t+1}-x \rangle  \nn\\
& 	\quad +\langle F(x_t)-F(x_{t+1}) , x_{t+1} - x \rangle \nn \\
	& \quad +  	\langle\tilde{F}(x_t,\xi_t)-F(x_{t}) , x_{t+1} - x \rangle.
	\end{align}	
	By Lipschitz-continuity the second inner product term  is lower bounded as 
	\begin{equation}
	\label{bound_15}
	\langle F(x_t)-F(x_{t+1}) , x_{t+1} - x \rangle
	\geq -L\|x_t-x_{t+1}\|\|x_{t+1}-x\|.
	\end{equation} 	
	As for the third inner product term in \eqnok{bound_1} we interject an intermediate iterate $x_{t-\kappa}$
	before applying the Lipschitz-continuity conditions. 
	\begin{align}
		\langle\tilde{F}(x_t,\xi_t)-F(x_{t}) , x_{t+1} - x \rangle
	& = \langle\tilde{F}(x_t,\xi_t)-F(x_{t}) , x_{t+1} - x_{t-\kappa} \rangle + \langle\tilde{F}(x_t,\xi_t)-\tilde{F}(x_{t-\kappa},\xi_t) , x_{t-\kappa} - x \rangle\nn\\
	& \quad+\langle\tilde{F}(x_{t-\kappa},\xi_t)-F(x_{t-\kappa}) , x_{t-\kappa} - x \rangle+\langle F(x_{t-\kappa})-F(x_{t}) , x_{t-\kappa} - x \rangle\nn\\
	& \geq  - \|\tilde{F}(x_t,\xi_t)-F(x_{t}) \|_*\|x_{t+1}-x_{t-\kappa}\| - \tilde{L}\|x_t-x_{t-\kappa}\|\|x_{t-\kappa}-x\|\nn\\
	& \quad +  \langle\tilde{F}(x_{t-\kappa},\xi_t)-F(x_{t-\kappa}) , x_{t-\kappa} - x \rangle - L\|x_t-x_{t-\kappa}\|\|x_{t-\kappa}-x\|.\label{bound_2}
	\end{align}	
	We will  bound the individual terms in the rhs of \eqnok{bound_2}.
	By invoking Lemma \ref{contraction_TD} and the triangle inequality		
	\begin{align} \label{bound_tau}
	\|x_t - x_{t-\kappa}\|  
	 \leq \tsum_{t' = t-\kappa}^{t-1}\gamma_{t'}\|\tilde{F}(x_{t'},\xi_{t'})\|_*  
	 \leq  \tsum_{t' = t-\kappa}^{t-1}\gamma_{t'}(\tilde{L}\|x_{t'}-x\|+\|\tilde{F}(x,\xi_{t'})\|_*).
	\end{align}
Utilizing the above bound and Young's inequality we obtain	
		\begin{align} \label{bound_3}
		\|x_t - x_{t-\kappa}\|\|x_{t-\kappa}-x\| 
		\leq ~& \tsum_{t' = t-\kappa}^{t-1}\gamma_{t'}(\bar{L}\|x_{t'}-x\|\|x_{t-\kappa}-x\| +\|\tilde{F}(x,\xi_{t'})\|_*\|x_{t-\kappa}-x\| )\nn\\
		\leq ~&\tsum_{t' = t-\kappa}^{t-1}\tfrac{1}{2}\gamma_{t'}(\bar{L}\|x_{t'}-x\|^2+2\bar{L}\|x_{t-\kappa}-x\|^2+\tfrac{1}{\bar{L}}\|\tilde{F}(x,\xi_{t'})\|_*^2).
		\end{align}
The  relation in \eqnok{bound_3} pertains to the two terms in \eqnok{bound_2} stemming from applying the Lipschitz-continuity of the operator and its stochastic version respectively. 
We proceed to  the term $\|\tilde{F}(x_t,\xi_t)-F(x_{t}) \|_*\|x_{t+1}-x_{t-\kappa}\|$ in \eqnok{bound_2},
where again we utilize \eqnok{bound_tau} and Young's inequality to obtain
		\begin{align} \label{bound_4}
		\|\tilde{F}(x_t,\xi_t)-F(x_{t}) \|_*\|x_{t+1}-x_{t-\kappa}\|
		& \leq  \|\tilde{F}(x_t,\xi_t)-F(x_{t}) \|_*\|x_{t+1}-x_{t}\| + \|\tilde{F}(x_t,\xi_t)-F(x_{t}) \|_*\|x_{t}-x_{t-\kappa}\|\nn\\
		& \leq  \|\tilde{F}(x_t,\xi_t)-F(x_{t}) \|_*\big(\tsum_{t' = t-\kappa}^{t-1}\gamma_{t'}(\tilde{L}\|x_{t'}-x\|+\|\tilde{F}(x,\xi_{t'})\|_*)\big)\nn\\
		& \quad + \|\tilde{F}(x_t,\xi_t)-F(x_{t}) \|_*\|x_{t+1}-x_{t}\|\nn\\
		&\leq  \tsum_{t' = t-\kappa}^{t-1}\gamma_{t'}(\|\tilde{F}(x_t,\xi_t)-F(x_{t}) \|_*^2 + \tfrac{\bar{L}^2}{2}\|x_{t'}-x\|^2+\tfrac{1}{2}\|\tilde{F}(x,\xi_{t'})\|_*^2)\nn\\
		& \quad + \|\tilde{F}(x_t,\xi_t)-F(x_{t}) \|_*\|x_{t+1}-x_{t}\|.
		\end{align}
The derived bounds in \eqnok{bound_3} and \eqnok{bound_4} are substituted back into \eqnok{bound_2}, 
		\begin{align} \label{bound_5}
		\langle\tilde{F}(x_t,\xi_t)-F(x_{t}) , x_{t+1} - x \rangle 
		& \geq  - \tsum_{t' = t-\kappa}^{t-1}\gamma_{t'}(\|\tilde{F}(x_t,\xi_t)-F(x_{t}) \|_*^2 + \tfrac{\bar{L}^2}{2}\|x_{t'}-x\|^2+\tfrac{1}{2}\|\tilde{F}(x,\xi_{t'})\|_*^2)\nn\\
		& \quad - \|\tilde{F}(x_t,\xi_t)-F(x_{t}) \|_*\|x_{t+1}-x_{t}\|+ \langle\tilde{F}(x_{t-\kappa},\xi_t)-F(x_{t-\kappa}) , x_{t-\kappa}-x\rangle \nn\\
		& \quad -\bar{L}\tsum_{t' = t-\kappa}^{t-1}\gamma_{t'}(\bar{L}\|x_{t'}-x\|^2+2\bar{L}\|x_{t-\kappa}-x\|^2+\tfrac{1}{\bar{L}}\|\tilde{F}(x,\xi_{t'})\|_*^2)\nn\\
		& = - \tsum_{t' = t-\kappa}^{t-1}\gamma_{t'}(\|\tilde{F}(x_t,\xi_t)-F(x_{t}) \|_*^2 + \tfrac{3\bar{L}^2}{2}\|x_{t'}-x\|^2+\tfrac{3}{2}\|\tilde{F}(x,\xi_{t'})\|_*^2 \nn\\
		& \quad + 2\bar{L}^2\|x_{t-\kappa}-x\|^2)  - \|\tilde{F}(x_t,\xi_t)-F(x_{t}) \|_*\|x_{t+1}-x_{t}\| \nn\\
		& \quad+ \langle\tilde{F}(x_{t-\kappa},\xi_t)-F(x_{t-\kappa}) , x_{t-\kappa}-x\rangle.
		\end{align}
By taking into account \eqref{opt_deter_step1}, \eqnok{bound_1} and\eqnok{bound_15} we can lower bound the difference $\Delta V_t(x)  =  V(x_t, x) - V(x_{t+1}, x)$, 
		\begin{align}
		\label{Delta_V_bound}
		\Delta V_t(x) & \geq  \gamma_t \langle F(x_{t+1}),x_{t+1}-x \rangle - \gamma_t\bar{L}\|x_t-x_{t+1}\|\|x_{t+1}-x\|+V(x_t,x_{t+1})\nn\\
		& \quad - \tsum_{t' = t-\kappa}^{t-1}\gamma_t\gamma_{t'}(\|\tilde{F}(x_t,\xi_t)-F(x_{t}) \|_*^2 + \tfrac{3\bar{L}^2}{2}\|x_{t'}-x\|^2+\tfrac{3}{2}\|\tilde{F}(x,\xi_{t'})\|_*^2 +2\bar{L}^2\|x_{t-\kappa}-x\|^2)\nn\\
		& \quad  -  \gamma_t\|\tilde{F}(x_t,\xi_t)-F(x_{t}) \|_*\|x_{t+1}-x_{t}\|+ \gamma_t\langle\tilde{F}(x_{t-\kappa},\xi_t)-F(x_{t-\kappa}) , x_{t-\kappa}-x\rangle \nn\\
		& \geq \gamma_t \langle F(x_{t+1}),x_{t+1}-x \rangle -2\gamma_t^2\bar{L}^2V(x_{t+1}-x)-\gamma_t^2\|\tilde{F}(x_t,\xi_t)-F(x_{t}) \|_*^2\nn\\
		& \quad -  \tsum_{t' = t-\kappa}^{t-1}\gamma_t\gamma_{t'}(\|\tilde{F}(x_t,\xi_t)-F(x_{t}) \|_*^2 + \tfrac{3\bar{L}^2}{2}\|x_{t'}-x\|^2+\tfrac{3}{2}\|\tilde{F}(x,\xi_{t'})\|_*^2+2\bar{L}^2\|x_{t-\kappa}-x\|^2)\nn\\
		& \quad + \gamma_t\langle\tilde{F}(x_{t-\kappa},\xi_t)-F(x_{t-\kappa}) , x_{t-\kappa}-x\rangle,
		\end{align}
		where the second inequality follows from 
		\begin{align*}
		& \quad - \gamma_t\bar{L}\|x_t-x_{t+1}\|\|x_{t+1}-x\| + V(x_t,x_{t+1}) -\gamma_t\|\tilde{F}(x_t,\xi_t)-F(x_{t}) \|_*\|x_{t+1}-x_{t}\|\\
		& \geq  (\tfrac{1}{4}\|x_{t+1}-x_t\|^2 - \gamma_t\bar{L}\|x_t-x_{t+1}\|\|x_{t+1}-x\| ) + (\tfrac{1}{4}\|x_{t+1}-x_t\|^2 -\gamma_t\|\tilde{F}(x_t,\xi_t)-F(x_{t}) \|_*\|x_{t+1}-x_{t}\|)\\
		& \geq  -2\gamma_t^2\bar L^2 V(x_{t+1},x)-\gamma_t^2\|\tilde{F}(x_t,\xi_t)-F(x_{t}) \|_*^2. 
		\end{align*}
	The desired result is obtained  from \eqnok{Delta_V_bound} by rearranging  terms.
\end{proof}

We now consider stochastic GSMVIs  which satisfy \eqnok{G_monotone1} for some $\mu > 0$.
\begin{theorem}\label{TD_convergence0}
	Let \textit{Assumptions A, B, C} and \textit{D} hold. Assume that \eqnok{G_monotone1} holds for some $\mu > 0$, 
	let $x^*$ be a solution of problem  \eqref{VIP}, and $\{\theta_t\} $ a sequence of nonnegative numbers.
	{\color{black}Let the time scale parameter $\tau$ satisfy a lower bound as in \eqnok{def_tau}.}
	  Suppose that the parameters $\{\g_t\}$, $\{\t_t\}$  satisfy for $t=1,..., k-1$
	\begin{align} \label{cond_10}
	B_{t+1} \leq A_t
	\end{align}
where 
	\begin{align} \label{ABD}
A_t &= \theta_t (1+2\mu\gamma_t-2\gamma_t^2\bar{L}^2),\nn\\
B_t & = \left \{
\begin{array}{ll}
\theta_1(1+\gamma_1^2E)+\sum\limits_{t'=1}^{\tau+1}2\theta_{t'}\gamma_{t'}C\rho^{t'-1}+\sum\limits_{t'=2}^{\tau+1}3 \theta_{t'}\gamma_{t'}\gamma_1 \bar{L}^2+\sum\limits_{t'=2}^{\tau+1} \sum\limits_{t''=1}^{t'-1}4\theta_{t'}\gamma_{t'}\gamma_{t''}\bar{L}^2,
& t=1\\
\theta_t(1+\sum\limits_{t' = 1}^{t}\gamma_t\gamma_{t'}E)  +2\theta_{t+\tau}\gamma_{t+\tau}C\rho^\tau+ \sum\limits_{t'=t+1}^{t+\tau}3 \theta_{t'}\gamma_{t'}\gamma_t \bar{L}^2+ \sum\limits_{t'=t}^{t+\tau-1}4\theta_{t+\tau}\gamma_{t+\tau}\gamma_{t'}\bar{L}^2,                    & t \in [2,\tau]\\
\theta_t(1+\sum\limits_{t' = t-\tau}^{t}\gamma_t\gamma_{t'}E) +2\theta_{t+\tau}\gamma_{t+\tau}C\rho^\tau+ \sum\limits_{t'=t+1}^{t+\tau}3 \theta_{t'}\gamma_{t'}\gamma_t \bar{L}^2+ \sum\limits_{t'=t}^{t+\tau-1}4\theta_{t+\tau}\gamma_{t+\tau}\gamma_{t'}\bar{L}^2, & t\in (\tau, k-\tau]\\
\theta_t(1+\tsum_{t' = t-\tau}^{t}\gamma_t\gamma_{t'}E) + \tsum_{t'=t+1}^{k}3 \theta_{t'}\gamma_{t'}\gamma_t \bar{L}^2, & t \in (k-\tau, k],
\end{array}
\right.\nn\\
D& = \tsum_{t=1}^{k-\tau} \tsum_{t'=t}^{t+\tau}\theta_{t'}\gamma_{t'}\gamma_t+
\tsum_{t=k-\tau+1}^{k} \tsum_{t'=t}^{k}\theta_{t'}\gamma_{t'}\gamma_t,\quad E = 16\bar{L}^2+2\varsigma^2.
\end{align}	
	Then for all $k \geq \tau $,
	$$
	A_k\bbe[V(x_{k+1},x^*)]\leq B_1V(x_1,x^*)+D(4\sigma^2+3\|F(x^*)\|^2_*).
	$$
\end{theorem}

\begin{proof}
	Let us fix $x=x^*$ and take expectation on both sides of Proposition \ref{lemma_dist1},
	\begin{align*}
&	(1-2\gamma_t^2\bar{L}^2)\bbe[V(x_{t+1},x^*)]+\gamma_t\bbe[\langle F(x_{t+1}),x_{t+1}-x^*\rangle]+ \gamma_t\bbe[\langle\tilde{F}(x_{t-\kappa},\xi_t)-F(x_{t-\kappa}) , x_{t-\kappa}-x^*\rangle ]
	\\   & \leq \bbe[V(x_{t},x^*) ]+\gamma_t^2\bbe[\|\tilde{F}(x_t,\xi_t)-F(x_{t}) \|_*^2]+ \tsum_{t' = t-\kappa}^{t-1}\gamma_t\gamma_{t'}(\bbe[\|\tilde{F}(x_t,\xi_t)-F(x_{t}) \|_*^2] \\
	& \quad + \tfrac{3\bar{L}^2}{2}\bbe[\|x_{t'}-x^*\|^2]+\tfrac{3}{2}\bbe[\|\tilde{F}(x^*,\xi_{t'})\|_*^2]+2\bar{L}^2\bbe[\|x_{t-\kappa}-x^*\|^2]).
	\end{align*}
	Invoking the generalized strong monotonicity condition \eqnok{G_monotone1} and Young's inequality gives
	\begin{align*}
& 	(1+2\mu\gamma_t-2\gamma_t^2\bar{L}^2)\bbe[V(x_{t+1},x^*)]+ \gamma_t\bbe[\langle\tilde{F}(x_{t-\kappa},\xi_t)-F(x_{t-\kappa}) , x_{t-\kappa}-x^*\rangle ]
	\\
&	\leq \bbe[V(x_{t},x^*) ]+\gamma_t^2\bbe[\|\tilde{F}(x_t,\xi_t)-F(x_{t}) \|_*^2]+ \tsum_{t' = t-\kappa}^{t-1}\gamma_t\gamma_{t'}\big(\bbe[\|\tilde{F}(x_t,\xi_t)-F(x_{t}) \|_*^2] \\
& \quad + \tfrac{3\bar{L}^2}{2}\bbe[\|x_{t'}-x^*\|^2]+3\bbe[\|\tilde{F}(x^*,\xi_{t'})-F(x^*)\|_*^2]+3\|F(x^*)\|^2_*+2\bar{L}^2\bbe[\|x_{t-\kappa}-x^*\|^2]\big).
	\end{align*}
By taking into account \eqnok{bound_delta_inner1} and \eqnok{bound_delta_norm0}, we obtain
{\color{black}
	\begin{align*}
&	(1+2\mu\gamma_t-2\gamma_t^2\bar{L}^2)\bbe[V(x_{t+1},x^*)]- \gamma_tC\rho^\tau\bbe[\|x_{t-\tau}-x^*\|^2]e
	\\
	& \leq \bbe[V(x_{t},x^*) ]+\gamma_t^2\big(\sigma^2+8 \bar{L}^2\bbe[\|x_t-x^*\|^2]+\varsigma^2 \bbe[\|x_t-x^*\|^2]\big)+ \tsum_{t' = t-\kappa}^{t-1}\gamma_t\gamma_{t'}\big(\sigma^2+8 \bar{L}^2\bbe[\|x_t-x^*\|^2] \\
	& \quad +\varsigma^2 \bbe[\|x_t-x^*\|^2]+ \tfrac{3\bar{L}^2}{2}\bbe[\|x_{t'}-x^*\|^2]+3\sigma^2+3\|F(x^*)\|^2_*+2\bar{L}^2\bbe[\|x_{t-\kappa}-x^*\|^2]\big).
	\end{align*}}
	Using \eqnok{strong_convex_V}, we obtain
	\begin{align}\label{ineq1}
	(1+2\mu\gamma_t-2\gamma_t^2\bar{L}^2)\bbe[V(x_{t+1},x^*)]\leq
	\big(1+\tsum_{t' = t-\kappa}^{t}\gamma_t\gamma_{t'}(16\bar{L}^2+2\varsigma^2)\big)\bbe[V(x_{t},x^*) ]+ 2\gamma_tC\rho^\kappa\bbe[V(x_{t-\kappa},x^*)]\nn\\
	+ \tsum_{t' = t-\kappa}^{t-1}\gamma_t\gamma_{t'}(  3\bar{L}^2\bbe[V(x_{t'},x^*)]+4\bar{L}^2\bbe[V(x_{t-\kappa},x^*)]) + \tsum_{t' = t-\kappa}^{t}\gamma_t\gamma_{t'}(4\sigma^2+3\|F(x^*)\|_*^2).
	\end{align}
	{\color{black} For $t \geq \tau+1$, we set $\kappa=\tau$ in \eqref{ineq1}. As for the case of $t \leq \tau$, we set $\kappa = t-1$ in \eqref{ineq1}. Multiplying  both sides  with $\theta_t$ and
		subsequently  summing up from $t=1$ to $k$, we get}

	\begin{align}\label{ineq3}
	\tsum_{t=1}^k A_t\bbe[V(x_{t+1},x^*)] \leq \tsum_{t=1}^k B_t\bbe[V(x_{t},x^*)]+D(4\sigma^2+3\|F(x^*)\|^2_*),
	\end{align}
	where $A_t$, $B_t$ and $D$ are defined at \eqnok{ABD}.
Finally by	invoking the condition \eqnok{cond_10} we obtain the result.
\end{proof}

All our algorithms involve a time scale parameter $ \tau $, a lower bound of which is defined as
\begin{align}\label{def_tau}
\underline{\tau} = \lceil \tfrac{\log{(1/\mu)}+\log(9C)}{\log{(1/\rho)}} \rceil.
\end{align}
We now specify the selection of a particular
stepsize policy for solving stochastic GSMVIs. 
\begin{corollary}\label{step_size_TD1}
	Assume that \eqnok{G_monotone1} holds for some $\mu > 0$, let $ \underline{\tau} $ be defined in \eqnok{def_tau}. If
	 $\tau \geq \underline{\tau}$
	and
	$$
	{\color{black}t_0 = \tfrac{(\tau+1)(184\bar L^2+16\varsigma^2)}{3\mu^2}}, ~~~ \gamma_t= \tfrac{2}{\mu ( t_0 + t-1) },~~~ \theta_t  = (t + t_0)(t+t_0+1),
	$$
	then
	\begin{align}\label{corr_TD1}
	\bbe[V(x_{k+1},x^*)] 
	\leq \tfrac{MV(x_{1},x^*)}{ (k + t_0)(k+t_0+1)}  +  \tfrac{20k(\tau+1) \sigma^2}{\mu^2  (k + t_0)(k+t_0+1)}+\tfrac{15k(\tau+1) \|F(x^*)\|_*^2}{\mu^2  (k + t_0)(k+t_0+1)},
	\end{align}
	where
$
	M =t_0(t_0+1) +\tfrac{4C(t_0+\tau+4)}{\mu(1-\rho)}+ \tfrac{3\tau(\tau+2)\bar{L}^2}{\mu^2}.
$
\end{corollary}

\begin{proof}
	In order to check \eqnok{cond_10}, we observe that
	\begin{align*}
	A_t-B_{t+1} & \geq\theta_t (1+2\mu\gamma_t-2\gamma_t^2\bar{L}^2) - \theta_{t+1}\big(1+\tau\gamma_{t+1}\gamma_{t-\tau+1}(16\bar{L}^2+2\varsigma^2)\big) - \theta_{t+\tau+1}\gamma_{t+\tau+1} (2C\rho^\tau + 7 \tau \gamma_t \bar{L}^2)\\
	&\geq \theta_t( 1+2\mu\gamma_t) - \theta_{t+1}(1+\tfrac{3}{4}\mu\gamma_t+\tfrac{1}{4}\mu\gamma_t) \geq (t+t_0+1)\big[ \tfrac{(t+t_0)(t+t_0+3)}{t+t_0-1} - \tfrac{(t+t_0+2)^2}{t+t_0}\big] \geq 0.
	\end{align*}
	Here, the first inequality follows from the definition of $A_t$ and $B_t$, the 
	\textcolor{black}{second and third inequality follow}
	 from the selection of $\rho^\tau$ and $t_0$.
	The result then follows from Theorem~\ref{TD_convergence0} and the following 
	simple calculations.
	\begin{align*}
	B_1 &= \theta_1\big(1+\gamma_1^2 (16\bar{L}^2+2\varsigma^2)\big)+\tsum_{t'=1}^{\tau+1}2\theta_{t'}\gamma_{t'}C\rho^{t'-1}+\tsum_{t'=2}^{\tau+1}3 \theta_{t'}\gamma_{t'}\gamma_1 \bar{L}^2+\tsum_{t'=2}^{\tau+1}\big( \tsum_{t''=1}^{t'-1}4\theta_{t'}\gamma_{t'}\gamma_{t''}\bar{L}^2\big)\\
	&\leq t_0(t_0+1) + \tfrac{4C(t_0+\tau+4)}{\mu(1-\rho)} + \tfrac{3\tau(\tau+2)\bar{L^2}}{\mu^2},\\
	D &= \tsum_{t=1}^{k-\tau} \big(\tsum_{t'=t}^{t+\tau}\theta_{t'}\gamma_{t'}\gamma_t\big)+
	\tsum_{t=k-\tau+1}^{k} \big(\tsum_{t'=t}^{k}\theta_{t'}\gamma_{t'}\gamma_t\big)\\	
	&\leq \tfrac{4k(\tau+1)(t_0+\tau)(t_0+\tau+1)}{\mu^2 (t_0+1)(t_0+\tau)} \leq \tfrac{5k(\tau+1)}{\mu^2}.
	\end{align*}
\end{proof}

In view of Corollary \ref{step_size_TD1} the number of iterations performed by the temporal difference algorithm to find a solution $ \bar{x} \in X$ s.t. $ \mathbb{E}[V(\bar{x}, x^*)] \leq \epsilon $ is bounded by 
{\color{black}
$$
\mathcal{O}\{ \max\big( (\tfrac{\underline{\tau} (\bar{L}^2+\varsigma^2)}{\mu^2}+ 
\sqrt{\tfrac{\underline{\tau}  (\bar{L}^2+\varsigma^2)}{\mu^{3} (1-\rho)}}) 
\sqrt{\tfrac{V(x_1, x^*)}{\epsilon}} , \tfrac{\underline{\tau}}{\mu^2 \epsilon}(\sigma^2+\|F(x^*)\|^2_*)\big)   \}.
$$}
To our knowledge this is the first analysis of stochastic GSMVIs in the Markovian noise setting. When $F(x^*)=0$, the nature of the bound suggests the benefits of applying mini-batch to reduce $\sigma$ in terms of the resulting convergence rate.  
Note that when $\sigma = 0, \varsigma=0$ (the deterministic case) and  $F(x^*)= 0$, the convergence rate achieved in Corollary \ref{step_size_TD1} is not linear and hence
not optimal. Assuming that the total number of iterations $k$ is given in advance, we can select a novel stepsize policy that  improves this convergence rate.

\begin{corollary}
	\label{step_size_TD1_constant}
	Assume that \eqnok{G_monotone1} holds for some $\mu > 0$, let $ \underline{\tau} $ be defined in \eqnok{def_tau}. If
	$\tau \geq \underline{\tau}$
	and
	$$
	{\color{black}\gamma_t =\gamma= \min\{\tfrac{3\mu }{(\tau+1)(92\bar{L}^2+8\varsigma^2)}, }\tfrac{q\log k}{\mu k}\},~~~ \theta_t  =\theta^t= (\mu \gamma+1)^{t},  
	$$
	where
	$
	q = 2(1+\tfrac{\log(2\mu^2MV(x_{1},x^*)/[(\tau+1)(4\sigma^2+3\|F(x^*)\|_*^2)])}{\log k}),
	$
	then
	\begin{align*}
	\bbe[V(x_{k+1},x^*)] 
	\leq \big(1+\tfrac{3\mu^2 }{(\tau+1)(92\bar{L}^2+8\varsigma^2)} \big)^{-k}MV(x_{1},x^*) + \tfrac{(1+2q\log k)(\tau+1)(4\sigma^2+3\|F(x^*)\|_*^2)}{2\mu^2 k},
	\end{align*}
	with $
	M = 1 + \tfrac{2[(\theta\rho)^\tau-1]\gamma C}{\theta\rho-1} + \tfrac{4\bar L^2(\theta^{\tau+1}-1)}{\mu^2}.
	$
\end{corollary}

\begin{proof}
	In order to check \eqnok{cond_10}, we observe that
	\begin{align*}
	A_t-B_{t+1} & \geq\theta_t (1+2\mu\gamma_t-2\gamma_t^2\bar{L}^2)- \theta_{t+1}(1+\tau\gamma_{t+1}\gamma_{t-\tau+1}\big(16\bar{L}^2+2\varsigma^2)\big) - \theta_{t+\tau+1}\gamma_{t+\tau+1}
	(2C\rho^\tau + 7 \tau\gamma_t \bar{L}^2)\\
	&\geq \theta_t\big[ \mu\gamma -\theta^{\tau+1}(\tau+1)\gamma^2(23\bar{L}^2+2\varsigma^2)-2\theta^{\tau+1} C\rho^\tau \gamma
	\big] \geq \theta_t\big[ \mu\gamma -\tfrac{3}{4}\mu\gamma-\tfrac{1}{4}\mu\gamma] \geq 0.
	\end{align*}
	The result then follows from Theorem~\ref{TD_convergence} and the following 
	calculations.
	\begin{align*}
	\tfrac{B_1}{ A_k}V(x_{1},x^*)&\leq (\mu\gamma+1)^{-k}MV(x_{k+1},x^*) \leq (1+\tfrac{3\mu^2 }{(\tau+1)(92\bar{L}^2+8\varsigma^2)})^{-k}MV(x_{1},x^*) + \tfrac{MV(x_{1},x^*)}{k^{\frac{q}{2}}}\\
	&=(1+\tfrac{3\mu^2 }{(\tau+1)(92\bar{L}^2+8\varsigma^2)})^{-k}MV(x_{1},x^*)  + \tfrac{(\tau+1)(4\sigma^2+3\|F(x^*)\|_*^2)}{2\mu^2k},\\
	\tfrac{D}{A_k} &
	\leq \tsum_{t=1}^k (\tau+1) \theta_t\gamma_t^2\theta_k^{-1}  \leq \tfrac{(\tau+1)\gamma (1+\mu\gamma)}{\mu} \leq \tfrac{2(\tau+1)q\log k}{\mu^2k}.
	\end{align*}
\end{proof}

In view of Corollary \ref{step_size_TD1_constant} the number of iterations performed by the temporal difference algorithm to find a solution $ \bar{x} \in X$ s.t. $ \mathbb{E}[V(\bar{x}, x^*)] \leq \epsilon $ is bounded by 
{\color{black}
$$
\mathcal{O}\{ \max( \tfrac{\underline{\tau} (\bar{L}^2+\varsigma^2)}{\mu^2}
\log{\tfrac{V(x_1, x^*)}{\epsilon(1-\rho)}} , \tfrac{\underline{\tau}(\sigma^2+\|F(x^*)\|_*^2)}{\mu^2 \epsilon} \log{\tfrac{1}{\epsilon}}) \}.
$$}
This complexity bound is nearly optimal, \textcolor{black} {up to a logarithmic factor, in terms of the dependence on $\epsilon$ for solving GSMVI problems in a Markovian setting.} Note that when 
$\sigma = 0$, $ \varsigma = 0 $ (i.e., the deterministic case) and  $F(x^*)= 0$, the convergence rate achieved in Corollary~\ref{step_size_TD1_constant} with a constant step-size is  linear.
Given the nature of the step-size under consideration, implementation of the algorithm requires estimation of $\sigma $  
as well as $\tau$. The latter realization provides motivation for our next algorithm applying updates in perioric intervals of length 
$ \tau$ while achieving several appealing properties.

\subsection{Conditional temporal difference algorithm}   

In the previous considerations the analysis was facilitated by \eqref{bound_delta_inner1} and
\eqref{bound_delta_norm1}. In the current setting we will incorporate the parameter $\tau$ in 
the algorithmic design.
We refer to the new algorithm as the conditional TD algorithm given that the parameter $\tau$ needs to satisfy a condition as exemplified in Corollary  \ref{step_size_TD_skipping}.
Between each two updates of the sequence of iterates $ \{ x_t \} $, one collects $\tau$  samples without updating $ \{ x_t \} $. 
The idea of collecting several samples before updating the iterate is not completely new.  It 
has been explored in practice as well as for theoretical purposes \cite{bresler2020squares}.
The authors in \cite{bresler2020squares}   associate the time interval between updates with the mixing time of the underlying Markov process, and hence this interval depends on the target accuracy. In our work $\tau$ is chosen as 
  a constant relating both to the convergence properties of the underlying Markov process, as well as the modulus of strong monotonicity. 
Furthermore we provide the first concrete  explanation of the algorithmic advantages of using the periodic updating scheme in terms of the improved convergence rate and relaxed assumptions for the analysis. Specifically 	\textcolor{black}{Assumption} A
is not needed for CTD  and the ensuing FTD algorithm.

\begin{algorithm}[H]  \caption{Conditional Temporal Difference Algorithm}  
	\label{alg:TD_skipping}
	\begin{algorithmic} 
		\STATE{Let $x_1 \in X$, and the nonnegative parameters $\{ \gamma_t\}$ be given. }
		\FOR{$ t = 1, \ldots, k$}
		\STATE { Collect $\tau$ state transition steps without updating the sequence of iterates $\{x_t\}$, denoted as $\{\xi_t^1, \xi_t^2, \dots, \xi_t^\tau\}$.}
		\beq \label{TD_skipping_step}
		x_{t+1} = \argmin_{x\in X} ~   \gamma_t \big\langle \tilde{F}(x_t,\xi_t^\tau)  , x \big\rangle + V(x_t, x).
		\eeq
		\ENDFOR
	\end{algorithmic}
\end{algorithm}

Henceforth for a given sequence of iterates $\{x_t\}$ and $\{\xi_t^\tau\}$ we will use the notation 
\beq \label{def_Delta_V}
\Delta F_t := F(x_t) - F(x_{t-1})
\ \ \mbox{and} \ \
\delta_t^\tau := \tilde{F}(x_t, \xi_t^\tau) - F(x_t),
\eeq
and $\delta_t^\tau$ denote the error associated with the computation of the operator $F(x_t)$. 

\begin{proposition}\label{lemma_dist}
	Let $\{x_t\}$ be generated by Algorithm \ref{alg:TD_skipping}.
 		\textcolor{black}{Then} for any $x \in X$,
	\begin{align*}
	\gamma_t \langle  F(x_{t+1}) , x_{t+1}-x\rangle + (1-2\gamma_t^2L^2)V(x_{t+1},x)+  \gamma_t\langle \delta_t^\tau,x_{t}-x \rangle\leq V(x_t,x)+\gamma_t^2  \|\delta_t^\tau\|_*^2.
	\end{align*}
\end{proposition}
\begin{proof} 
It follows from \eqnok{lemma_wo_projection} that
	\begin{eqnarray*} 
	\gamma_t \langle  F(x_t) , x_{t+1} - x \rangle +\gamma_t\langle \delta_t^\tau,x_{t+1}-x \rangle+ V(x_t, x_{t+1}) \leq V(x_t, x) - V(x_{t+1}, x), \forall x \in X.
	\end{eqnarray*}
	Adding and subtracting $\langle F(x_{t+1}), x_{t+1}-x \rangle$ on the left hand side, we obtain
	\begin{eqnarray} 
	\label{lemma_3}
	\gamma_t \langle  F(x_{t+1}) , x_{t+1} - x \rangle +D_t \leq V(x_t, x) - V(x_{t+1}, x),
	\end{eqnarray}
	in which
	\begin{eqnarray*} 
	D_t = \gamma_t \langle  \Delta F_t, x_{t+1} - x \rangle +\gamma_t\langle \delta_t^\tau,x_{t+1}-x \rangle+ V(x_t, x_{t+1}).
	\end{eqnarray*}
	Using \eqnok{strong_convex_V}  and the Lipschitz condition \eqref{Lipschitz} we can lower bound the term $D_t$ as follows:
	\begin{eqnarray*}
		D_t	&\geq&    \gamma_t \langle  \Delta F_t, x_{t+1} - x \rangle + \gamma_t\langle \delta_t^\tau,x_{t+1}-x_t \rangle + \gamma_t\langle \delta_t^\tau,x_{t}-x \rangle  + \tfrac{1}{2} \|x_{t+1}-x_t\|^2 \\
		& \geq &  -\gamma_tL\|x_{t+1}-x_t\|\|x_{t+1}-x\| +  \tfrac{1}{4} \|x_{t+1}-x_t\|^2 -\gamma_t\|\delta_t^\tau\|_*\|x_{t+1}-x_t\| +  \tfrac{1}{4} \|x_{t+1}-x_t\|^2 +  \gamma_t\langle \delta_t^\tau,x_{t}-x \rangle \\
		&\geq&    -\gamma_t^2 L^2\|x_{t+1}-x\|^2 - \gamma_t^2  \|\delta_t^\tau\|_*^2 + \gamma_t \langle \delta_t^\tau,x_{t}-x \rangle \geq    -2\gamma_t^2 L^2V(x_{t+1},x) - \gamma_t^2  \|\delta_t^\tau\|_*^2 + \gamma_t \langle \delta_t^\tau,x_{t}-x \rangle.
	\end{eqnarray*} 
	Here, the first inequality follows from \eqnok{strong_convex_V}, the second inequality follows from Cauchy-Schwarz inequality and \eqref{Lipschitz}, the third inequality follows from Young's inequality, the fourth inequality follows from \eqnok{strong_convex_V}.
	Then, rearranging the terms in \eqref{lemma_3}, we obtain the result.
	\end{proof}

Now we consider the generalized strongly monotone VIs which satisfy \eqnok{G_monotone1}
for some $\mu > 0$.

\begin{theorem}\label{TD_convergence}
\textcolor{black}{Let \textit{Assumptions C} and \textit{D} hold.} Assume that \eqnok{G_monotone1} holds for some $\mu > 0$.
	Let $x^*$ be a solution of problem  \eqref{VIP}, and $\{\theta_t\} $ a sequence of nonnegative numbers.  Suppose that the parameters $\{\g_t\}$, $\{\t_t\}$  satisfy
	\begin{align}
	 \theta_t(1+4\gamma_t^2  C^2\rho^{2\tau}+ 2\gamma_tC\rho^\tau+2\gamma_t^2\varsigma^2) \leq \theta_{t-1}(1+2\mu\g_{t-1}-2\gamma_{t-1}^2L^2), \label{eqn:4}
	\end{align}	
		Then for all $k \geq 1 $,
	$$
	\theta_k(1+2\mu\g_k-2\gamma_k^2L^2)\mathbb{E}[V(x_{k+1},x^*)]  \leq  \theta_1(1+4\gamma_1^2  C^2\rho^{2\tau}+ 2\gamma_1C\rho^\tau+2\gamma_1^2\varsigma^2)  V(x_1,x^*)+ \tsum_{t=1}^k \theta_t\gamma_t^2 \sigma^2.
	$$
\end{theorem}

\begin{proof}
	By Proposition~\ref{lemma_dist} and strong monotone condition \eqnok{G_monotone1}, we have
	\begin{eqnarray*}
		(1+2\mu\g_t-2\gamma_t^2L^2)V(x_{t+1},x)+ \gamma_t \langle \delta_t^\tau,x_{t}-x \rangle\leq V(x_t,x)+\gamma_t^2  \|\delta_t^\tau\|_*^2.
	\end{eqnarray*} 
	Let us fix $x=x^*$ and take expectation on both sides of the inequality, we obtain
	\begin{eqnarray*}
		(1+2\mu\g_t-2\gamma_t^2L^2)\mathbb{E}[V(x_{t+1},x^*)]\leq \mathbb{E}[V(x_t,x^*)]+\gamma_t^2  \mathbb{E}[\|\delta_t^\tau\|_*^2]+ \gamma_t \mathbb{E}[|\langle \delta_t^\tau,x_{t}-x^* \rangle|],
	\end{eqnarray*} 
	which together with \eqnok{bound_delta_inner1} and \eqnok{bound_delta_norm1}, we obtain
	{\color{black}
	\begin{eqnarray}
		(1+2\mu\g_t-2\gamma_t^2L^2)\mathbb{E}[V(x_{t+1},x^*)]&\leq& \mathbb{E}[V(x_t,x^*)]+(2\gamma_t^2  C^2\rho^{2\tau}+ \gamma_tC\rho^\tau+\gamma_t^2 \varsigma^2)\mathbb{E}[\|x_t-x^*\|^2] + \gamma_t^2 \sigma^2\nn\\
		&\leq& (1+4\gamma_t^2  C^2\rho^{2\tau}+ 2\gamma_tC\rho^\tau+2\gamma_t^2 \varsigma^2)\mathbb{E}[V(x_t,x^*)] + \gamma_t^2 \sigma^2, \label{converge_step}
	\end{eqnarray}}
the second inequality follows from \eqnok{strong_convex_V}. Multiplying $\theta_t$ on both sides of \eqnok{converge_step}  and summing up from $t=1$ to $k$, we obtain
	\begin{align*}
	\tsum_{t=1}^k \theta_t(1+2\mu\g_t-2\gamma_t^2L^2)\mathbb{E}[V(x_{t+1},x^*)]
	\leq      \tsum_{t=1}^k \{\theta_t(1+4\gamma_t^2  C^2\rho^{2\tau}+ 2\gamma_tC\rho^\tau+2\gamma_t^2 \varsigma^2)\mathbb{E}[V(x_t,x^*)] 
	+  \theta_t\gamma_t^2 \sigma^2\},
	\end{align*}
	which together with  \eqnok{eqn:4}   then imply the desired result.
\end{proof}

We now specify the selection of a particular
stepsize policy for solving generalized strongly monotone
VI problems with Markovian noise.
\begin{corollary}\label{step_size_TD_skipping}
	Assume that \eqnok{G_monotone1} holds for some $\mu > 0$, let $ \underline{\tau} $ be defined in \eqnok{def_tau}. If
	$\tau \geq \underline{\tau}$
	and
	$$
	{\color{black}t_0 = \max\{\tfrac{8L^2}{\mu^2},\tfrac{16\varsigma^2}{\mu^2}\},} ~~~ \gamma_t= \tfrac{2}{\mu ( t_0 + t-1) },~~~ \theta_t  = (t + t_0)(t+t_0+1),
	$$
	then
	\begin{align}\label{corr_CTD1}
	\bbe[V(x_{k+1},x^*)] 
	\leq \tfrac{2( t_0 + 1)(t_0+2) V(x_1,x^*)}{ (k + t_0)(k+t_0+1)}  +  \tfrac{6k \sigma^2}{\mu^2  (k + t_0)(k+t_0+1)}.
	\end{align}
\end{corollary}

\begin{proof}
	In order to check \eqnok{eqn:4}, we observe that
	\begin{align*}
	 &\theta_{t}(1+2\mu\g_{t}-2\gamma_{t}^2L^2) - \theta_{t+1}(1+4\gamma_{t+1}^2  C^2\rho^{2\tau}+ 2\gamma_{t+1}C\rho^\tau+2\gamma_{t+1}^2\varsigma^2) \\
	 \geq~& \theta_t(1+\tfrac{3}{2}\mu\gamma_t)-\theta_{t+1}(1+\tfrac{1}{2}\mu\gamma_{t+1})\\
	 =~&(t+t_0+1)(t+t_0+2)(\tfrac{t+t_0}{t+t_0-1}-\tfrac{t+t_0+1}{t+t_0}) \geq 0.
	\end{align*}
	The result then follows from Theorem~\ref{TD_convergence} and the following 
	simple calculations.
	\begin{align*}
	\theta_k(1+2\mu\g_k-2\gamma_k^2L^2) &\geq (k+t_0)(k+t_0+1), \\
	\theta_1(1+4\gamma_1^2  C^2\rho^{2\tau}+ 2\gamma_1C\rho^\tau+2\gamma_1^2\varsigma^2) &\leq 2\theta_1 = 2 (t_0+1)(t_0+2) \\
	\tsum_{t=1}^{k} \theta_t \gamma_t^2 &\le 	\tsum_{t=1}^{k}\tfrac{6}{\mu^2} = \tfrac{6k}{\mu^2}.
	\end{align*}
\end{proof}

In view of Corollary \ref{step_size_TD_skipping} the number of iterations performed by the conditional temporal difference algorithm to find a solution $ \bar{x} \in X$ s.t. $ \mathbb{E}[V(\bar{x}, x^*)] \leq \epsilon $ is bounded by 
{\color{black}
$$
\mathcal{O}\{ \max( \tfrac{\underline{\tau} (L^2+\varsigma^2)}{\mu^2}
\sqrt{\tfrac{V(x_1, x^*)}{\epsilon}} , \tfrac{\underline{\tau} \sigma^2}{\mu^2 \epsilon}  )\}.
$$}
We observe that the conditional temporal difference method with diminishing step-size improves \textcolor{black}{in terms of} the  convergence rate 
achieved in Corollary \ref{step_size_TD1} in particular when $ \rho$ is close to 1.

Assuming that the total number of iterations $k$ is given in advance, we can select a stepsize policy that  improves the convergence rate achieved in Corollary \ref{step_size_stoch_strong_mon_k_unknown}.

\begin{corollary}\label{step_size_TD_skipping_constant}
	Assume that \eqnok{G_monotone1} holds for some $\mu > 0$, let $ \underline{\tau} $ be defined in \eqnok{def_tau}. If
	$\tau \geq \underline{\tau}$
	and
	$$
	{\color{black}\gamma_t =\gamma= \min\{\tfrac{\mu}{6L^2},\tfrac{\mu}{8\varsigma^2}},\tfrac{q\log k}{\mu k}\},~~~ \theta_t  = (\mu \gamma+1)^{t}, 
	$$ 
	in which
	$
	q = 2(1+\tfrac{\log(\mu^2V(x_1,x^*)/\sigma^2)}{\log k}),
	$
	then
	\begin{align*}
	\bbe[V(x_{k+1},x^*)] 
	\leq 2\big(1+\tfrac{\mu^2}{6L^2} \big)^{-k}V(x_1,x^*) + \tfrac{(1+4\log k+4\log\tfrac{\mu^2V(x_1,x^*)}{\sigma^2})\sigma^2}{\mu^2 k}.
	\end{align*}
\end{corollary}

\begin{proof}
	In order to check \eqnok{eqn:4}, we observe that
	\begin{align*}
	&\theta_{t}(1+2\mu\g_{t}-2\gamma_{t}^2L^2) - \theta_{t+1}(1+4\gamma_{t+1}^2  C^2\rho^{2\tau}+ 2\gamma_{t+1}C\rho^\tau+2\gamma_{t+1}^2\varsigma^2) \\
	\geq~&\theta_t\big[1+2\mu\g-2\gamma^2L^2 - (1+\mu\gamma)(1+\tfrac{1}{4}\gamma^2  \mu^2+ \tfrac{1}{2}\gamma\mu+ \tfrac{1}{4}\gamma\mu)\big]\\
	\geq~ & \theta_t\big[ \tfrac{1}{2}\mu\gamma-\tfrac{3}{4}\mu^2\gamma^2-2\gamma^2L^2-\tfrac{1}{4}\mu^3 \gamma^3\big]\geq 0.
	\end{align*}
	The result then follows from Theorem~\ref{TD_convergence} and the following 
	calculations,
	\begin{align*}
	\tfrac{\theta_1(1+2\gamma_1^2  L^2\rho^{2\tau}+ 2\gamma_1C\rho^\tau+2\gamma_1^2\varsigma^2) }{ \theta_k(1+2\mu\g_k-2\gamma_k^2L^2)}V(x_1,x^*)&\leq (\mu\gamma+1)^{-k}V(x_1,x^*) \leq (1+\tfrac{\mu^2}{6L^2})^{-k}V(x_1,x^*) + \tfrac{V(x_1,x^*)}{k^{\frac{q}{2}}}\\
	&=(1+\tfrac{\mu^2}{6L^2})^{-k}V(x_1,x^*) + \tfrac{\sigma^2}{\mu^2k}\\
	\tfrac{1}{ \theta_k(1+2\mu\g_k-2\gamma_k^2L^2)}\tsum_{t=1}^k \theta_t\gamma_t^2 \sigma^2 &
	\leq \tfrac{\gamma \sigma^2}{\mu} \leq \tfrac{\sigma^2q\log k}{\mu^2k}.
	\end{align*}
\end{proof}

In view of Corollary \ref{step_size_TD_skipping_constant} the number of iterations performed by the conditional temporal difference algorithm to find a solution $ \bar{x} \in X$ s.t. $ \mathbb{E}[V(\bar{x}, x^*)] \leq \epsilon $ is bounded by 
{\color{black}
$$
\mathcal{O}\{ \max( \tfrac{\underline{\tau} (L^2+\varsigma^2)}{\mu^2}
\log{\tfrac{V(x_1, x^*)}{\epsilon}} , \tfrac{\underline{\tau} \sigma^2}{\mu^2 \epsilon} \log{\tfrac{1}{\epsilon}} ) \}.
$$}
The second term of this complexity bound is nearly optimal, up to a logarithmic factor, for solving GSMVI problems under Markovian noise. 
In order to improve this complexity bound,
we develop a more advanced stepsize policy obtained by properly resetting the iteration
index to zero for the stepsize policy in Corollary~\ref{step_size_TD_skipping}.
In particular, the CTD iterations will be grouped into epochs indexed by $s$, and each epoch contains
$k_s$ iterations.
A local iteration index $\tilde t$, which is set to $0$ whenever a new epoch starts,
will take place of $t$ in the definitions of $\gamma_t$ and $\theta_t$ in Corollary~\ref{step_size_TD_skipping}.

\begin{corollary}\label{restart_stepsize}
Assume $\mu > 0$, let $ \underline{\tau} $ be defined in \eqnok{def_tau}. Set
$\tau \geq \underline{\tau}$,
 {\color{black}$t_0 = \max\{\tfrac{8L^2}{\mu^2},\tfrac{16\varsigma^2}{\mu^2}\}$}, and let
	$$k_s = \max\{(2\sqrt{2}-1)t_0+4, ~\tfrac{3\cdot2^{s+2}\sigma^2}{\mu^2 V(x_1,x^*)}\} ,~s\in \mathbb{Z}^+, ~~ K_0=0, ~~\mbox{and}~~K_s=\tsum_{s'=1}^{s}k_{s'}.
	$$
	For $t=1,2,...,$ introduce the epoch index $\tilde s$ and local iteration index $\tilde t$ such that
	$$\tilde s = \argmax_{\{s\in \mathbb{Z}^+\}}\mathbbm{1}_{\{K_{s-1}<t\le K_s\}}~~\mbox{and}~~
	\tilde t:=t-K_{\tilde s-1}.$$
	For the stepsize policy
	$$\gamma_t = \tfrac{2}{\mu(t_0+\tilde t-1)},~~~\theta_t = (\tilde t+t_0)(\tilde t+t_0+1),$$
	it holds that
	$
	\bbe[V(x_{K_s+1},x^*)]\leq 2^{-s}V(x_1,x^*)
	$ for any $s \ge 1$.
\end{corollary}
\begin{proof}
	First we note  that in each epoch we use the stepsize policy of Corollary \ref{step_size_TD_skipping}. This enables us to infer that, for $s=1,2,\ldots,$
	\begin{align} \label{convergence_epoch}
	\bbe[V(x_{K_{s}+1},x^*)] 
	\leq \tfrac{2( t_0 + 1)( t_0 + 2)    \bbe[V(x_{K_{s-1}+1},x^*)]}{ (k_s + t_0 )(k_s + t_0 +1)}  +  \tfrac{6k_s \sigma^2}{\mu^2  (k_s + t_0)(k_s + t_0+1)}.
	\end{align}
	As such by taking the specification  of $k_s$ into account, we obtain
	\begin{align*}
	\bbe[V(x_{K_{1}+1},x^*)] 
	\leq &~\tfrac{2( t_0 + 1)  (t_0+2) V(x_{1},x^*)}{ (k_1 + t_0 + 1)  (k_1 + t_0)}  +  \tfrac{6(k_1+1) \sigma^2}{\mu^2  (k_1 + t_0 + 1)  (k_1 + t_0)}\\
	\leq &~ \tfrac{2( t_0 + 1)  (t_0+2) V(x_{1},x^*)}{ (2\sqrt{2}t_0 + 5)  (2\sqrt{2}t_0+4)} +  \tfrac{6 \sigma^2}{\mu^2  (k_1 + t_0 + 1) }
	\leq \tfrac{V(x_1,x^*)}{2}.
	\end{align*}
	The desired convergence result follows  by recursively using 	
	\eqref{convergence_epoch}.
\end{proof}

In view of corollary \ref{restart_stepsize},
the number of epochs performed by the CTD method to find a solution $\bar x \in X$ s.t. $\mathbb{E}[V(\bar x,x^*)] \le \epsilon$ is bounded by $\log_2 (V(x_1, x^*)/\epsilon)$. 
Then together with the length of each epoch, the total complexity is bounded by
{\color{black}
\beq \label{bound_CTD}
{\cal O}\{ \max( \tfrac{\underline{\tau} (L^2+\varsigma^2) }{\mu^2} \log\tfrac{V(x_1,x^*)}{\epsilon},  \tfrac{\underline{\tau} \sigma^2 }{\mu^2 \epsilon} ) \}.
\eeq}

\subsection{Fast temporal difference algorithm}   

We consider an accelerated Temporal Difference Learning algorithm with operator extrapolation. Between each two updates of the sequence of iterates $ \{ x_t \} $, it collects $\tau$ Markovian samples without updating $ \{ x_t \} $.
\textcolor{black}{We will distinguish two cases, namely whether the feasible region $X$ is bounded or not. In the latter case the \textcolor{black}{projecting} region at each instant of time  is potentially time-varying and will be denoted henceforth as $X_t$. We always ensure that the particular choice of $X_t$ should contain $X^*$. }


\begin{algorithm}[H]  \caption{Fast Temporal Difference (FTD) Algorithm}  
	\label{alg:FastTD_skipping}
	\begin{algorithmic} 
		\STATE{{\color{black}Let $x_0 = x_1 \in X_1$}, and the nonnegative parameters $\{ \gamma_t\}$ and $\{\lambda_t\}$ be given. }
		\FOR{$ t = 1, \ldots, k$}
		\STATE { Collect $\tau$ state transition steps without updating the sequence of iterates $\{x_t\}$, denote as $\{\xi_t^1, \xi_t^2, \dots, \xi_t^\tau\}$.}
		\beq \label{FastTD_skipping_step}
		x_{t+1} = \argmin_{x\in {\color{black}X_{t+1}} }~   \gamma_t \big\langle \tilde{F}(x_t,\xi_t^\tau) + \lambda_t[\tilde{F}(x_t,\xi_t^\tau) -\tilde{F}(x_{t-1},\xi_{t-1}^\tau) ] , x \big\rangle + V(x_t, x).
		\eeq
		\ENDFOR
	\end{algorithmic}
\end{algorithm} 

Similar to Lemma~\ref{lemma_wo_projection},
we characterize the optimality condition of \eqnok{FastTD_skipping_step} as follows.
\begin{lemma} \label{lemma_projection_FastTD}
	Let $x_{t+1}$ be defined in \eqnok{FastTD_skipping_step}. Then we have \textcolor{black}{$\forall x \in X_{t+1}$},
	\beq \label{opt_FastTD_step}
	\gamma_t \langle  \tilde F(x_t,\xi_t^\tau) + \lambda_t \big( \tilde F(x_t,\xi_t^\tau) - \tilde F(x_{t-1},\xi_{t-1}^\tau) \big) , x_{t+1} - x \rangle + V(x_t, x_{t+1}) \leq V(x_t, x) - V(x_{t+1}, x)
	\eeq
\end{lemma}
What follows is the counterpart of proposition \ref{lemma_dist} for the FTD algorithm.
\begin{proposition} \label{prop_FastTD}
	Let $\{x_t\}$ be generated by the FTD method and $\{\theta_t\} $ be a sequence of nonnegative numbers.
	If the parameters in this method satisfy 
	\begin{align}
	\theta_{t-1}\gamma_{t-1} &= \theta_{t}\gamma_{t}\lambda_{t},  \label{theta_t}\\
	\theta_{t-1} &\geq 16 L^2 \gamma_t^2 \lambda_t^2 \theta_t, \label{split_8}
	\end{align}
	for all $t = 2, \ldots, k$, then \textcolor{black}{for any $x \in \bigcap_{t=1}^{k+1} X_t$,}
	\begin{align}
	\tsum_{t=1}^k   \theta_t \big[ \gamma_t \langle  F(x_{t+1}) , x_{t+1} - x \rangle + V(x_{t+1},x)\big] - 4 L^2 \theta_k \gamma_k^2 \| x_{k+1} - x \|^2 + \tsum_{t=1}^k \theta_t\gamma_t\lambda_t\langle \delta_t^\tau,x_{t}-x \rangle \nn\\
	\leq \tsum_{t=1}^k \theta_t V(x_t,x) +  \tsum_{t=1}^k \left( \theta_t \gamma_t^2 \lambda_t^2 \|\delta_t^\tau - \delta_{t-1} ^\tau\|_*^2 \right)+4\theta_k\gamma_k^2 \|\delta_k^\tau\|_*^2  - \theta_k\gamma_k  \langle \delta_k^\tau , x_{k} - x \rangle. \label{prop_general_FastTD_result}
	\end{align}
\end{proposition}

\begin{proof}
	It follows from \eqnok{opt_FastTD_step} after multiplying with $ \theta_t$ that $\forall x \in X_{t+1}$
	\begin{align}
	\theta_t [V(x_t, x) - V(x_{t+1}, x)] & \ge \theta_t\gamma_t \langle  \tilde F(x_t,\xi_t^\tau) + \lambda_t \big( \tilde F(x_t,\xi_t^\tau) - \tilde F(x_{t-1},\xi_{t-1}^\tau) \big) , x_{t+1} - x \rangle +\theta_t V(x_t, x_{t+1})
	\nn\\
	&= \theta_t\gamma_t \langle F(x_{t+1}), x_{t+1} - x \rangle  -
	\theta_t \gamma_t \langle \Delta F_{t+1}, x_{t+1} - x \rangle+ \theta_t \gamma_t \lambda_t \langle \Delta  F_t, x_{t} - x \rangle \nn\\
	& \quad + \theta_t \gamma_t \lambda_t \langle  \Delta F_t, x_{t+1} - x_t \rangle  +\theta_t V(x_{t+1},x_t)+ \Delta_t, \label{3point_FastTD}
	\end{align}
	in which, 
	\begin{align*}
	\Delta_t &= \theta_t\gamma_t(1+\lambda_t) \langle \delta_t^\tau,x_{t+1}-x \rangle - \theta_t\gamma_t\lambda_t \langle \delta_{t-1}^\tau,x_{t+1}-x \rangle\\
	&= \theta_t\gamma_t\langle \delta_t^\tau,x_{t+1}-x \rangle - \theta_t\gamma_t\lambda_t \langle \delta_{t-1}^\tau,x_{t}-x \rangle \\
	& \quad + \theta_t\gamma_t\lambda_t\langle \delta_t^\tau-\delta_{t-1}^\tau,x_{t+1}-x_t \rangle+ \theta_t\gamma_t\lambda_t\langle \delta_t^\tau,x_{t}-x \rangle.
	\end{align*}
	Summing up \eqref{3point_FastTD}  from $t =  1$ to $k $, invoking \eqref{theta_t} and $x_1 = x_0$,  and assuming also $ \delta_1^\tau = \delta_0^\tau$  we obtain \textcolor{black}{$\forall x \in \bigcap_{t=1}^{k+1} X_t$},
	\begin{align}
	\tsum_{t=1}^k    \theta_t [V(x_t, x) - V(x_{t+1}, x)]   & \geq  \tsum_{t=1}^k \big[ \theta_t \gamma_t \langle  F(x_{t+1}) , x_{t+1} - x \rangle \big] + \tsum_{t=1}^k \theta_t\gamma_t\lambda_t\langle \delta_t^\tau,x_{t}-x \rangle\nn \\
	 & \quad - \theta_k \gamma_k \langle \Delta  F_{k+1} , x_{k+1} - x \rangle+ \theta_k \gamma_k \langle \delta_k^\tau , x_{k+1} - x \rangle   +  \tilde Q, \label{eqn:common_bnd_VI_FastTD}
	\end{align}	
	with
	\begin{align}
	\tilde Q &:= \tsum_{t=1}^k  
	\left[ \theta_t \gamma_t \lambda_t \langle\Delta  F_t, x_{t+1} - x_t \rangle +\theta_t \gamma_t \lambda_t \langle\delta_t^\tau-\delta_{t-1}^\tau, x_{t+1} - x_t \rangle + \theta_t V(x_t, x_{t+1}) \right]. \label{def_tildeQ_FastTD}
	\end{align}
Using  the Lipschitz condition \eqref{Lipschitz} and $x_1 = x_0$, we can lower bound the term $\tilde Q_t$ as follows
\begin{align*}
	\tilde Q
	&\geq   \tsum_{t=1}^k    \left[ -   \theta_t \gamma_t \lambda_t L \|   x_t - x_{t-1} \| 
	\|   x_{t+1} - x_t \|  +   \theta_t V(x_t, x_{t+1}) +   \theta_t \gamma_t \lambda_t \langle\delta_t^\tau - \delta_{t-1}^\tau, x_{t+1} - x_t \rangle\right] \\
	& \geq   \tsum_{t=1}^k  \left[-    \theta_t \gamma_t \lambda_t L \|   x_t - x_{t-1} \| 
	\|   x_{t+1} - x_t \|  +   \tfrac{\theta_t}{8} \|  x_t - x_{t+1} \|^2 +   \tfrac{\theta_{t-1} }{8} \|  x_t - x_{t-1} \|^2  \right]  +       \tfrac{ \theta_k}{8} \|  x_k - x_{k+1} \|^2  \\
	& \quad +   \tsum_{t=1}^k \left[ \theta_t \gamma_t \lambda_t \langle\delta_t^\tau - \delta_{t-1}^\tau, x_{t+1} - x_t \rangle + \tfrac{\theta_{t} }{4} \|  x_t - x_{t+1} \|^2  \right] \\
	&\geq    \tfrac{ \theta_k}{8} \|  x_k - x_{k+1} \|^2 + \tsum_{t=1}^k \left[ \theta_t \gamma_t \lambda_t \langle\delta_t^\tau - \delta_{t-1}^\tau, x_{t+1} - x_t \rangle + \tfrac{\theta_{t} }{4} \|  x_t - x_{t+1} \|^2  \right]
	\\
	&\geq \tfrac{ \theta_k}{8} \|  x_k - x_{k+1} \|^2 -  \tsum_{t=1}^k \left( \theta_t \gamma_t^2 \lambda_t^2 \|\delta_t^\tau - \delta_{t-1} ^\tau\|_*^2 \right) ,
\end{align*} 
where the second inequality follows from \eqnok{strong_convex_V}, the third inequality follows from \eqnok{split_8} and the last one follows from the Young's inequality. Plugging the above bound of $\tilde Q$ in \eqnok{eqn:common_bnd_VI_FastTD}, and by using the notation
 $\Delta V_t(x)  =  V(x_t, x) - V(x_{t+1}, x)$, 
we obtain
\begin{align}
\tsum_{t=1}^k     \theta_t \Delta V_t(x) & \geq   \tsum_{t=1}^k \big[ \theta_t \gamma_t \langle  F(x_{t+1}) , x_{t+1} - x \rangle \big] + \tsum_{t=1}^k \theta_t\gamma_t\lambda_t\langle \delta_t^\tau,x_{t}-x \rangle+ \tfrac{ \theta_k}{8} \|  x_k - x_{k+1} \|^2\nn\\
& \quad - \theta_k \gamma_k \langle \Delta F_{k+1} , x_{k+1} - x \rangle + \theta_k \gamma_k \langle \delta_k^\tau , x_{k+1} - x \rangle  -   \tsum_{t=1}^k \left( \theta_t \gamma_t^2 \lambda_t^2 \|\delta_t^\tau - \delta_{t-1}^\tau \|_*^2 \right) \nn \\
&\quad \ge \tsum_{t=1}^k \big[ \theta_t \gamma_t \langle \tilde F(x_{t+1}) , x_{t+1} - x \rangle \big] + \tsum_{t=1}^k \theta_t\gamma_t\lambda_t\langle \delta_t^\tau,x_{t}-x \rangle+ \theta_k\gamma_k  \langle \delta_k^\tau , x_{k} - x \rangle\nn\\
&\quad - 4 L^2 \theta_k \gamma_k^2  \|  x - x_{k+1} \|^2  -  4\theta_k\gamma_k^2 \|\delta_k^\tau\|_*^2 -  \tsum_{t=1}^k \left( \theta_t \gamma_t^2 \lambda_t^2 \|\delta_t^\tau - \delta_{t-1} ^\tau\|_*^2 \right) , \label{FastTD_bound_general_tmp}
\end{align}
where the second inequality follows from
\begin{align*}
& \quad - \gamma_k \langle \Delta  F_{k+1} , x_{k+1} - x \rangle  + \gamma_k \langle \delta_k^\tau , x_{k+1} - x \rangle 
+\tfrac{1 }{8} \|  x_k - x_{k+1} \|^2 \\
& \geq 
-  \gamma_k L \|  x_k - x_{k+1} \|  \|  x - x_{k+1} \| - \gamma_k\|\delta_k^\tau\|_*\|  x_k - x_{k+1} \| +
\tfrac{1 }{8} \|  x_k - x_{k+1} \|^2 +  \gamma_k \langle \delta_k^\tau , x_k - x \rangle  \\
& \geq   - 4 L^2 \gamma_k^2  \|  x - x_{k+1} \|^2  -  4\gamma_k^2 \|\delta_k^\tau\|_*^2 +  \gamma_k \langle \delta_k ^\tau, x_k - x \rangle.
\end{align*}
\end{proof}

Now we describe the main convergence properties for stochastic generalized (strongly) monotone VIs,
i.e., when \eqnok{G_monotone1} holds for some $\mu > 0$.

\begin{theorem} \label{the_FastTD_GMVI}
	\textcolor{black}{Let \textit{Assumptions A, B} and \textit{C} hold.} Suppose that \eqnok{G_monotone1} holds for some $\mu > 0$.
	If the parameters in the FTD algorithm satisfy 
	\eqnok{theta_t}, \eqnok{split_8}, as well as, 
	\begin{align}
	\theta_t +2\theta_t\gamma_t\lambda_tC\rho^\tau+8 \theta_t\g_t^2\lambda_t^2C^2\rho^{2\tau}+8 \theta_{t+1}\g_{t+1}^2\lambda_{t+1}^2C^2\rho^{2\tau} &\leq \theta_{t-1} (2 \mu \gamma_{t-1} + 1), ~~t = 2, \ldots, k-1, \nn\\
	\theta_k +2\theta_k\gamma_k\lambda_kC\rho^\tau+2\theta_k\gamma_kC\rho^\tau+8 \theta_k\gamma_k^2\lambda_k^2C^2\rho^{2\tau}+16 \theta_{k}\gamma_{k}^2C^2\rho^{2\tau} &\leq \theta_{k-1} (2 \mu \gamma_{k-1} + 1),
	\label{gamma_strong_mon_FastTD} 
	\\
	16 L^2 \gamma_k^2 &\leq 1, \label{gamma_strong_mon_FastTD_k}
	\end{align}
	then 
	{\color{black}
	\begin{align*}
	\theta_k (2 \mu \gamma_k + \tfrac{1}{2}) \bbe[V(x_{k+1},x^*)]  
	 \leq (\theta_1 +2\theta_1\gamma_1\lambda_1C\rho^\tau+16 \theta_1\g_1^2\lambda_1^2C^2\rho^{2\tau}+8 \theta_{2}\g_2^2\lambda_{2}^2C^2\rho^{2\tau}) V(x_1,x^*) \\
	 \quad  + \tsum_{t=1}^k \theta_t \gamma_t^2 \lambda_t^2  \left(4\sigma^2 +2\varsigma^2\bbe[\|x_t-x^*\|^2] +2\varsigma^2\bbe[\|x_{t-1}-x^*\|^2]\right)
	+ 4 \theta_k  \gamma_k^2 (\sigma^2+\varsigma^2\bbe[\|x_k-x^*\|^2]).
	\end{align*}}
\end{theorem}

\begin{proof}
	Let us fix $x = x^*$ and take expectation on both sides of \eqnok{prop_general_FastTD_result} \textcolor{black}{with respect to the underlying measure, which is completely speficifed by the stochastic kernel of the Markov chain and its initial distribution,}
	\begin{align}
	&\tsum_{t=1}^k   \theta_t \big\{ \gamma_t \bbe[ \langle  F(x_{t+1}) , x_{t+1} - x^* \rangle]+  \gamma_t\lambda_t\bbe[\langle \delta_t^\tau,x_{t}-x^* \rangle] + \bbe[V(x_{t+1},x^*)]\big\} - 4 L^2 \theta_k \gamma_k^2\bbe[ \| x_{k+1} - x^* \|^2]  \nn\\
	&\leq \tsum_{t=1}^k \theta_t \bbe[V(x_t,x^*)] +  \tsum_{t=1}^k \left( \theta_t \gamma_t^2 \lambda_t^2\bbe[ \|\delta_t^\tau - \delta_{t-1}^\tau \|_*^2] \right) + 4\theta_k\gamma_k^2\|\delta_k^\tau\|_*^2-\theta_k\gamma_k\mathbb{E}[\langle\delta_k^\tau,x_k-x^*\rangle].
	\label{bnd_FastTD_temp}
	\end{align}
	Note that 
	\begin{align*}
	\bbe[\|\delta_t^\tau - \delta_{t-1} ^\tau\|_*^2] &\le 2 (\bbe[\|\delta_t^\tau\|_*^2] + \bbe[\|\delta_{t-1}^\tau\|_*^2]) ,
	\end{align*}
	together with   \eqnok{bound_delta_inner1} and \eqnok{bound_delta_norm1}, we obtain
	{\color{black}
	\begin{align}
& \tsum_{t=1}^k \big\{  \theta_t \big( \gamma_t \bbe[ \langle  F(x_{t+1}) , x_{t+1} - x^* \rangle] + \bbe[V(x_{t+1},x^*)]\big)\big\} - 4 L^2 \theta_k \gamma_k^2\bbe[ \| x_{k+1} - x^* \|^2]  \nn\\
&	\leq \tsum_{t=1}^k \theta_t \bbe[V(x_t,x^*)] +  \tsum_{t=1}^k \left( \theta_t \gamma_t^2 \lambda_t^2C^2\rho^{2\tau}(4\bbe[\|x_t-x^*\|^2]+4\bbe[\|x_{t-1}-x^*\|^2])\right)  \nn\\
&\quad+ \tsum_{t=1}^k  [\theta_t \gamma_t^2 \lambda_t^2 (4\sigma^2 +2\varsigma^2\bbe[\|x_t-x^*\|^2]+ 2\varsigma^2\bbe[\|x_{t-1}-x^*\|^2])]
	+ \tsum_{t=1}^k\theta_t\gamma_t\lambda_tC\rho^\tau\bbe[\|x_t-x^*\|^2]\nn\\
	&\quad+ 8\theta_k\gamma_k^2C^2\rho^{2\tau}\bbe[\|x_k-x^*\|^2]+4\theta_k\gamma_k^2\sigma^2+4\theta_k\gamma_k^2\varsigma^2\bbe[\|x_k-x^*\|^2]+\theta_k\gamma_kC\rho^\tau\bbe[\|x_k-x^*\|^2].
	\label{f}
	\end{align}}
	Invoking \eqnok{G_monotone1} and rearranging  terms, we have
	{\color{black}
	\begin{align*}
	&\tsum_{t=1}^k \big\{ \theta_t (2 \mu \gamma_t  + 1) \bbe[V(x_{t+1},x^*)]\big\}  - 4 L^2 \theta_k \gamma_k^2\bbe[ \| x_{k+1} - x^* \|^2]  \nn\\
	& \leq \tsum_{t=2}^{k-1} (\theta_t +2\theta_t\gamma_t\lambda_tC\rho^\tau+8 \theta_t\g_t^2\lambda_t^2C^2\rho^{2\tau}+8 \theta_{t+1}\g_{t+1}^2\lambda_{t+1}^2C^2\rho^{2\tau}) \bbe[V(x_t,x^*)] \nn\\
	& \quad + (\theta_k +2\theta_k\gamma_k\lambda_kC\rho^\tau+2\theta_k\gamma_kC\rho^\tau+8 \theta_k\gamma_k^2\lambda_k^2C^2\rho^{2\tau}+16 \theta_{k}\gamma_{k}^2C^2\rho^{2\tau}) \bbe[V(x_k,x^*)]\nn\\
	& \quad + (\theta_1 +2\theta_1\gamma_1\lambda_1C\rho^\tau+16 \theta_1\g_1^2\lambda_1^2C^2\rho^{2\tau}+8 \theta_{2}\g_2^2\lambda_{2}^2C^2\rho^{2\tau}) V(x_1,x^*) \nn\\
	 & \quad+\tsum_{t=1}^k \theta_t \gamma_t^2 \lambda_t^2  \left(4\sigma^2 +2\varsigma^2\bbe[\|x_t-x^*\|^2] +2\varsigma^2\bbe[\|x_{t-1}-x^*\|^2]\right)
	+ 4 \theta_k  \gamma_k^2 (\sigma^2+\varsigma^2\bbe[\|x_k-x^*\|^2]).
	\end{align*}}
	From  \eqnok{gamma_strong_mon_FastTD} and using \eqnok{strong_convex_V}, we have
	{\color{black}
	\begin{align*}
	&\theta_k (2 \mu \gamma_k + \tfrac{1}{2}) \bbe[V(x_{k+1},x^*)] 
	+ (\tfrac{1}{4} - 4 L^2 \gamma_k^2) \theta_k \bbe[ \| x_{k+1} - x^* \|^2] \\
	& \leq (\theta_1 +2\theta_1\gamma_1\lambda_1C\rho^\tau+16 \theta_1\g_1^2\lambda_1^2C^2\rho^{2\tau}+8 \theta_{2}\g_2^2\lambda_{2}^2C^2\rho^{2\tau}) V(x_1,x^*)\nn\\
	 & \quad+\tsum_{t=1}^k \theta_t \gamma_t^2 \lambda_t^2  \left(4\sigma^2 +2\varsigma^2\bbe[\|x_t-x^*\|^2] +2\varsigma^2\bbe[\|x_{t-1}-x^*\|^2]\right)
	+ 4 \theta_k  \gamma_k^2 (\sigma^2+\varsigma^2\bbe[\|x_k-x^*\|^2]),
	\end{align*}}
	which, in view of \eqnok{gamma_strong_mon_FastTD_k}, clearly implies the result.
\end{proof}

\subsubsection{Convergence result of FTD with bounded feasible region}\label{FTD_bounded}
{\color{black}
We first specify a particular stepsize selection when the feasible region is bounded. In this case we set $X_1=X, \hdots, X_k = X$, and denote $D_X^2 :=\max_{x,y\in X} \|x-y\|^2 <\infty$.}
\begin{corollary} \label{step_size_stoch_strong_mon_k_unknown}   \label{step_size_FastTD_skipping}
	Let $ \underline{\tau} $ be defined in \eqnok{def_tau} and \textcolor{black}{$\mu > 0 $}. If
	$\tau \geq \underline{\tau}$
	and
	$$
	t_0 =\tfrac{8L}{ \mu}, ~~~ \gamma_t= \tfrac{2}{\mu ( t+t_0 -1) },~~~ \theta_t  = (t + t_0)(t+t_0+1),~~\mbox{and}~~\lambda_t = \tfrac{\theta_{t-1}\gamma_{t-1}}{\theta_t\gamma_t},
	$$
	then
	\begin{align}\label{corr_FTD1}
	\bbe[V(x_{k+1},x^*)] 
	\leq \tfrac{2( t_0 + 1)(t_0+2) V(x_1,x^*)}{ (k + t_0)(k+t_0+1)}  +  \tfrac{40(k+1) (\sigma^2+\varsigma^2 D_X^2)}{\mu^2  (k + t_0)(k+t_0+1)}.
	\end{align}
\end{corollary}
\begin{proof}
	Note that \eqnok{theta_t} holds by the definition of $\lambda_t$.
	Observe that
	\begin{align*}
	16 L^2 \gamma_k^2  \leq  1, \quad\quad 16 L^2 \gamma_t^2 \lambda_t^2 \theta_t = 
	16 L^2 \tfrac{\theta_{t-1}^2 \gamma_{t-1}^2}{\theta_t} \le 16 L^2 \gamma_{t-1}^2 \theta_{t-1} \le \theta_{t-1},
	\end{align*}
	and thus both  \eqnok{split_8} and \eqnok{gamma_strong_mon_FastTD_k} hold.
	Also, \eqnok{gamma_strong_mon_FastTD} holds, because
	\begin{align*}
	&\theta_t +2\theta_t\gamma_t\lambda_tC\rho^\tau+8 \theta_t\g_t^2\lambda_t^2C^2\rho^{2\tau}+8 \theta_{t+1}\g_{t+1}^2\lambda_{t+1}^2C^2\rho^{2\tau} - \theta_{t-1} (2 \mu \gamma_{t-1} + 1) \\
	\leq&~ \theta_t - \theta_{t-1}(2\mu\gamma_{t-1}+1-2\gamma_{t-1}C\rho^\tau-16\gamma_{t-1}^2\tfrac{\theta_{t-1}}{\theta_t}C^2\rho^{2\tau})
	\leq~ \theta_t-\theta_{t-1}(\mu\gamma_{t-1}+1)\leq 0,~~~t = 2, \ldots, k-1,\\
	&\theta_k +2\theta_k\gamma_k\lambda_kC\rho^\tau+2\theta_k\gamma_kC\rho^\tau+8 \theta_k\gamma_k^2\lambda_k^2C^2\rho^{2\tau}+16 \theta_{k}\gamma_{k}^2C^2\rho^{2\tau} - \theta_{k-1} (2 \mu \gamma_{k-1} + 1)\\
	\leq&~ \theta_k - \theta_{k-1}(2\mu\gamma_{k-1}+1-2\gamma_{k-1}C\rho^\tau-24\gamma_{k-1}^2\tfrac{\theta_{k-1}}{\theta_k}C^2\rho^{2\tau})
	\leq~ \theta_k-\theta_{k-1}(\mu\gamma_{k-1}+1)\leq~  0.
	\end{align*}
	The result then follows from Theorem~\ref{the_FastTD_GMVI} and the following 
	simple calculations:
	\begin{align*}
	\tsum_{t=1}^k \theta_t\gamma_t^2\lambda_t^2 &= \tsum_{t=1}^k\tfrac{\theta_{t-1}^2\gamma_{t-1}^2}{\theta_t}
	=\tsum_{t=1}^k \tfrac{4(t+t_0-1)^2(t+t_0)}{\mu^2(t+t_0-2)^2(t+t_0+1)}\leq \tfrac{5k}{\mu^2},\\
	\theta_k\gamma_k^2 & = \tfrac{4(k+t_0+1)(k+t_0)}{\mu^2(k+t_0-1)^2}\leq \tfrac{5}{\mu^2}.
	\end{align*}
\end{proof}

In view of Corollary \ref{step_size_FastTD_skipping} the number of iterations performed by the FTD algorithm to find a solution $ \bar{x} \in X$ s.t. $ \mathbb{E}[V(\bar{x}, x^*)] \leq \epsilon $ is bounded by 
{\color{black}$$
\mathcal{O}\{ \max( \tfrac{\underline{\tau} L}{\mu}
\sqrt{\tfrac{V(x_1, x^*)}{\epsilon} }, \tfrac{\underline{\tau} (\sigma^2+\varsigma^2D_X^2)}{\mu^2 \epsilon}  ) \}
$$}

Assuming that the total number of iterations $k$ is given in advance, we can select a stepsize policy that  improves on the convergence rate achieved in Corollary \ref{step_size_stoch_strong_mon_k_unknown}.

\begin{corollary}\label{step_size_stoch_strong_mon_k_known}
	Assume that \eqnok{G_monotone1} holds for some $\mu > 0$, let $ \underline{\tau} $ be defined in \eqnok{def_tau}. If
	$\tau \geq \underline{\tau}$
	and
	$$
	\gamma_t =\gamma= \min\{\tfrac{1}{4L},\tfrac{q\log k}{\mu k}\},~~~ \theta_t  = (\tfrac{4\mu \gamma}{3}+1)^{t}, ~~~ \lambda_t=\tfrac{3}{4\mu\gamma+3},
	$$
	in which
	$
	q = \tfrac{3}{2}\big(1+\tfrac{\log(\mu^2V(x_1,x^*)/(\sigma^2+\varsigma^2 D_X^2))}{\log k}\big),
	$
	then
	\begin{align*}
	\bbe[V(x_{k+1},x^*)] 
	\leq 2\big(1+\tfrac{\mu}{3L} \big)^{-k}V(x_1,x^*) + \tfrac{(2+9\log k+9\log\tfrac{\mu^2V(x_1,x^*)}{\sigma^2+\varsigma^2 D_X^2})(\sigma^2+\varsigma^2 D_X^2)}{\mu^2 k} +\tfrac{4q^2\log k^2(\sigma^2+\varsigma^2 D_X^2)}{\mu^2k^2}.
	\end{align*}
\end{corollary}

\begin{proof}
	Note that \eqnok{theta_t} holds by the definition of $\lambda_t$.
	Observe that
	\begin{align*}
	16 L^2 \gamma_k^2  \leq  1, \quad \quad 16 L^2 \gamma_t^2 \lambda_t^2 \theta_t = 
	16 L^2 \tfrac{\theta_{t-1}^2 \gamma_{t-1}^2}{\theta_t} \le 16 L^2 \gamma_{t-1}^2 \theta_{t-1} \le \theta_{t-1},
	\end{align*}
	and thus both  \eqnok{split_8} and \eqnok{gamma_strong_mon_FastTD_k} hold.
	Also, \eqnok{gamma_strong_mon_FastTD} holds, because
	\begin{align*}
	&\theta_t +2\theta_t\gamma_t\lambda_tC\rho^\tau+8 \theta_t\g_t^2\lambda_t^2C^2\rho^{2\tau}+8 \theta_{t+1}\g_{t+1}^2\lambda_{t+1}^2C^2\rho^{2\tau} - \theta_{t-1} (2 \mu \gamma_{t-1} + 1) \\
	&\leq \theta_{t-1}\big[ (\tfrac{4}{3}\mu\gamma+1)(\tfrac{1}{2}\mu\gamma+1)-(1+2\mu\gamma)\big]\leq 0,~~~t = 2, \ldots, k-1,\\
	&\theta_k +2\theta_k\gamma_k\lambda_kC\rho^\tau+2\theta_k\gamma_kC\rho^\tau+8 \theta_k\gamma_k^2\lambda_k^2C^2\rho^{2\tau}+16 \theta_{k}\gamma_{k}^2C^2\rho^{2\tau} - \theta_{k-1} (2 \mu \gamma_{k-1} + 1)\\
	&\leq \theta_{k-1}\big[ (\tfrac{4}{3}\mu\gamma+1)(\tfrac{1}{2}\mu\gamma+1)-(1+2\mu\gamma)\big]\leq 0.
	\end{align*}
	The result then follows from Theorem~\ref{the_FastTD_GMVI} and the following 
	calculations. Set $\Theta_1 = \theta_1 +2\theta_1\gamma_1\lambda_1C\rho^\tau+16 \theta_1\g_1^2\lambda_1^2C^2\rho^{2\tau}+8 \theta_{2}\g_2^2\lambda_{2}^2C^2\rho^{2\tau}$, and note that
	\begin{align*}
	\tfrac{\theta_k^{-1}\Theta_1}{2\mu \gamma + \tfrac{1}{2}}V(x_1,x^*)&\leq 2(\tfrac{4}{3}\mu\gamma+1)^{-k}V(x_1,x^*) \leq 2(1+\tfrac{\mu}{3L})^{-k}V(x_1,x^*) + \tfrac{2(\sigma^2+\varsigma^2 D_X^2)}{\mu^2k},\\
	\tfrac{\theta_k^{-1}}{2\mu\gamma+\tfrac{1}{2}}\tsum_{t=1}^{k} \theta_t \gamma_t^2 \lambda_t^2 &=\tfrac{\theta_k^{-1}}{2\mu\gamma+\tfrac{1}{2}} 	\tsum_{t=1}^{k} \tfrac{\theta_{t-1}^2 \gamma_{t-1}^2}{\theta_t} 
	\leq \tfrac{3\gamma }{2\mu} \leq \tfrac{3q\log k}{2\mu^2k}.
	\end{align*}
\end{proof}

In view of Corollary \ref{step_size_stoch_strong_mon_k_known} the number of iterations performed by the FTD algorithm to find a solution $ \bar{x} \in X$ s.t. $ \mathbb{E}[V(\bar{x}, x^*)] \leq \epsilon $ is bounded by 
$$
\mathcal{O}\{ \max( \tfrac{\underline{\tau} L}{\mu}
\log{\tfrac{V(x_1, x^*)}{\epsilon}} , \tfrac{\underline{\tau} (\sigma^2+\varsigma^2 D_X^2)}{\mu^2 \epsilon} \log{\tfrac{1}{\epsilon}} )\} 
$$

Similarly as in Corollary \ref{restart_stepsize} of the CTD algorithm we can benefit from an index resetting stepsize policy in order to achieve the optimal rate. 

\begin{corollary}\label{restart_stepsize_FastTD}
	Assume $\mu > 0$, let $ \underline{\tau} $ be defined in \eqnok{def_tau}. Set
	$\tau \geq \underline{\tau}$,
	$t_0=\tfrac{8L}{\mu}$, 
	and let
	$$k_s = \max\{(2\sqrt{2}-1)t_0+4, ~\tfrac{5\cdot2^{s+4}(\sigma^2+\varsigma^2 D_X^2)}{\mu^2 V(x_1,x^*)}\} ,~s\in \mathbb{Z}^+, ~~ K_0=0, ~~\mbox{and}~~K_s=\tsum_{s'=1}^{s}k_{s'}.
	$$
	For $t=1,2,...,$ introduce the epoch index $\tilde s$ and local iteration index $\tilde t$ such that
	$$\tilde s = \argmax_{\{s\in \mathbb{Z}^+\}}\mathbbm{1}_{\{K_{s-1}<t\le K_s\}}~~\mbox{and}~~
	\tilde t:=t-K_{\tilde s-1}.$$
	For the stepsize policy
	$$\gamma_t = \tfrac{2}{\mu(t_0+\tilde t-1)},~~~\theta_t = (\tilde t+t_0+1)(\tilde t+t_0) ,~~\mbox{and}~~\lambda_t = \begin{cases}
	& 0, ~~~~~~~~~~~t = 1, \\
	&\tfrac{\theta_{t-1}\gamma_{t-1}}{\theta_{t}\gamma_{t}},~~~   t\ge 2,
	\end{cases}$$
	it holds that
	$
	\bbe[V(x_{K_s+1},x^*)]\leq 2^{-s}V(x_1,x^*)
	$ for any $s \ge 1$.
\end{corollary}
\begin{proof}
	First we note  that in each epoch we use the stepsize policy of Corollary \ref{step_size_stoch_strong_mon_k_unknown}. This enables us to infer that, for $s=1,2,\ldots,$
	\begin{align} \label{convergence_epoch_FastTD}
	\bbe[V(x_{K_{s}+1},x^*)] 
	\leq \tfrac{2( t_0 + 1)  (t_0+2) \bbe[V(x_{K_{s-1}+1},x^*)]}{ (k_s + t_0 + 1)  (k_s + t_0)}  +  \tfrac{40(k_s+1) (\sigma^2+\varsigma^2 D_X^2)}{\mu^2  (k_s + t_0 + 1)  (k_s + t_0)}.
	\end{align}
	As such by taking the specification  of $k_s$ into account, we obtain
	\begin{align*}
	\bbe[V(x_{K_{1}+1},x^*)] 
	\leq &~\tfrac{2( t_0 + 1)  (t_0+2) V(x_{1},x^*)}{ (k_1 + t_0 + 1)  (k_1 + t_0)}  +  \tfrac{40(k_1+1) (\sigma^2+\varsigma^2 D_X^2)}{\mu^2  (k_1 + t_0 + 1)  (k_1 + t_0)}\\
	\leq &~ \tfrac{2( t_0 + 1)  (t_0+2) V(x_{1},x^*)}{ (2\sqrt{2}t_0 + 5)  (2\sqrt{2}t_0+4)} +  \tfrac{40 (\sigma^2+\varsigma^2 D_X^2)}{\mu^2  (k_1 + t_0 + 1) }\leq  \tfrac{V(x_1,x^*)}{2}.
	\end{align*}
	The desired convergence result follows  by recursively using 	
	\eqref{convergence_epoch}.
\end{proof}

{\color{black}
	In view of corollary \ref{restart_stepsize_FastTD},
	the number of epochs performed by the FTD method to find a solution $\bar x \in X$ s.t. $\mathbb{E}[V(\bar x,x^*)] \le \epsilon$ is bounded by $\log_2 (V(x_1, x^*)/\epsilon)$. 
	Then together with the length of each epoch, the total complexity is bounded by
	$${\cal O}\{ \max( \tfrac{\underline{\tau} L }{\mu} \log\tfrac{V(x_1,x^*)}{\epsilon},  \tfrac{\underline{\tau} (\sigma^2+\varsigma^2 D_X^2) }{\mu^2 \epsilon} ) \}.$$
	The above bound improves the one in \eqnok{bound_CTD} in terms of the dependence on $L/\mu$ for the first term.
}

\subsubsection{Convergence result of FTD with unbounded feasible region}

{\color{black}
When the feasible region is unbounded, the stepsize policies in \ref{FTD_bounded} cannot be applied directly. We have two  methods to resolve this issue. One is to apply a projection step, assuming we can find an upper bound to the size of the optimal solution, i.e.  $\|x^*\|\leq G$, and the other is to apply mini-batch. We will elaborate on both methods.

\begin{corollary} \label{FTD_boundness1_1}
	Let $\{x_t\}$ be generated by Algorithm \ref{alg:FastTD_skipping} with $X_t=\{x|\|x\|\leq G\}$ for $t\leq \lceil t_0^2 \rceil$ and $X_t=X$ for $t>\lceil t_0^2 \rceil$. Let $ \underline{\tau} $ be defined in \eqnok{def_tau}. If
	$\tau \geq \underline{\tau}$
	and
	$$
	t_0 =\max\{\tfrac{8L}{\mu},\tfrac{11 \varsigma}{\mu}\}, ~~~ \gamma_t= \tfrac{2}{\mu ( t+t_0 -1) },~~~ \theta_t  = (t + t_0)(t+t_0+1),~~\mbox{and}~~\lambda_t = \tfrac{\theta_{t-1}\gamma_{t-1}}{\theta_t\gamma_t},
	$$
	then
	$$
	\bbe[\|x_t-x^*\|^2]\leq R^2:=4V(x_1,x^*)+4G^2+\tfrac{2}{\varsigma^2}\sigma^2, ~\forall t\in \mathbb{Z}_+,
	$$
	and 
	\begin{align}\label{corr_FTD1_1}
	\bbe[V(x_{k+1},x^*)] 
	\leq \tfrac{2( t_0 + 1)(t_0+2) V(x_1,x^*)}{ (k + t_0)(k+t_0+1)}  +  \tfrac{40(k+1) (\sigma^2+\varsigma^2R^2)}{\mu^2  (k + t_0)(k+t_0+1)}.
	\end{align}
\end{corollary}
\begin{proof}
	In view of the projection step it follows that  $\bbe[\|x_t-x^*\|^2]\leq 4 G^2 \leq R^2,$  $ \forall t \leq \lceil t_0^2 \rceil$. Similar as \eqnok{corr_FTD1}, for $t=\lceil t_0^2 \rceil$ we have
	\begin{align*}
	\bbe[V(x_{t+1},x^*)] 
	\leq \tfrac{2(t_0+1)(t_0+2)V(x_{1},x^*)}{ ( t_0^2 + t_0)(t_0^2+t_0+1)} +  \tfrac{40 (t_0^2+1)  (\sigma^2+4\varsigma^2 G^2)}{\mu^2  (t_0^2 + t_0)(t_0^2+t_0+1)}\leq 2 V(x_1,x^*) + \tfrac{1}{3\varsigma^2} \sigma^2 +\tfrac{4}{3} G^2 \leq \tfrac{R^2}{2}.
	\end{align*}
	Consequently, $\bbe[\|x_{t+1}-x^*\|^2]\leq 2 \bbe[V(x_{t+1},x^*)] \leq R^2$. Now assume that $\bbe[\|x_t-x^*\|^2]\leq R^2, ~\forall t=1,2,...,\hat{k}$, where $\hat{k}>\lceil t_0^2 \rceil$ then
	\begin{align*}
	\bbe[V(x_{\hat k+1},x^*)] 
	\leq \tfrac{2(t_0+1)(t_0+2)V(x_{1},x^*)}{ (\hat k + t_0)(\hat k+t_0+1)} +  \tfrac{40 (\hat k+1) (\sigma^2+\varsigma^2 R^2)}{\mu^2  (\hat k+ t_0)(\hat k+t_0+1)}\leq \tfrac{2}{3} V(x_1,x^*) + \tfrac{1}{3\varsigma^2} \sigma^2 +\tfrac{1}{3} R^2 \leq\tfrac{R^2}{2}.
	\end{align*}
	By inductive arguement, $\bbe[\|x_t-x^*\|^2]\leq R^2,~ \forall t\in \mathbb{Z}_+$. Together with \eqref{corr_FTD1}, we obtain the desired result.
\end{proof} 

In view of Corollary \ref{FTD_boundness1_1} 
the number of samples required by the FTD method to find a solution $\bar x \in X$ s.t. $\mathbb{E}[V(\bar x,x^*)] \le \epsilon$ is bounded by 
$$
	\mathcal{O}\{ \max( \tfrac{\underline{\tau} (L+\varsigma) }{\mu}
	\sqrt{\tfrac{V(x_1, x^*)}{\epsilon} }, \tfrac{\underline{\tau} \left(\sigma^2+\varsigma^2R^2 \right)}{\mu^2 \epsilon}  ) \}
	$$
Note that  our FTD algorithm only requires projection steps for a specific number of   $ \lceil t_0^2 \rceil$ iterations. 

By using the same stepsize as Corollary \ref{step_size_stoch_strong_mon_k_known} and adding a projection step for each iteration (i.e. $X_t=\{x|\|x\|\leq G\}$), we can bound the number of iterations of finding an $\epsilon$-solution by $$
\mathcal{O}\{ \max( \tfrac{\underline{\tau} L}{\mu}
\log{\tfrac{V(x_1, x^*)}{\epsilon}} , \tfrac{\underline{\tau} (\sigma^2+\varsigma^2 G^2)}{\mu^2 \epsilon} \log{\tfrac{1}{\epsilon}} )\}. 
$$

Similarly, we can modify the stepsize of index-resetting in Corollary \ref{restart_stepsize_FastTD} accordingly and achieve accelerated convergence rate as stated in the following corollary.  The proof follows directly by applying Corollary \ref{FTD_boundness1_1} into the proof of Corollary \ref{restart_stepsize_FastTD}. 

\begin{corollary}\label{restart_stepsize_FastTD_2}
	Let $ \underline{\tau} $ be defined in \eqnok{def_tau}. Set
	$\tau \geq \underline{\tau}$,
	$t_0=\max\{\tfrac{8L}{\mu},\tfrac{11 \varsigma}{\mu}\}$, 
	and let
	$$k_s = \max\{(2\sqrt{2}-1)t_0+4, ~\tfrac{5\cdot2^{s+4}(\sigma^2+\varsigma^2 R^2)}{\mu^2 V(x_1,x^*)}\} ,~s\in \mathbb{Z}^+, ~~ K_0=0, ~~\mbox{and}~~K_s=\tsum_{s'=1}^{s}k_{s'}.
	$$
	For $t=1,2,...,$ introduce the epoch index $\tilde s$ and local iteration index $\tilde t$ such that
	$$\tilde s = \argmax_{\{s\in \mathbb{Z}^+\}}\mathbbm{1}_{\{K_{s-1}<t\le K_s\}}~~\mbox{and}~~
	\tilde t:=t-K_{\tilde s-1}.$$
	Taking $X_t=\{x|\|x\|\leq G\}$ if $\tilde t \leq \lceil t_0^2 \rceil$ and otherwise $X_t=X$. For the stepsize policy
	$$\gamma_t = \tfrac{2}{\mu(t_0+\tilde t-1)},~~~\theta_t = (\tilde t+t_0+1)(\tilde t+t_0) ,~~\mbox{and}~~\lambda_t = \begin{cases}
	& 0, ~~~~~~~~~~~t = 1, \\
	&\tfrac{\theta_{t-1}\gamma_{t-1}}{\theta_{t}\gamma_{t}},~~~   t\ge 2,
	\end{cases}$$
	it holds that
	$
	\bbe[V(x_{K_s+1},x^*)]\leq 2^{-s}V(x_1,x^*)
	$ for any $s \ge 1$.
\end{corollary}

In view of Corollary \ref{restart_stepsize_FastTD_2} the number of samples required by the FTD algorithm to find a solution $ \bar{x} \in X$ s.t. $ \mathbb{E}[V(\bar{x}, x^*)] \leq \epsilon $ is bounded by 
$${\cal O}\{ \max( \tfrac{\underline{\tau} (L+\varsigma) }{\mu} \log\tfrac{V(x_1,x^*)}{\epsilon},  \tfrac{\underline{\tau} (\sigma^2+\varsigma^2 R^2) }{\mu^2 \epsilon} ) \}.$$

Now we proceed to the mini-batch method. For this method we are bound to use 
$V(x,y)=\tfrac{\|x-y\|_2^2}{2}$. We assume that we have access to $m$ independent Markovian streams denoted by
$ \xi_{0,(i)}, \hdots, \xi_{t ,(i)}$, where $i \in [m]$.  At the time of an algorithm  update we form
$\tilde{F}(x_t, \xi_{t ,(i)}^{\tau})$ separately and via a parallel summation algorithm we define 
\begin{align}\label{def_mini_batch}
 \tilde{F}(x_t, \xi_t^{\tau}) := \tfrac{1}{m} \sum_{i=1}^m \tilde{F}(x_t, \xi_{t ,(i)}^{\tau}).
\end{align}
 Note that in the mini-batch setting
\begin{align}\label{property_mini_batch}
& \mathbb{E}[ \| \tilde{F}(x_t, \xi_t^{\tau}) - \mathbb{E}[\tilde{F}(x_t, \xi_t^{\tau})| \mathcal{F}_{t-1}]\|^2]  \nn\\  = &  
\mathbb{E}[ \tfrac{1}{m^2} \langle \tsum_{i=1}^m \tilde{F}(x_t, \xi_{t ,(i)}^{\tau}) - \mathbb{E}[\tilde{F}(x_t, \xi_{t ,(i)}^{\tau})| \mathcal{F}_{t-1}],  \tsum_{i=1}^m \tilde{F}(x_t, \xi_{t ,(i)}^{\tau}) - \mathbb{E}[\tilde{F}(x_t, \xi_{t ,(i)}^{\tau})| \mathcal{F}_{t-1}]\rangle ] \nn\\ 
 = & \tfrac{1}{m^2} \tsum_{i=1}^m \mathbb{E}[ \| \tilde{F}(x_t, \xi_{t ,(i)}^{\tau}) - \mathbb{E}[\tilde{F}(x_t, \xi_{t ,(i)}^{\tau})| \mathcal{F}_{t-1}]\|^2]  \leq \tfrac{\sigma_t^2}{m},
\end{align}
where $\sigma_t^2 :=\tfrac{\sigma^2 + \varsigma^2\bbe[\|x_t-x^*\|^2]}{2}$. The second equality utilizes the independence of the individual Markovian sample trajectories, and the inequality utilizes \eqref{variance_bound}.

\begin{corollary} \label{FTD_boundness2_1}
	Let $\{x_t\}$ be generated by Algorithm \ref{alg:FastTD_skipping} with constant batch-size $m=\lceil\tfrac{\varsigma}{\mu}\rceil$ for the first $\lceil t_0^2 \rceil$ iterations and $m=1$ after $\lceil t_0^2 \rceil$ iterations. If the parameters $\tau$, $\{\gamma_t\}$, $\{\lambda_t\}$, $\{\theta_t\}$ are selected as in Corollary \ref{step_size_FastTD_skipping} and $t_0 = \max\{\tfrac{8L}{\mu},\tfrac{60 \varsigma }{\mu}\}$, then
	$$
	\bbe[\|x_t-x^*\|^2]\leq R^2, ~\forall t\in \mathbb{Z}_+,
	$$
	and 
	\begin{align}\label{corr_FTD1_2}
	\bbe[V(x_{k+1},x^*)] 
	\leq \tfrac{2( t_0 + 1)(t_0+2) V(x_1,x^*)}{ (k + t_0)(k+t_0+1)}  +  \tfrac{40(k+1) (\sigma^2+\varsigma^2R^2)}{\mu^2  (k + t_0)(k+t_0+1)}.
	\end{align}
	where $R^2:=6 \|x_1-x^*\|^2+\tfrac{2}{\varsigma^2}\sigma^2$.
\end{corollary}
\begin{proof}
	Note that  \eqnok{corr_FTD1} still holds with $(k+1) (\sigma^2+\varsigma^2 D_X^2)$ replaced by $\tsum_{t=1}^{k}2\sigma_t^2+\sigma_1^2+\sigma_k^2$. In view of the mini-batch, we have that for all $k<\lceil t_0^2 \rceil$,
	$$
	\bbe[V(x_{k+1},x^*)] 
	\leq \tfrac{2( t_0 + 1)(t_0+2) V(x_1,x^*)}{ (k + t_0)(k+t_0+1)}  +  \tfrac{40(\sum_{t=1}^{k}2\sigma_t^2+\sigma_1^2+\sigma_k^2)}{\varsigma\mu  (k + t_0)(k+t_0+1)}.
	$$
	By \eqnok{corr_FTD1}, for $k=1$ we have 
	$$
	\bbe[\|x_{2}-x^*\|^2] 
	\leq 2\|x_1-x^*\|^2 +  \tfrac{80 (\sigma^2+\varsigma^2\|x_1-x^*\|^2)}{\varsigma\mu  (1 + t_0)(2+t_0)}
	\leq 2\|x_1-x^*\|^2 + \tfrac{2}{3\varsigma^2}\sigma^2 + \tfrac{2}{3}\|x_1-x^*\|^2\leq R^2.
	$$
	Now assume that $\bbe[\|x_t-x^*\|^2]\leq R^2, ~\forall t=1,2,...,\hat{k}$, where $\hat{k} \leq \lceil t_0^2 \rceil$ then
	\begin{align*}
	\bbe[\|x_{\hat k+1}-x^*\|^2] 
	\leq \tfrac{2(t_0+1)(t_0+2)\|x_1-x^*\|^2}{ (\hat k + t_0)(\hat k+t_0+1)} +  \tfrac{40 (\hat k+1) ( \sigma^2+ \varsigma^2R^2)}{\varsigma\mu  (\hat k+ t_0)(\hat k+t_0+1)}
	\leq 2\|x_1-x^*\|^2 + \tfrac{2}{3\varsigma^2}\sigma^2 + \tfrac{2}{3}R^2 \leq R^2.
	\end{align*}
	Now assume that $\bbe[V(x_t,x^*)]\leq R^2, ~\forall t=1,2,...,\hat{k}$, where $\hat{k}>\lceil t_0^2 \rceil$ then
	\begin{align*}
	\bbe[\|x_{\hat k+1}-x^*\|^2] 
	\leq \tfrac{2(t_0+1)(t_0+2)\|x_1-x^*\|^2}{ (\hat k + t_0)(\hat k+t_0+1)} +  \tfrac{40 (\hat k+1) ( \sigma^2+ \varsigma^2R^2)}{\mu^2  (\hat k+ t_0)(\hat k+t_0+1)}\leq 2\|x_1-x^*\|^2  + \tfrac{2}{3\varsigma^2}\hat \sigma^2 + \tfrac{2}{3}R^2 \leq R^2.
	\end{align*}
	As such, $\bbe[\|x_t-x^*\|^2]\leq R^2,~ \forall t\in \mathbb{Z}_+$.
\end{proof} 

The method above has the same convergence rate as Corollary \ref{FTD_boundness1_1}. 
}

\subsection{Robust analysis for FTD}

In this subsection, we consider stochastic generalized monotone VIs which satisfy  \eqnok{G_monotone0}
with $\mu = 0$. Throughout this subsection we assume $\|\cdot\|=\|\cdot\|_2$, $V(x,y)=\tfrac{1}{2}\|x-y\|^2$, and consequently
\beq \label{omega_Lip}
\|\nabla \o(x_1) - \nabla \o(x_2)\| \le  \|x_1 - x_2\|, \ \ \forall x_1, x_2 \in X.
\eeq
 Our goal is to show the FTD method is robust in the sense that it converges when the
modulus $\mu$ is rather small if we use a mini-batch of size $m$, where the mini-batch operator is defined in \eqref{def_mini_batch}.  First, we use the induction method to prove that under constant stepsize, $\bbe[\|x_t-x^*\|^2]$ is bounded.

\begin{lemma} \label{bound_norm}
	\textcolor{black}{Let \textit{Assumptions A, B} and \textit{C} hold.} Let $x_t$, $t=1, \ldots, k+1$, be generated by the FTD method in Algorithm~\ref{alg:FastTD_skipping}. If \beq \label{step_condition} \theta_t=1,~\lambda_t=1,~\gamma_t=\gamma= \min\{\tfrac{1}{4L},\tfrac{1}{8\sqrt{2}\varsigma}\},~\rho^\tau\le\tfrac{1}{16\gamma C(k+1)}\eeq and  a constant mini-batch size is set to $m=k+1$, then
		{\color{black}
	\beq \label{bound_R}
	\bbe[\|x_{t}-x^*\|^2] \leq R^2, ~~\forall t =1,2,...,k+1,
	\eeq 
	where
	\beq \label{def_R}
	 R^2:=4\|x_1-x^*\|^2+32\gamma^2\sigma^2.
	\eeq}
\end{lemma}
	
\begin{proof}
	\textcolor{black}{
From the parameters selected, \eqnok{theta_t} and \eqnok{split_8} are satisfied. Therefore, considering \eqref{property_mini_batch}, \eqnok{f} still holds with $\sigma^2$ replaced by $\tfrac{\sigma^2}{k+1}$ and $\varsigma^2$ replaced by $\tfrac{\varsigma^2}{k+1}$. By plugging $\theta_t =1$, $\lambda_t=1$ and $\gamma_t=\gamma$ into \eqnok{f}, we can simplify it to}
{\color{black}
\begin{align*}
&\tsum_{t=1}^k \gamma \bbe[ \langle  F(x_{t+1}) , x_{t+1} - x^* \rangle] + \bbe[V(x_{k+1},x^*)]- 4 L^2  \gamma^2\bbe[ \| x_{k+1} - x^* \|^2]  \nn\\
& \leq V(x_1,x^*) +  \tsum_{t=1}^k4\gamma^2 C^2\rho^{2\tau}(\|x_t-x^*\|^2+\|x_{t-1}-x^*\|^2) + \tsum_{t=1}^k   \tfrac{\gamma^2}{k+1} (4\sigma^2+2\varsigma^2\bbe[\|x_t-x^*\|^2]+2\varsigma^2\bbe[\|x_{t-1}-x^*\|^2])\nn\\
& \quad + \tsum_{t=1}^k\gamma C\rho^\tau\bbe[\|x_t-x^*\|^2]+ 8\gamma^2C^2\rho^{2\tau}\bbe[\|x_k-x^*\|^2]+\tfrac{4\gamma^2}{k+1}(\sigma^2+\varsigma^2\bbe[\|x_k-x^*\|^2])+\gamma C\rho^\tau\bbe[\|x_k-x^*\|^2].
\end{align*}
Invoking the fact $ \langle F(x_{t+1}) , x_{t+1} - x^* \rangle \ge 0$ due to \eqnok{G_monotone1}, and the condition $V(x,y)=\tfrac{1}{2}\|x-y\|^2$, we have
\begin{align*}
(\tfrac{1}{2}- 4 L^2\gamma^2)  \bbe[ \| x_{k+1} - x^* \|^2]  
\leq& \tfrac{1}{2}\|x_1-x^*\|^2 +  \tsum_{t=1}^k 4\gamma^2 C^2\rho^{2\tau}(\|x_t-x^*\|^2+\|x_{t-1}-x^*\|^2)+\tsum_{t=1}^k   \tfrac{\gamma^2}{k+1} (4\sigma^2\nn\\
+& 2\varsigma^2\bbe[\|x_t-x^*\|^2]+2\varsigma^2\bbe[\|x_{t-1}-x^*\|^2])+\tsum_{t=1}^k\gamma C\rho^\tau\bbe[\|x_t-x^*\|^2]\nn\\
+& 8\gamma^2C^2\rho^{2\tau}\bbe[\|x_k-x^*\|^2]+\tfrac{4\gamma^2}{k+1}(\sigma^2+\varsigma^2\bbe[\|x_k-x^*\|^2])+\gamma C\rho^\tau\bbe[\|x_k-x^*\|^2].
\end{align*}
Pluging the selection of $\tau$ into the above inequality, we obtain
\begin{align}\label{f4}
(\tfrac{1}{2}- 4 L^2\gamma^2)  \bbe[ \| x_{k+1} - x^* \|^2]  
\leq& \tfrac{1}{2}\|x_1-x^*\|^2 +  \tsum_{t=1}^k \tfrac{1}{64(k+1)}(\|x_t-x^*\|^2+\|x_{t-1}-x^*\|^2)+\tsum_{t=1}^k   \tfrac{\gamma^2}{k+1} (4\sigma^2\nn\\
+& 2\varsigma^2\bbe[\|x_t-x^*\|^2]+2\varsigma^2\bbe[\|x_{t-1}-x^*\|^2])+\tsum_{t=1}^k\tfrac{1}{16(k+1)}\bbe[\|x_t-x^*\|^2]\nn\\
+& \tfrac{1}{32(k+1)}\bbe[\|x_k-x^*\|^2]+\tfrac{4\gamma^2}{k+1}(\sigma^2+\varsigma^2\bbe[\|x_k-x^*\|^2])+\tfrac{1}{16(k+1)}\bbe[\|x_k-x^*\|^2].
\end{align}}
We now prove the result by induction. For the base case, $t=1$, obviously, we have
$$\|x_1-x^*\|^2 \leq R^2.$$
Now assume that $\bbe[\|x_{t}-x^*\|^2] \leq R^2, ~~\forall t =1,2,...,\hat{k}$, in which $\hat{k}<k$. From \eqnok{step_condition} and \eqnok{f4}, we have
{\color{black}
\begin{align}
	\tfrac{1}{4}\bbe[\|x_{\hat k +1}-x^*\|^2]\leq \tfrac{1}{2}\|x_1-x^*\|^2+\tfrac{3R^2}{32}+4\gamma^2  \sigma^2 + 4\gamma^2 \varsigma^2 R^2 \leq \tfrac{1}{2}\|x_1-x^*\|^2+\tfrac{R^2}{8}+4\gamma^2  \sigma^2  = \tfrac{R^2}{4},
\label{f5}
\end{align}}
so that 
$
\bbe[ \| x_{\hat{k}+1} - x^* \|^2] \leq R^2.
$
As such, $\bbe[\|x_{t}-x^*\|] \leq R^2, ~~\forall i =1,2,...,k+1.$
\end{proof}

\textcolor{black}{With the proof above, mini-batch with constant batch-size $m = k$ is required to show the boundness of the iterates (see \eqref{f4} and \eqref{f5}). However, if the feasible region $X$ is bounded, we can untilize a progressively increasing batch-size, i.e.,  $m_t=t$, and the convergence result in the following lemma will be valid as well.} Next, Lemma~\ref{lemma_bnd_res_FastTD} provides a technical result regarding
the relation between the residual of $x_{r+1}$ and the summation of squared distances $\tsum_{t=1}^k\bbe[ \|x_{t+1}- x_t\|^2]$.  We define the 
output solution of the FTD method as $x_{r+1}$, where $r$ is uniformly chosen from $\{2, 3, \ldots, k\}$.

\begin{lemma} \label{lemma_bnd_res_FastTD}
	\textcolor{black}{Let \textit{Assumptions A, B} and \textit{C} hold.} Let $x_t$, $t=1, \ldots, k+1$, be generated by the FTD method in Algorithm~\ref{alg:FastTD_skipping}.
	Assume $r$ is uniformly chosen from $\{2, 3, \ldots, k\}$. If
	\beq \label{cond_rel_two_iterate_FastTD}
	\tsum_{t=1}^k \bbe[\|x_{t+1}-x_t\|^2] \le \tilde{\delta},
	\eeq 
	and the parameters satisfy \eqnok{step_condition},
	then 
	{\color{black}
		\begin{align}\label{res}
	\bbe[\res(x_{r+1})] \le(1+2\lambda_r)\sqrt{\tfrac{\sigma^2}{k+1}+\tfrac{\varsigma^2 R^2}{k+1}+2C^2\rho^{2\tau}R^2}+ 2 (L + \tfrac{1}{\gamma_r} + L \lambda_r)\tfrac{\sqrt{2 \tilde \delta}}{\sqrt{k-1}}.
	\end{align}}
\end{lemma}

\begin{proof}
	Observe that by the optimality condition of \eqnok{FastTD_skipping_step}, we have
	\beq \label{opt_det_step_FastTD}
	\langle F(x_{r+1}) + \delta_r, x - x_{r+1} \rangle \ge 0 \ \ \forall x \in X,
	\eeq
	with
	\begin{align*}
	\tilde \delta_r := &\tilde F(x_r, \xi_r^\tau) -F(x_{r+1}) + \lambda_r [\tilde F(x_r, \xi_r^\tau) - \tilde F(x_{r-1}, \xi_{r-1}^\tau)] + \tfrac{1}{\gamma_r} [\nabla \o(x_{r+1}) - \nabla \o(x_{r})]\\
	= &(1+\lambda_r) [\tilde F(x_r, \xi_{r}^\tau)-F(x_r)]+ \lambda_r[\tilde F(x_{r-1},, \xi_{r-1}^\tau) -F(x_{r-1})]+[F(x_r) -F(x_{r+1})] \\
	&+ \lambda_R [F(x_r) - F(x_{r-1})] + \tfrac{1}{\gamma_r} [\nabla \o(x_{r+1}) - \nabla \o(x_{r})]
	\end{align*}
	By \eqnok{Lipschitz} and \eqnok{omega_Lip}, we have
	\begin{align}
	\|\tilde \delta_r\| \le&~ (1+\lambda_r) \|\tilde F(x_r, \xi_{r}^\tau)-F(x_r)\| +  \lambda_r \|\tilde F(x_{r-1}, \xi_{r-1}^\tau) -F(x_{r-1})\|\nn \\
	&+\| F(x_r) - F(x_{r+1}) \| + \lambda_r \| F(x_r) -  F(x_{r-1})\| +   \tfrac{1}{\gamma_r} \|\nabla \o(x_{r+1}) - \nabla \o(x_{r})\| \nn \\
	\le&~  (1+\lambda_r) \|\tilde F(x_r, \xi_{r}^\tau)-F(x_r)\| +  \lambda_r \|\tilde F(x_{r-1}, \xi_{r-1}^\tau) -F(x_{r-1})\|+ (L + \tfrac{1}{\gamma_r})\|x_{r+1} - x_r\| + L \lambda_r \|x_r - x_{r-1}\|. \label{bnd_delta_R_FastTD} 
	\end{align}
	{\color{black}
	Taking expectation on both side of the above inequality, and by applying Jensen's ineuqality, $$\bbe[\|\tilde F(x_t, \xi_{t}^\tau)-F(x_t)\|]\le \sqrt{\bbe[\|\tilde F(x_t, \xi_{t}^\tau)-F(x_t)\|^2]} \le \sqrt{\tfrac{\sigma^2}{k+1}+\tfrac{\varsigma^2 R^2}{k+1}+2C^2\rho^{2\tau}R^2},$$ 
	where the second inequality follows from \eqref{property_mini_batch} and \eqref{bound_delta_norm1}.  As such, 
	\begin{align}
	\bbe[\|\tilde \delta_r\|] \le&(1+2\lambda_r)\sqrt{\tfrac{\sigma^2}{k+1}+\tfrac{\varsigma^2 R^2}{k+1}+2C^2\rho^{2\tau}R^2} + (L+\lambda_rL+\tfrac{1}{\gamma_r})(\bbe[\|x_{r+1} - x_r\|]+\bbe[\|x_r - x_{r-1}\|]). \label{bnd_delta_R_FastTD_2} 
	\end{align}}
	It follows from \eqnok{cond_rel_two_iterate_FastTD} that
	\begin{align}
	\tsum_{t=1}^k \left(\bbe[\|x_{t+1} - x_t\|^2] + \bbe[\|x_t - x_{t-1}\|^2] \right) \le 2 \tsum_{t=1}^2 \bbe[\|x_{t+1} - x_t\|^2] \le 2\tilde \delta.
	\end{align}
	The previous conclusion and the fact that $r$ is uniformly chosen from $\{2, 3, \ldots, k\}$ imply that
	\[
	\bbe[\|x_{r+1} - x_r\|^2] + \bbe[\|x_r- x_{r-1}\|^2] \le \tfrac{2 \tilde \delta}{k-1}.
	\]
	and hence 
	\beq \label{bnd_two_diff_FastTD}
	\max\{\bbe[\|x_{r+1} - x_r\|], \bbe[\|x_{r-1} - x_r\|]\} \le \tfrac{\sqrt{2 \tilde \delta}}{\sqrt{k-1}}.
	\eeq
	We then obtain the desired result from the definition of $\res(\cdot)$  and relations
	\eqnok{opt_det_step_FastTD}, \eqnok{bnd_delta_R_FastTD}, and \eqnok{bnd_two_diff_FastTD}.
\end{proof}

We can now show the convergence of the FTD method by showing the convergence of $\bbe[\res(x_{r+1})] $.
{\color{black}
\begin{theorem}\label{lemma_dist_no_strong_monotone_FastTD}
	\textcolor{black}{Let \textit{Assumptions A, B} and \textit{C} hold.} Let $\{x_t\}$ be generated by Algorithm \ref{alg:FastTD_skipping}.
	If the parameters $\{\theta_t\} $, $\{\gamma_t\}$ and $\{\lambda_t\}$ in Algorithm \ref{alg:FastTD_skipping} satisfy \eqnok{step_condition} and the constant batch size $m=k+1$, then 
	\begin{align}
	\tsum_{t=1}^k   \bbe[\|x_{t+1}-x_t\|^2] 
	\leq 4\|x_1-x^*\|^2 + 64\gamma^2\sigma^2+R^2. \label{GMVI_result_FastTD}
	\end{align}
	where $R^2$ is defined in \eqnok{def_R}. Invoking \eqnok{res} and the fact that $r$ is uniformly chosen from $\{2, 3, \ldots, k\}$,
	then 
	\beq \label{bnd_res_no_FastTD}
	\bbe[\res(x_{r+1})] \le3\sqrt{\tfrac{\sigma^2}{k+1}+\tfrac{\varsigma^2 R^2}{k+1}+\tfrac{L^2R^2}{8(k+1)^2}}+ (4L + \tfrac{2}{\gamma} )\tfrac{\sqrt{16V(x_1,x^*)+2R^2+128\gamma^2\sigma^2}}{\sqrt{k}}.
	\eeq

\end{theorem}}

\begin{proof} 
	Observe that  \eqnok{eqn:common_bnd_VI_FastTD} still holds. We plug $\theta_t =1$, $\lambda_t=1$ and $\gamma_t=\gamma$ into \eqnok{eqn:common_bnd_VI_FastTD}. However,
	we will bound $\tilde Q_t$ in a slightly different way. 
	\begin{align*}
		\tilde Q_t	
		&\geq   \tsum_{t=1}^k    \left[ -   \gamma L \|   x_t - x_{t-1} \| 
		\|   x_{t+1} - x_t \|  +   V(x_t, x_{t+1}) +   \gamma \langle\delta_t^\tau - \delta_{t-1}^\tau, x_{t+1} - x_t \rangle\right] \\
		& \geq   \tsum_{t=1}^k  \left[-    \gamma L \|   x_t - x_{t-1} \| 
		\|   x_{t+1} - x_t \|  +   \tfrac{1}{8} \|  x_t - x_{t+1} \|^2 +   \tfrac{1 }{8} \|  x_t - x_{t-1} \|^2  \right]  +       \tfrac{1}{8} \|  x_k - x_{k+1} \|^2  \\
		& \quad +   \tsum_{t=1}^k \left[ \gamma \langle\delta_t^\tau - \delta_{t-1}^\tau, x_{t+1} - x_t \rangle + \tfrac{1}{8} \|  x_t - x_{t+1} \|^2  \right] +\tsum_{t=1}^k \tfrac{1}{8} \|  x_t - x_{t+1} \|^2 \\
		&\geq    \tfrac{ 1}{8} \|  x_k - x_{k+1} \|^2 + \tsum_{t=1}^k \left[ \gamma\langle\delta_t^\tau - \delta_{t-1}^\tau, x_{t+1} - x_t \rangle + \tfrac{\theta_{t} }{8} \|  x_t - x_{t+1} \|^2  \right]
		+\tsum_{t=1}^k \tfrac{1 }{8} \|  x_t - x_{t+1} \|^2\\
		&\geq \tfrac{ 1}{8} \|  x_k - x_{k+1} \|^2 -  \tsum_{t=1}^k \left( 2\gamma^2 \|\delta_t^\tau - \delta_{t-1}^\tau \|^2 \right) +\tsum_{t=1}^k \tfrac{1 }{8} \|  x_t - x_{t+1} \|^2,
	\end{align*} 
	where the second inequality follows from \eqnok{strong_convex_V}, the third inequality follows from \eqnok{split_8} and the last one follows from the Young's inequality. Using the above bound of $\tilde Q_t$ in \eqnok{eqn:common_bnd_VI_FastTD}, we obtain
	\begin{align}
	\tsum_{t=1}^k     \Delta V_t(x)  & \geq   \tsum_{t=1}^k \big[ \gamma \langle  F(x_{t+1}) , x_{t+1} - x \rangle \big] + \tsum_{t=1}^k \gamma\langle \delta_t^\tau,x_{t}-x \rangle- \gamma\langle \Delta F_{k+1} , x_{k+1} - x \rangle \nn\\
	& \quad + \gamma\langle \delta_k^\tau , x_{k+1} - x \rangle + \tfrac{1}{8} \|  x_k - x_{k+1} \|^2  -     \tsum_{t=1}^k 2\gamma^2 \|\delta_t^\tau - \delta_{t-1}^\tau \|^2  +\tsum_{t=1}^k \tfrac{1 }{8} \|  x_t - x_{t+1} \|^2\nn \\
	&\ge \tsum_{t=1}^k \big[ \gamma\langle \tilde F(x_{t+1}) , x_{t+1} - x \rangle \big] + \tsum_{t=1}^k \gamma\langle \delta_t^\tau,x_{t}-x \rangle- 4\gamma^2(L^2   \|  x - x_{k+1} \|^2+\|\delta_k^\tau\|^2) \nn\\
	& \quad    +\tsum_{t=1}^k \tfrac{1 }{8} \|  x_t - x_{t+1} \|^2 +\gamma \langle \delta_k^\tau , x_{k} - x \rangle-  \tsum_{t=1}^k 2\gamma^2 \|\delta_t ^\tau- \delta_{t-1} ^\tau\|^2, \label{FastTD_bound_general_tmp3}
	\end{align}
	where the second inequality follows from
	\begin{align*}
	&- \gamma \langle \Delta  F_{k+1} , x_{k+1} - x \rangle  + \gamma \langle \delta_k^\tau , x_{k+1} - x \rangle 
	+\tfrac{1 }{8} \|  x_k - x_{k+1} \|^2 \\
	\geq &
	-  \gamma L \|  x_k - x_{k+1} \|  \|  x - x_{k+1} \| - \gamma\|\delta_k^\tau\|\|  x_k - x_{k+1} \| +
	\tfrac{1 }{8} \|  x_k - x_{k+1} \|^2 +  \gamma \langle \delta_k^\tau , x_k - x \rangle  \\
	\geq &  - 4 L^2 \gamma^2  \|  x - x_{k+1} \|^2  -  4\gamma^2 \|\delta_k^\tau\|^2 +  \gamma \langle \delta_k^\tau , x_k - x \rangle.
	\end{align*}
	Fixing $x=x^*$, taking expectation on both side of the inequality and using the fact $ \langle F(x_{t+1}) , x_{t+1} - x^* \rangle \ge 0$ due to \eqnok{G_monotone0} , we have
	{\color{black}
	\begin{align*}
		& \tsum_{t=1}^k  \big( \bbe[ V(x_{t+1},x^*)]+ \tfrac{1}{8}  \bbe[\|x_t-x_{t+1}\|^2]\big)- 4 L^2 \gamma^2 \bbe[\| x_{k+1} - x^* \|^2] \nn\\
		& \leq \tsum_{t=1}^k \bbe[V(x_t,x^*)] +\tsum_{t=1}^k\bigg \{ 4\gamma^2\left[\tfrac{2\sigma^2}{k+1}+\left(\tfrac{\varsigma^2}{k+1}+2\rho^{2\tau}C^2\right)\left(\bbe[\|x_t-x^*\|^2]+\bbe[\|x_{t-1}-x^*\|^2]\right)\right]\nn\\ & \quad +\gamma C \rho^\tau \bbe[\|x_t-x^*\|^2]  \bigg\}
		+4\gamma^2 \left\{\tfrac{\sigma^2}{k+1} +\left(\tfrac{\varsigma^2}{k+1}+2\rho^{2\tau}C^2\right)\bbe[\|x_k-x^*\|^2  ]\right \}+ \gamma C\rho^\tau \bbe[\|x_k-x^*\|^2].
	\end{align*}}
	Using the condition \eqnok{step_condition}, $V(x,y)=\tfrac{1}{2}\|x-y\|^2$ and we know that $16L^2\gamma_k^2\leq 1$, we have
	\begin{eqnarray*}
		\tsum_{t=1}^{k+1}   \bbe[\|x_{t+1}-x_t\|^2] 
		\leq 4\|x_1-x^*\|^2 +(k+1)\left[  64\gamma^2(\tfrac{\sigma^2+\varsigma^2 R^2}{k+1}+2\rho^{2\tau}C^2R^2) +8\gamma C \rho^\tau R^2  \right]
	\end{eqnarray*}
	Moreover, by choosing the parameter setting in \eqnok{step_condition}, we have
	$$
	\tsum_{i=1}^k \bbe[\|x_{t+1}-x_t\|^2 ]\le 4\|x_1-x^*\|^2 + (k+1) 64\gamma^2\tfrac{\sigma^2+\varsigma^2 R^2}{k+1} +\tfrac{R^2}{2} \leq 4\|x_1-x^*\|^2 + 64\gamma^2\sigma^2+R^2,
	$$
	in which $R^2=4\|x_1-x^*\|^2+32\gamma^2\sigma^2$ from \eqnok{def_R}.
	The result in \eqnok{bnd_res_no_FastTD} immediately follows from the previous conclusion and Lemma~\ref{lemma_bnd_res_FastTD}.
\end{proof}

\vgap

In view of Theorem~\ref{lemma_dist_no_strong_monotone_FastTD}, the FTD method
can find a solution $\bar x \in X$ s.t. $\bbe[\res(\bar x)] \le \epsilon$ in ${\cal O} (1/\epsilon^2)$
iterations \textcolor{black}{and ${\cal O} (1/\epsilon^4)$ overall sample complexity} for solving generalized stochastic monotone VIs in the Markovian noise setting. 
FTD seems to be the only one among these three algorithms, i.e., TD, CTD and FTD,
that can be applied to stochastic GMVIs with $\mu =0$.

\section{Policy evaluation in Markov decision processes} \label{MDP_RL}

In this section we consider the problem of policy evaluation which is an important step 
in many reinforcement learning algorithms.  The term policy evaluation pertains to the computation of the
value function $V_{\nu}$ of a given policy $\nu$ in a Markov decision process (MDP). 

We will restrict ourselves to MDPs with finite state and action spaces, though most of our results have natural counterparts in the countable case. An infinite horizon MDP is abstracted as a quintuple $ \mathcal{M} = ( \mathcal{S}, \mathcal{A}, p , r, \beta)$, where  $ \mathcal{S} = [n]$  denotes the state space, $ \mathcal{A}= [m]$ the action space, 
$ p : \clS \times \clS \times \clA \rightarrow \mathbb{R}$ the transition probability function, $ r : \clS \times \clS \times \clA \rightarrow \mathbb{R}$ is the reward function and $ \beta \in (0,1)$ the discount factor.

At time $t = 0,1,2, \hdots, $ as a result of choosing action $a \in \mathcal{A} $, while 
at state $s \in \mathcal{S}$, the system moves to some state $s' \in \mathcal{S}$ determined by the conditional probability 
$ \Pr[s_{t+1} = s' | s_t = s, a_t = a ] = p(s,s',a)$ and incurres the reward $r(s,s',a)$. We will restrict ourselves to stationary randomized policies $ \nu : \clS \times \clA \rightarrow \mathbb{R}$ where $ \nu(s,a) = \Pr[a_t = a| s_t = s]$. 
Once a policy is fixed, the sequence of states $ \{s_0, s_1, \hdots, \} $ becomes a time-homogeneous Markov-chain with transition probability matrix $P$, where  the $(i,j)$-th entry of $P$ is 
$ P_{ij} = \sum_{a\in \clA}  \nu(i,a) p(i,j,a)$. We assume 
that the resulting Markov chain has a single ergodic class with unique stationary distribution $\pi$, satisfying $ \pi = \pi P$. 
We denote by $ R(i) = \tsum_{j \in \clS} \tsum_{a \in \clA} p(i,j,a)  \nu(i,a) r(i,j, a)$ the expected instanteneous reward associated to state $i \in \mathcal{S}$.
The value function $ V_{\nu}: \mathcal{S} \rightarrow \mathbb{R}$ associated to the MDP is defined  as 
$$
V_{\nu}(s) = \mathbb{E} \bigg[ \tsum_{t=0}^{\infty} \beta^t r_t  | s_0 = s \bigg], ~~~\text{where}~ r_t = r(s_t, s_{t+1}, a_t), ~ t \in \mathbb{Z}_+.
$$
The expectation is taken over state-trajectories of the underlying Markov chain. We will think  interchangeably  of a map acting on a finite set $ \clS$ of cardinality $n$,  and a vector in $\mathbb{R}^n$ (e.g. $V_{\nu}$ and $R $).
The Bellman operator associated to the policy $\nu$ is denoted by $T_{\nu} : \mathbb{R}^n \rightarrow \mathbb{R}^n$
where 
$
T_{\nu} V(s) = R(s) + \beta \tsum_{s' \in \clS} P_{s s'} V(s'),~\forall s \in \clS.
$
The value function $ V_{\nu}$ corresponding to the policy $\nu$ satisfies the Bellman equation $ T_{\nu} V_{\nu} = V_{\nu}$, i.e.,
$
V_{\nu} = R  + \beta P V_{\nu}.
$
This is a fixed-point equation, which offers a natural gateway to our algorithmic developments on variational inequalities.

In many problems one  
 resorts to a parametric approximation of the value function and we will \textcolor{black}{focus on}  the case of linear function approximation. 
Given $d \leq n$ linearly independent  vectors $ \phi_1, \hdots, \phi_d$ in $ \mathbb{R}^n$,  we  define
$ V_{\theta} = \sum_{i=1}^d \phi_i \theta_i = \Phi   \theta,$ where $ \theta \in \mathbb{R}^d$. For a fixed state $s \in \mathcal{S}$ 
the feature vector associated to the state $s$ is denoted by $ \phi(s)^{\T} = [ \phi_1(s), \hdots, \phi_d(s)]$, while
its components are referred to as features.  

Given a positive-definite matrix $M \succ 0 $ we denote the weighted inner product as $ \langle x , y \rangle_M = 
x^{\T} M y $.  A natural weighting matrix is $ M = \diag(\pi)$, since $ \pi > 0 $ given our assumptions on the underlying Markov chain. For $\theta, \theta' \in \mathbb{R}^d$ one has
$$
\|  V_{\theta} - V_{\theta'} \|_{M} = \| \theta - \theta' \|_{\Sigma}, ~~~\text{where}~~~ \Sigma = \Phi^{\T} M \Phi,
$$ 
is the steady-state feature covariance matrix. 
We introduce the operator 
$F: \mathbb{R}^d \rightarrow \mathbb{R}^d$, where 
$$ F(\theta) = \Phi^{\T} M \big( \Phi \theta -   R - \beta  P \Phi \theta \big).$$
This is a Lipschitz continuous operator with strong monotonicity modulus $ \mu = \lambda_{\min}(\Sigma) (1 - \beta)$, where $\lambda_{\min}(\Sigma) $ is the smallest eigenvalue of the matrix $\Sigma$.  Denote  with $ \theta^*$  the vector for which $F(\theta^*)=0$. For simplicity we assume that $V_\nu$ is in the column-space of $\Phi$, i.e. $ V_{\nu} = \Phi \theta^*$. The consequences of relaxing this assumption are discussed in the remark at the end of the section.
At time instant $t \in \mathbb{Z}_+$, the corresponding stochastic operator is 
$$
\tilde{F}(\theta_t, \xi_t ) = \bigg( \langle \phi(s_t)  ,   \theta_t \rangle -   r_t   -   \beta   \langle  \phi(s_{t+1})   ,   \theta_t  \rangle  \bigg) ~ \phi(s_t), ~~~~~\text{where}~~ \xi_{t} = ( s_t, s_{t+1}, a_t).
$$
Let  $\Pi$ denote the stationary distribution
of $ \xi_t$ on $\Xi = \mathcal{S} \times \mathcal{S} \times \mathcal{A}$ and note that  $ \mathbb{E}_{\xi_t \sim \Pi} [\tilde{F}(\theta, \xi_t )] = F(\theta)$.
Applicability of our algorithmic developments hinges upon Assumptions A - D stated in Section \ref{sec_problem_statement}. 
To this end we note that \eqref{Lipschitz_pointwise} follows readily, given that $ \Xi $ is a finite set. In Assumption B, \eqref{unbiased_x*} is a direct consequence of the fact that $ V_{\nu} = \Phi \theta^*$ where $ F(\theta^*)=0$. Assumption C is more involved. For the TD and CTD algorithm, \eqref{variance_bound} is a consequence 
of the fact that $ \mathbb{E}[ \| \theta_t - \theta^* \|^2] $ is bounded and this circumstance is proven in the Appendix. As far as FTD is concerned we provide two ways
to bound the variance of the iterates, namely projection  and mini-batch.
Again the details are relegated to the Appendix.
 We  now verify \eqref{bound_delta_inner_lhs} in the ensuing lemma.  
\begin{lemma} \label{2_2_markovian}
Given the single ergodic class Markov chain $ \xi_1, \xi_2 , \hdots, $, there exists a constant $ \tilde{C} > 0 $ and $ \rho \in [0,1)$ such that for every 
$t, \tau \in \mathbb{Z}_+$ and  $\theta_t \in \mathbb{R}^d$ with probability 1, 
$$
\| F(\theta_t) - \mathbb{E}[\tilde{F}(\theta_t, \xi_{t+\tau})|\mathcal{F}_{t-1}] \|_*  \leq \tilde{C} \rho^{\tau} \sigma_{\max}(\Phi)^2 \sigma_{\max}( I - \beta P) \|\theta_t-\theta^*\|.
$$
\end{lemma}
\begin{proof}
We write 
	\begin{eqnarray*}
		\mathbb{E}[ 
		\tilde{F}(\theta_{t }, \xi_{t+\tau}) |\mathcal{F}_{t-1 } ]    & = &
		\sum_{i \in [n]} \sum_{j \in [n]} \sum_{a \in [m]}
		\Pr[s_{t+\tau} = i |s_{t}]  \nu(i, a) p(i,j,a)  \phi(	i )  \bigg(
	     (\phi(i)^\T - \beta \phi(j)^\T )   \theta_{t } -    r(i,j,a) \bigg)  \\
		& = & 
		  \Phi^{\T}   M_{\tau}(s_{t+ \tau})  \big( ( I - \beta P) \Phi \theta_{t } - R\big)
		  =  \Phi^{\T}   M_{\tau}(s_{t+ \tau}) ( I - \beta P) \Phi (\theta_t - \theta^*)
	\end{eqnarray*}
	where $ M_{\tau}(s_{t+ \tau}) = \diag( \AR{ccc}{\Pr[s_{t+\tau} = 1 |s_{t}] &\hdots& \Pr[s_{t+\tau} = n |s_{t}]})$. 
	Since
	$ F(\theta_{t }) = \Phi^{\T}   ~M~  ( I - \beta P)~ \Phi (\theta_{t } - \theta^*  ) $ we obtain 
	that 
	$$
	F(\theta_t) - \mathbb{E}[\tilde{F}(\theta_t, \xi_{t+\tau})|\mathcal{F}_{t-1}] = 
	\Phi^{\T}   ~(M- M_{\tau}(s_{t+ \tau}))~  ( I - \beta P)~ \Phi (\theta_{t } - \theta^*  )
	$$
	By \textcolor{black}{Theorem} 4.9 in \cite{levin2017markov} there exists $ \tilde{C} > 0 $ and $\rho \in [0,1) $ such that for all  $i \in [n]$ and for all  $s_{t} \in \mathcal{S}$
	$$
	|   \Pr[s_{t+\tau} = i ~|~s_{t}]  - \pi_{\infty,i} | \leq \tilde{C} \rho^{\tau}.
	$$
The desired result follows by noticing that $ \sigma_{\max}(\Phi^{\T}   ~(M- M_{\tau}(s_{t+ \tau}))~  ( I - \beta P)~ \Phi) \leq \tilde{C} \rho^{\tau} \sigma_{\max}(\Phi)^2 \sigma_{\max}( I - \beta P),$ where $\sigma_{\max}$ is the largest singular value.
\end{proof}

We now turn our attention towards summarizing our iteration complexity results in the context of policy evaluation in reinforcement learning. 
We also compare to the current state of the art, the projected Temporal Difference (PTD) algorithm that appears in \cite{russo_18}, and to this end we use the abbreviation $ \o = \lambda_{\min}(\Sigma)$ to reconcile the notation between the two papers. 

The inherent difference between our TD algorithm and PTD is the fact that the latter algorithm 
involves a projection step to a Euclidean ball of radius \textcolor{black}{$G$},   which is estimated as  \textcolor{black}{$G\leq \tfrac{2r_{\text{max}}}{\sqrt{\omega}(1-\beta)^{3/2}}$ in accordance to \cite{russo_18}, where $r_{\text{max}}$ is the maximum reward.}
 \textcolor{black}{
 In PTD the variable $ r_{\max} + 2 G$ is an upper bound to the magnitude of the stochastic operator.   We revisit \eqref{variance_bound}, adjusted to the current setting,
 $$
 \mathbb{E}[ \|\tilde{F}(x_t, \xi_{t+\tau}) -  
 \mathbb{E}[ \tilde{F}(x_t, \xi_{t+\tau}) | \mathcal{F}_{t-1}] \|_*^2 |\mathcal{F}_{t-1} ] \leq \tfrac{\sigma^2}{2} + 8\|x_t-x^*\|^2
 $$
and note that 
one of the advantages of our approach lies in the fact that we do not have to use an a-priori estimate on $\sigma^2$. Moreover, as it can be seen from \eqref{property_mini_batch}  we can improve the bound on the right hand side via mini-batches. This feature of our algorithms is important in the presence of distributed optimization capabilities as is the case in multiprocessor and multiagent parallel environments. }

Both algorithms involve a time scale parameter,  $\tau_{\text{mix}}(\epsilon)$ and $\underline{\tau}$ respectively.
In the PTD algorithm,  $\tau_{\text{mix}}(\epsilon) = \mathcal{O}( \log (1 / \epsilon))$, while the parameter  $\underline{\tau}$ that enters our analysis is a constant.

\begin{table}[H]  
	\centering
	\begin{tabular}{|c|c|}
		\hline
		Algorithm & Overall Complexity \\
		\hline 
		PTD  &  	$
		\mathcal{O}\big(   \tfrac{\tau_{\text{mix}}(\epsilon)   G^2 }{ (1-\beta)^2 \o^2  } \tfrac{1}{\epsilon} \log \tfrac{1}{\epsilon}\big)$   \\
			\hline				
		 TD  & $
		 \mathcal{O} \big(( \tfrac{\underline{\tau} }{(1-\beta)^2 \o^2}+ 
		 \sqrt{\tfrac{\underline{\tau}  }{(1-\beta)^3 \o^3 (1-\rho)}}) 
		 \tfrac{\|\theta_1-\theta^*\|}{\sqrt{\epsilon}} + \tfrac{\underline{\tau} \sigma^2 }{(1-\beta)^2 \o^2  \epsilon}  \big) 
		 $
		  \\
	\hline		
	CTD  &$ {\cal O} \big(\tfrac{\underline{\tau}  }{(1-\beta)^2 \o^2} \log\tfrac{\|\theta_1-\theta^*\|^2}{\epsilon} + \tfrac{\underline{\tau} \sigma^2 }{(1-\beta)^2 \o^2 \epsilon} \big)$ \\
\hline		
	FTD  &${\cal O}\big( \tfrac{\underline{\tau}  }{(1-\beta) \o} \log\tfrac{\|\theta_1-\theta^*\|^2}{\epsilon}+\tfrac{\underline{\tau} (\sigma^2+G^2) }{(1-\beta)^2 \o^2 \epsilon} \big)$ \\
\hline		
	\end{tabular}
	\caption{Summary of iteration complexity for obtaining a (stochastic) $\epsilon$-solution}\label{table1}
\end{table}
Since our algorithms benefit from variance reduction via mini-batch, we can employ this technique
to achieve the linear rate of convergence in terms the number of policy value updates made to $\theta_t$ in regards to CTD and FTD with the latter algorithm exhibiting then the fastest convergence. 

Furthermore,
as we can  observe in the above table the complexity iteration results deteriorate with
$ \beta \approx 1$. We provided a remedy to this situation via our robust FTD analysis 
in Theorem~\ref{lemma_dist_no_strong_monotone_FastTD}. The particular stepsize policy allows the FTD method
to find a solution $\bar{\theta} \in \mathbb{R}^d$ s.t. $\bbe[\res(\bar \theta)] \le \epsilon$ in ${\cal O} (1/\epsilon^2)$ iterations, while employing ${\cal O} (1/\epsilon^2)$ batch size, where the expected residual corresponds to the expected Bellman error in the RL terminology. 

We now address the situation when $V_{\nu} $ is not in the column-space of $\Phi$ and 
as such $ \Phi \theta^* \neq  V_\nu$. Our analysis shows that, in this case, the error 
$ \|  \Phi \theta_k - V_{\nu} \| $ will have a constant component, which  is in the same order of magnitude as $\|  \Phi \theta^* - V_{\nu} \|$.
\textcolor{black}{We will discuss this circumstance in the context of the FTD algorithm with the stepsize selection of Corollary \ref{FTD_boundness1_1}.}
The essence of our conclusions carries over to the other cases as well.  

\begin{lemma}\label{linear_approx_error}
Let  $\{\theta_t\}$ denote the iterates of the FTD algorithm. Assumption D \eqref{bound_delta_inner_lhs}  is now subject to an additional bias term
$$
\| F(\theta_t) - \mathbb{E}[\tilde{F}(\theta_t, \xi_{t+\tau})|\mathcal{F}_{t-1}] \|_*  \leq C\rho^\tau \|\theta_t-\theta^*\|+ \tfrac{C \rho^\tau}{\sigma_{\max}(\Phi)  }
\|  \Phi \theta^* - V_{\nu}  \|.
$$
Consequently \eqnok{bound_delta_inner1} and \eqnok{bound_delta_norm1} now become
\textcolor{black}{
\begin{align}
 \langle F(\theta_t) - \mathbb{E}[\tilde{F}(\theta_t, \xi_{t+\tau})|\mathcal{F}_{t-1}], \theta_t -\theta^*\rangle  \leq C\rho^\tau \|\theta_t-\theta^*\|^2+ \tfrac{C \rho^\tau}{\sigma_{\max}(\Phi)  }
 \|  \Phi \theta^* - V_{\nu}  \| \|\theta_t-\theta^*\|.\label{new_2_5}\\
\mathbb{E}[\| F(\theta_t) - \tilde{F}(\theta_t, \xi_{t+\tau}) \|^2 | \mathcal{F}_{t-1}] \leq \sigma^2 + (\varsigma^2 + 4 C^2\rho^{2\tau}) \|\theta_t-\theta^*\|^2 + 
\tfrac{4 C^2 \rho^{2 \tau}}{\sigma^2_{\max}(\Phi)  }
\|  \Phi \theta^* - V_{\nu} \|^2.\label{new_2_6}
\end{align}}
\end{lemma}
\begin{proof}
	Continuing the discussion in Lemma \ref{2_2_markovian}, we have
\begin{align*}
	&\quad F(\theta_t) - \bbe[\tilde F(\theta_t,\xi_{t+\tau})|\mathcal{F}_{t-1}] \\
	&= \big(F(\theta_t)-F(\theta^*)\big) - \bbe[\tilde F(\theta_t,\xi_{t+\tau})-\tilde F(\theta^*,\xi_{t+\tau})|\mathcal{F}_{t-1}]  + \big(F(\theta^*)-\bbe[\tilde F(\theta_t,\xi_{t+\tau}) ]\big) \\
	&= \Phi^\top (M-M_\tau(s_{t+\tau}))(I-\beta P) \Phi (\theta_t-\theta^*) + \Phi^\top (M-M_\tau(s_{t+\tau}))\big((I-\beta P) \Phi\theta^*-R\big)\\
	&= \Phi^\top (M-M_\tau(s_{t+\tau}))(I-\beta P) \Phi (\theta_t-\theta^*) + \Phi^\top (M-M_\tau(s_{t+\tau}))\big((I-\beta P) (\Phi\theta^*- V_\nu)\big),
\end{align*}
where the last equality follows from Bellman's equation. Consequently by utilizing the ergodicity condition,  
\begin{align*}
&\quad \|F(\theta_t) - \bbe[\tilde F(\theta_t,\xi_{t+\tau})|\mathcal{F}_{t-1}] \|\\
& \leq \tilde C \rho^\tau \sigma_{\max}(\Phi)^2\sigma_{\max} (I-\beta P) \|\theta_t-\theta^*\| + \tilde C \rho^\tau \sigma_{\max}(\Phi)\sigma_{\max} (I-\beta P) \|\Phi\theta^* - V_\nu\|
\end{align*}
By taking $C:= \tilde C \sigma_{\max}(\Phi)^2\sigma_{\max} (I-\beta P)$, it follows that
$$
\|F(\theta_t) - \bbe[\tilde F(\theta_t,\xi_{t+\tau})|\mathcal{F}_{t-1}] \| \leq C\rho^\tau \|\theta_t-\theta^*\| + \tfrac{C\rho^\tau}{\sigma_{\max}(\Phi)} \|\Phi \theta^*-V_\nu\|.
$$
\end{proof}

\textcolor{black}{Given the results in Lemma \ref{linear_approx_error}, we can infer the convergence results of our algorithms under inexact feature approximation. Taking the FTD algorithm as an instance, note that \eqnok{bnd_FastTD_temp} still holds, and we replace \eqnok{bound_delta_inner1} and \eqnok{bound_delta_norm1} used in \eqnok{f} by \eqref{new_2_5} and \eqref{new_2_6}. }
\textcolor{black}{Taking the same stepsize in Corollary \ref{FTD_boundness1_1} and the projection radius $G = \tfrac{2r_{\text{max}}}{\sqrt{\omega}(1-\beta)^{3/2}}$, it follows that 
$$
	\bbe[\| \theta_{k+1} - \theta^* \|^2] \leq   \tfrac{\| \theta_1 - \theta^*\|^2}{(1-\beta)^2\omega^2} \mathcal{O}(\tfrac{1}{k^2}) + \tfrac{\sigma^2+R^2}{(1-\beta)^2\omega^2} \mathcal{O}(\tfrac{1}{k})+ \tfrac{\|  \Phi \theta^* - V_{\nu} \|^2}{\sigma^2_{\max}(\Phi) } \mathcal{O}(1),
$$
where $R^2=\mathcal{O}\{\|\theta_1-\theta^*\|^2+G^2+\tfrac{\|  \Phi \theta^* - V_{\nu} \|^2}{\sigma^2_{\max}(\Phi) }\}$. Taking the triangle inequality  $ \|  \Phi \theta_k - V_{\nu} \| \leq \|  \Phi \theta_k - \Phi \theta^* \| + 
\|  \Phi \theta^* - V_{\nu} \|$ into account, we deduce that  
\begin{align*}
\bbe[\|  \Phi \theta_{k+1} - V_{\nu} \|^2] \leq
 \tfrac{\| \theta_1 - \theta^*\|^2\sigma^2_{\max}(\Phi) }{(1-\beta)^2\omega^2}  \mathcal{O}(\tfrac{1}{k^2}) + \tfrac{(\sigma^2+R^2)\sigma^2_{\max}(\Phi)}{(1-\beta)^2\omega^2}  \mathcal{O}(\tfrac{1}{k}) + \|  \Phi \theta^* - V_{\nu} \|^2 \mathcal{O}(1) .
\end{align*}}
From the above relationships we see that the error caused by the inexact feature approximation is not amplified by the algorithm.

\section{Numerical experiments} \label{sec-num}   

We demonstrate our stochastic policy evaluation algorithms 
on the basis of a classic
finite state-action space problem in reinforcement learning, referred to as the  
2D Grid-World example. 
 An agent obtains a positive reward when they reach a predetermined goal and negative ones when they go through the specific states, designated as traps, see also \cite{dann2014policy}. The agent picks with higher probability the direction (up, down, right and left) that points towards the goal, whereas ties are broken randomly.  We are interested in computing the value function $x^*$ for each possible initial state of the agent.  The discount factor is denoted by $\beta$. 
 \begin{figure}[H]
 	\centering
 	\includegraphics[width=6cm]{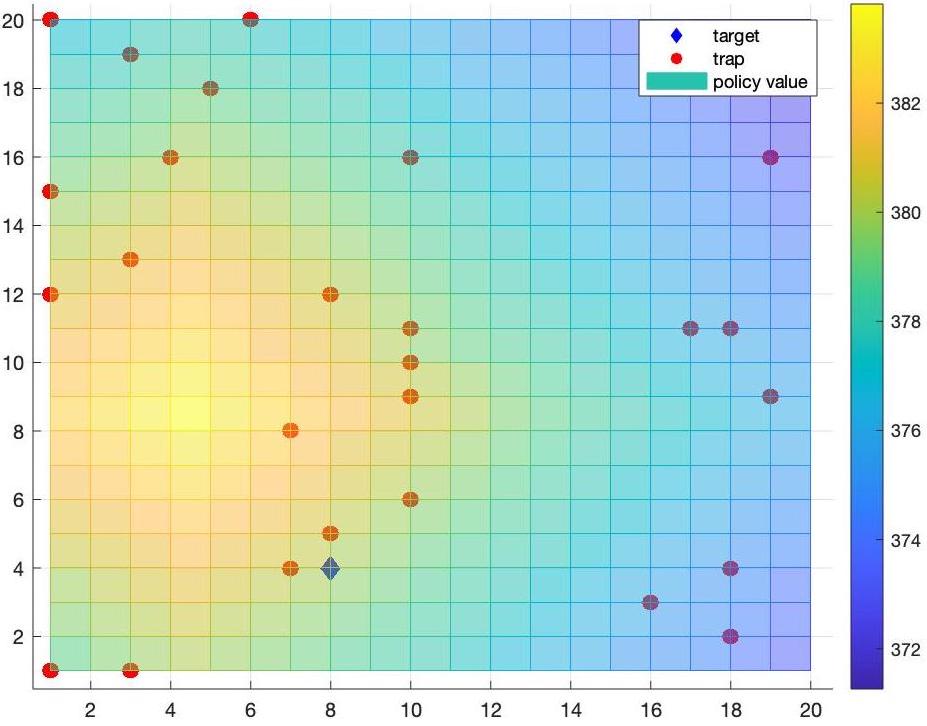}
 	\caption{2D-Grid World Example}
 	\label{fig_2D_Grid_0}
 \end{figure}
 We tested the algorithm on the synthetic data, in terms of the position of the goal and the locations of the traps. The dimension of the state space is set to $S:=|\mathcal{S}|=400$. Our square grid contains one goal-state (we assign to it a reward of $r= 1$) and 30 traps (we assign to them a reward of $r =-0.2$). With probability 0.95 the agent chooses a direction that points towards the
 goal and with probability 0.05 a random direction.
 
   The parameter $\tau$ was progressively increased, as in $\tau_k = 2^k , k \in [5]$. The results are depicted for $\tau = 8$, since inreasing beyond this value did not offer significant improvements in terms of performance.  
   The initial simulation segment that we used in order to decide an adequate 
   value for $\tau$ is depicted below.
\begin{figure}[H]
	\centering
	\begin{minipage}[t]{0.4\linewidth}
		\centering
		\includegraphics[width=5cm,height=4cm]{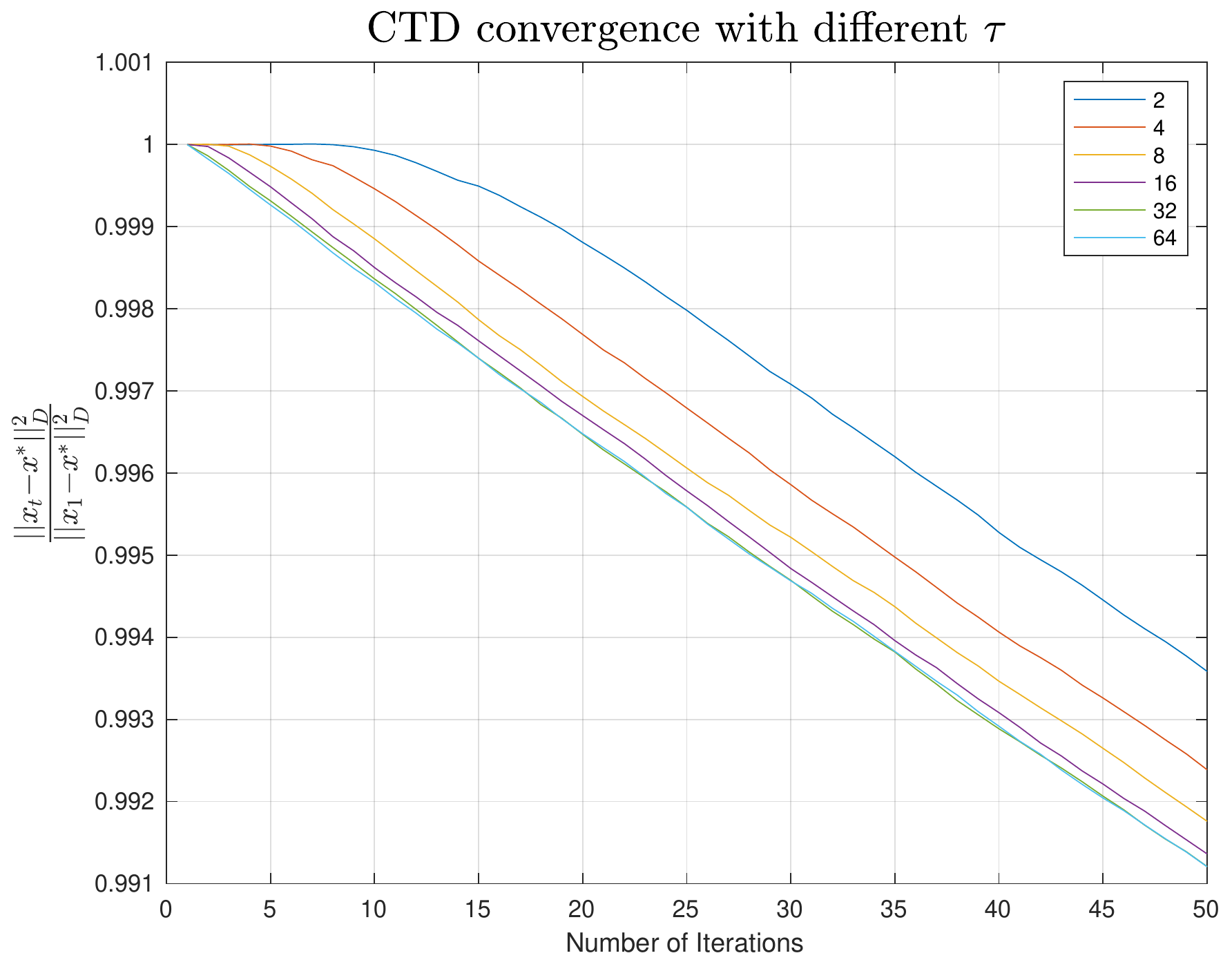}
	\end{minipage}
	\begin{minipage}[t]{0.4\linewidth}
		\centering
		\includegraphics[width=5cm,height=4cm]{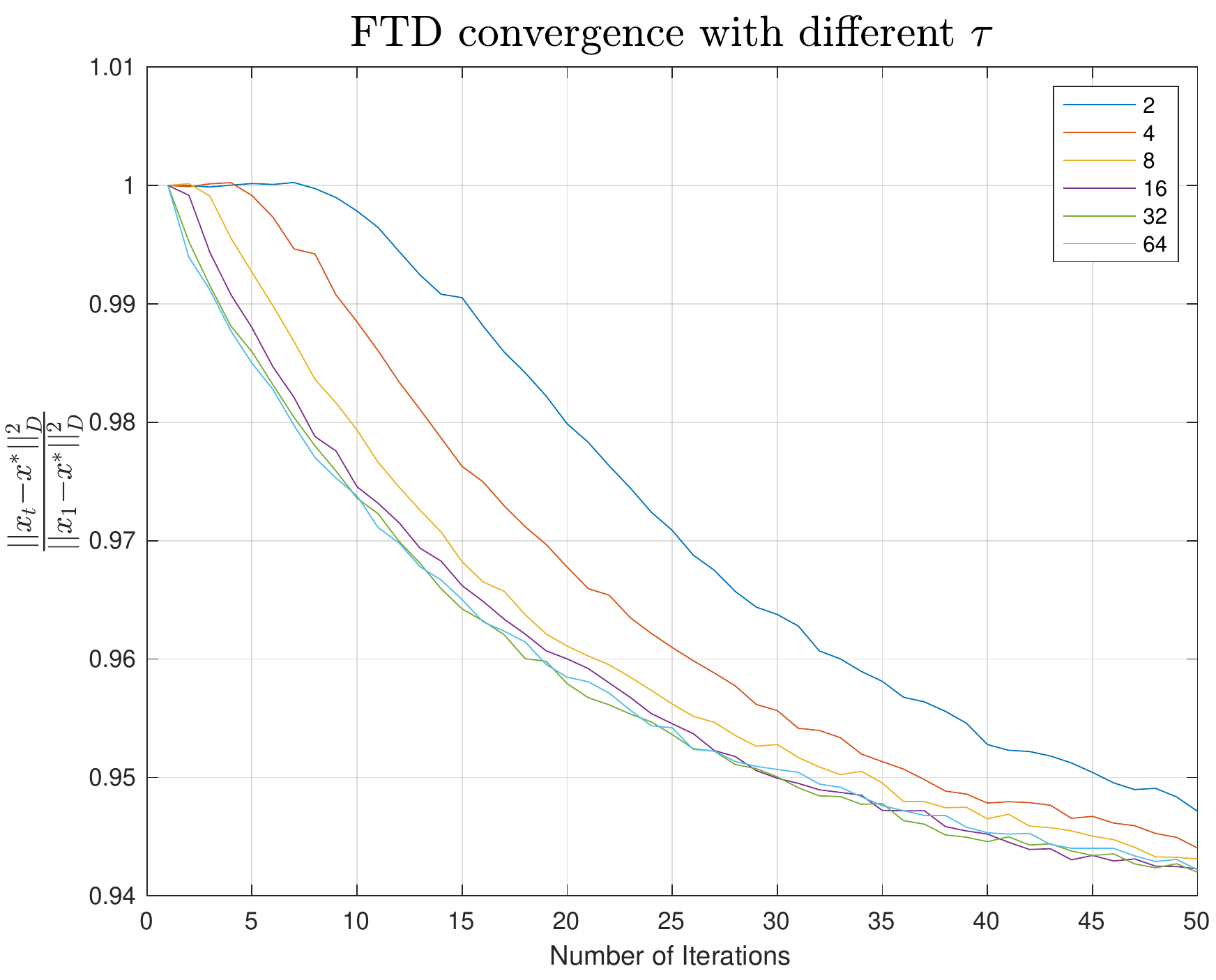}
	\end{minipage}
	\caption{ Initial simulation segment for determining $\tau$.}
\end{figure}
\textcolor{black}{In this example the selection of $\tau$ utilized a-priori knowledge of the optimal solution $x^*$. In problems where this information is not available a viable approach consists of building estimates for $\rho$ and $C$ and therefore $\tau$ via the Doeblin minorization condition, see \cite{rosenthal95} and the references therein.}
	 
We will apply the TD algorithm, three versions of the CTD algorithm and four versions of the FTD algorithm and compare to the Projected Temporal Difference (PTD) algorithm in \cite{russo_18}. The three versions of the CTD algorithm are implemented with the stepsizes selected as in Corollary \ref{step_size_TD_skipping} (CTD-1), Corollary \ref{step_size_TD_skipping_constant} (CTD-2) and Corollary \ref{restart_stepsize} (CTD-3). The four versions of the FTD algorithm are implemented with the stepsizes selected as in Corollary \ref{step_size_stoch_strong_mon_k_unknown} (FTD-1), Corollary \ref{step_size_stoch_strong_mon_k_known} (FTD-2), Corollary \ref{restart_stepsize_FastTD} (FTD-3) and Theorem \ref{lemma_dist_no_strong_monotone_FastTD} (FTD-4). 

The theoretical analysis provides us with conservative stepsize policies which will ensure the convergence for all the algorithms above. However, in terms of the actual implementation, we used the first 200 iterations to fine-tune each stepsize policy in order to achieve faster convergence. Our fine-tune principle is that we only change the value of the Lipschitz constant $L$.
For fairness purposes we maintain this convention  across all  algorithms and stepsize policies. Based on our experiments, we set $L_{\beta_1}=0.95$,  $L_{\beta_2}= 0.5$, $L_{\beta_3}=0.25$ for the cases $\beta_1=0.9$, $\beta_2=0.99$, $\beta_3= 0.999$ respectively. Our choice of $L$  improves the convergence speed while ensuring that the error of the algorithm decreases steadily. The motivation behind our approach lies in the fact that 
 the bound of \eqref{def_tildeQ_FastTD} only requires a local Lipschitz constant for each time step $t$, which can be much smaller than the global one. 

\begin{figure}[H]
	\centering
	\begin{minipage}[t]{0.4\linewidth}
		\centering		\includegraphics[width=5cm,height=3.8cm]{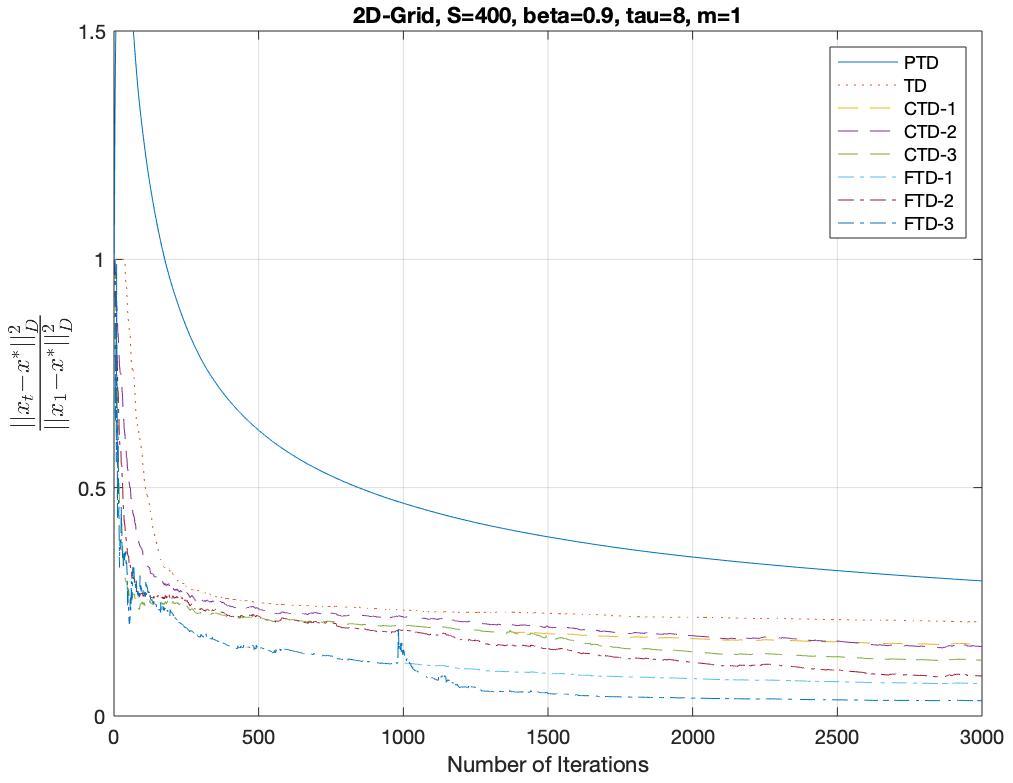}
	\end{minipage}
	\begin{minipage}[t]{0.4\linewidth}
		\centering
		\includegraphics[width=5cm,height=3.8cm]{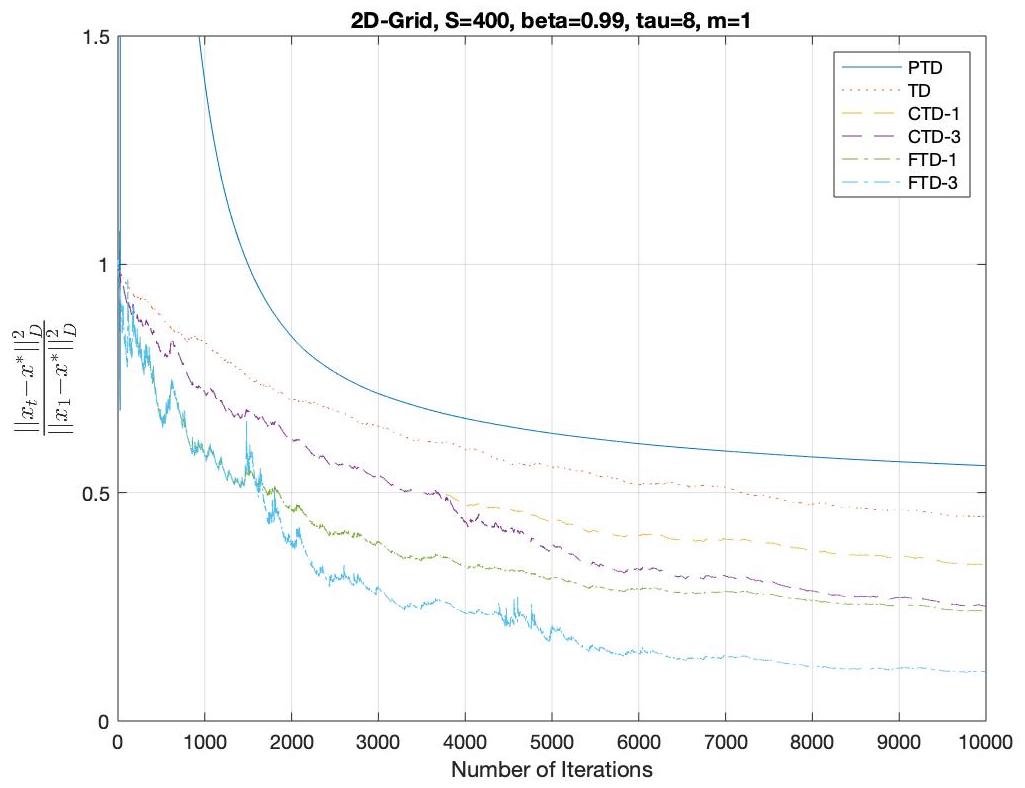}
	\end{minipage}
	\caption{Comparison of the algorithms for the 2D-Grid world example. From left to right $\beta$ is set to $0.9$, $0.99$ respectively. In the $y$-axis we report ratios in terms of the Euclidean norm $ \| \cdot \|_D$, where $D = \diag(\pi)$. }
	\label{fig_2D_Grid_1}
\end{figure}
In the two experiments above, we used a single trajectory ($m=1$) to generate the operator value. 
For $\beta =0.9$ the performance of TD algorithm is comparable to CTD and FTD. However, in the more challenging setting, when $\beta =0.99$, the advantage of CTD and FTD become pronounced.
The results indicate that the FTD algorithm exhibits faster convergence to the true value function. In particular the FTD-3 index-resetting stepsize policy obtains the fastest convergence. Similarly, we observe that the CTD-3 algorithm obtains faster convergence to the value function
than CTD-1. Additionally, in view of the expression for $q$ in Corollary \ref{step_size_TD_skipping_constant} and Corollary \ref{step_size_stoch_strong_mon_k_known} we see that as $\mu$ decreases 
one needs a higher $m$ to maintain a valid step-size, which is not the case for $\beta=0.99$ in our simulations.  
\begin{figure}[H]
	\centering
	\begin{minipage}[t]{0.3\linewidth}
		\centering
		\includegraphics[width=5cm,height=3.8cm]{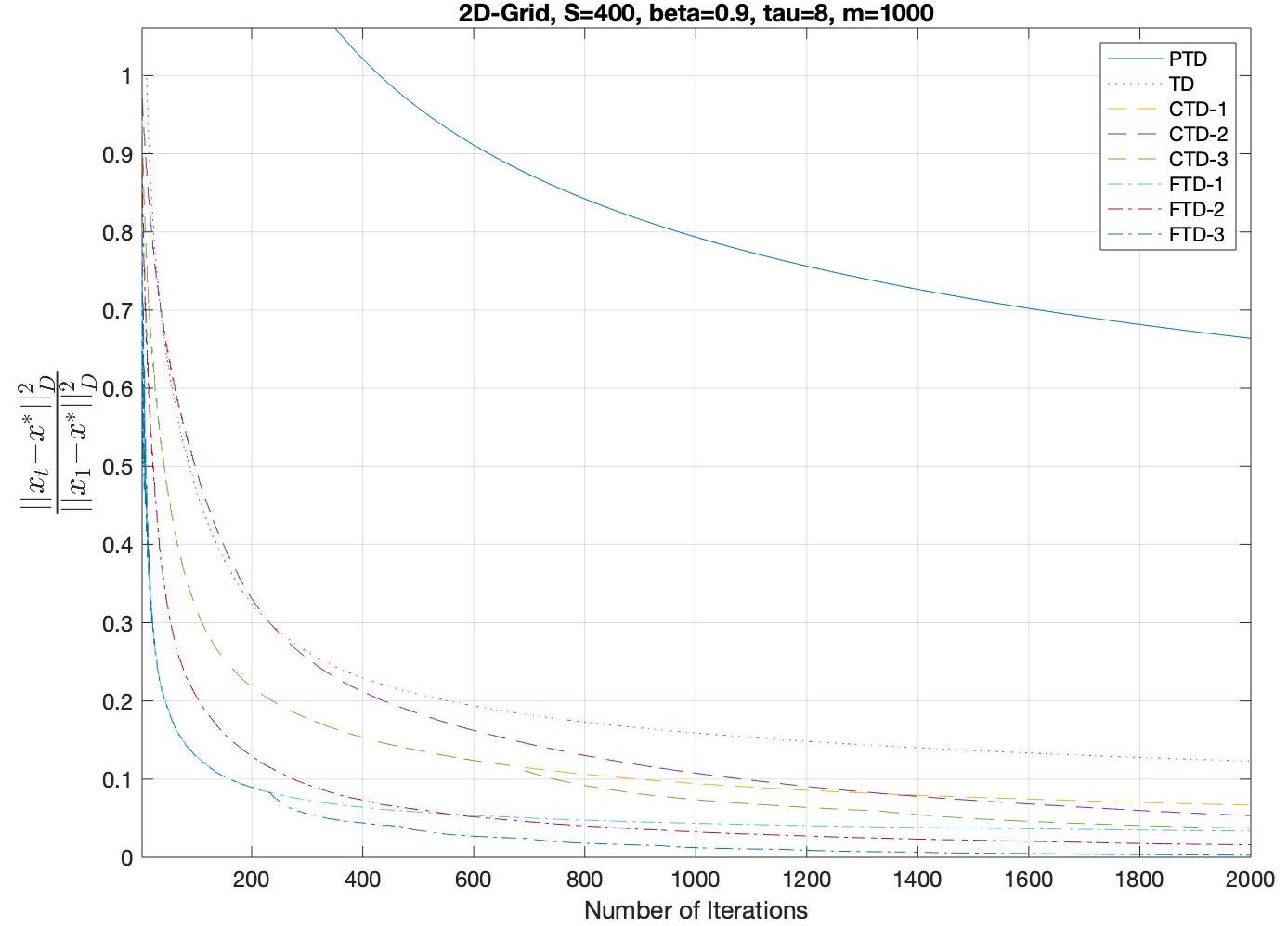}
	\end{minipage}
	\begin{minipage}[t]{0.3\linewidth}
		\centering
		\includegraphics[width=5cm,height=3.8cm]{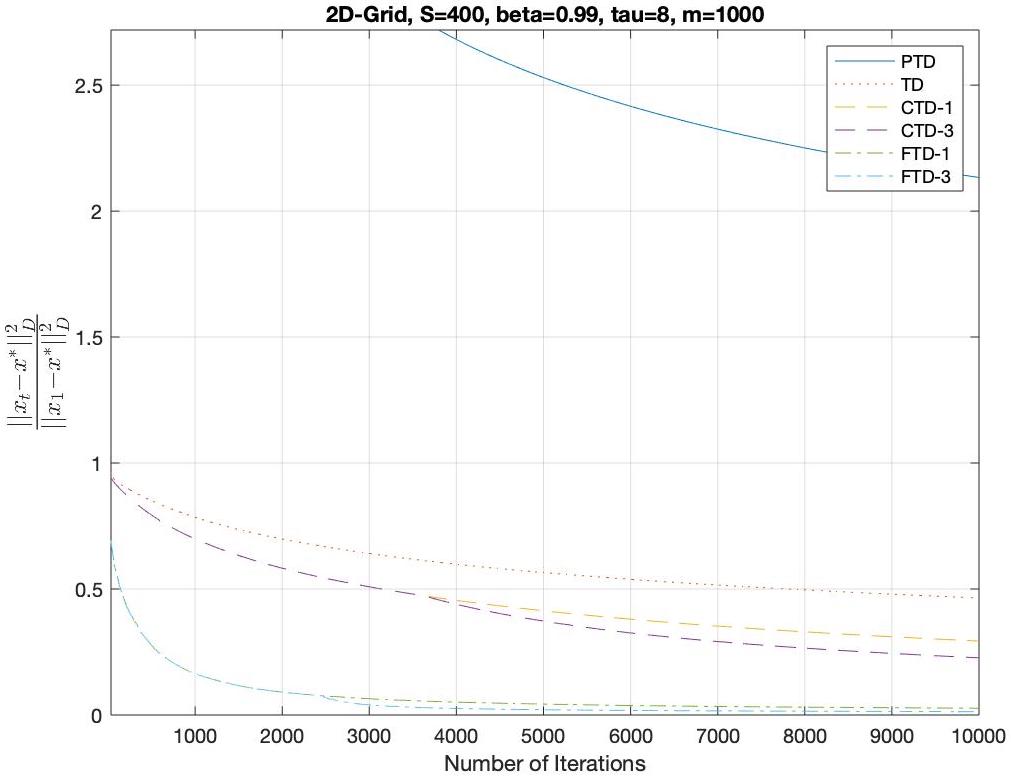}
	\end{minipage}
	\begin{minipage}[t]{0.3\linewidth}
		\centering
		\includegraphics[width=5cm,height=3.8cm]{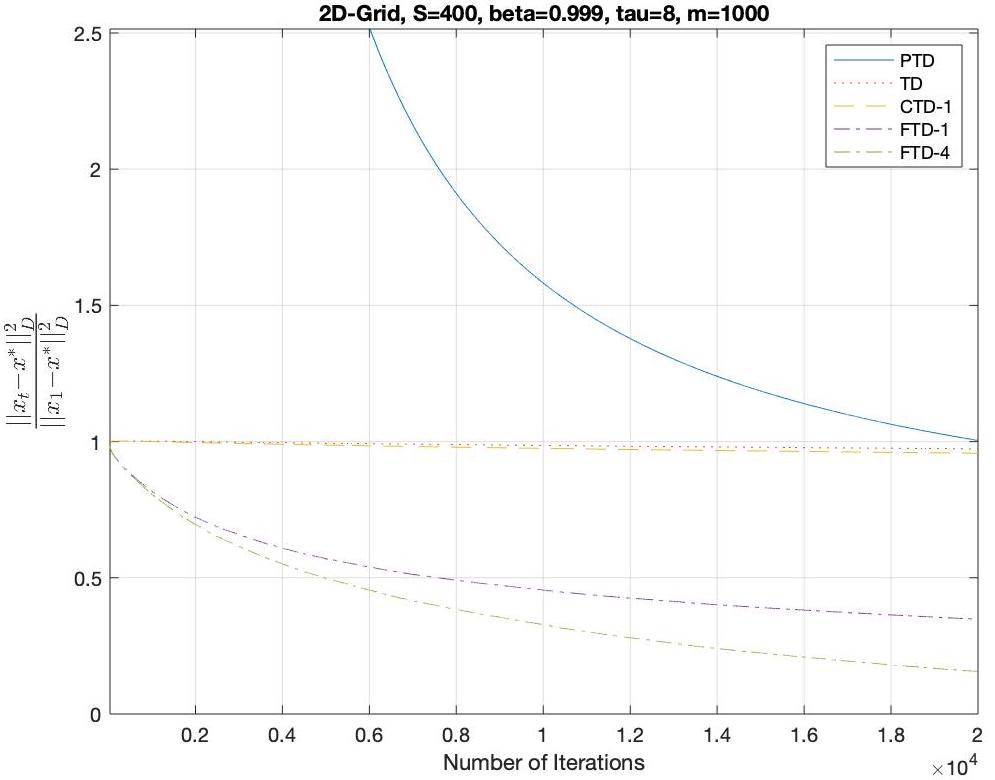}
	\end{minipage}
	\caption{Comparison of algorithms for the 2D-Grid world example. From left to right $\beta$ is set to $0.9$,  $0.99$, $0.999$ respectively. For each experiment the operator value was averaged with
		a batch size of $m=1000$. }
	\label{fig_2D_Grid_2}   
\end{figure}
In the second set of experiments (see Figure~\ref{fig_2D_Grid_2}), we increased the batch size to $m = 1000$ in order to control the variance $\sigma^2$ .
In all three experiments with averaged operator, the FTD algorithm converges faster to the true value function. The results
exhibit a similar trend as in the single trajectory experiments when $\beta=0.9$ and $\beta=0.99$. When $\beta=0.999$,
i.e., $ \mu = 0.001$,  
all the algorithm designs that evoke the generalized strong monotonicity condition converge slowly. However, in accordance to our expectations the implementation of the robust FTD-4 algorithm exhibited the fastest convergence.

\section{Concluding remarks} \label{sec_conclusion}
The  paper investigated stochastic variational inequalities (VI) under Markovian noise with a view towards 
 stochastic policy evaluation problem in reinforcement learning. 
 We developed a variety of simple TD learning  type algorithms motivated by its original version that maintain its simplicity, while offering distinct advantages in terms of non-asymptotic analysis.
 We analyzed the standard TD algorithm and developed two new stochastic algorithms
 referred to as CTD and FTD. The CTD algorithm involves periodic updates of the stochastic iterates, which reduces the bias and therefore exhibits improved iteration complexity. The FTD algorithm  combines elements of CTD and the stochastic operator extrapolation method	of the companion paper.  
For a novel index resetting policy FTD exhibits optimal convergence rate. We also devised a robust version of the algorithm  
that is particularly suitable for discounting factors close to 1. Numerical experiments conducted on a benchmark policy evaluation problem demonstrate the advantages of our proposed algorithms in comparison to prior literature.

\renewcommand\refname{Reference}

\bibliographystyle{plain}
\bibliography{revised_version_ARXIV.bib}

\end{document}